\pdfoutput=1
\documentclass[11pt]{article} 
\def\arXiv{1}  
\usepackage[inline, shortlabels]{enumitem}
\usepackage{amssymb}
\usepackage{amsbsy}
\usepackage{amsmath}
\usepackage{amsthm}
\usepackage{amsfonts}
\usepackage{latexsym}
\usepackage{graphicx}
\usepackage{xifthen}
\usepackage{mathtools}
\usepackage[table, dvipsnames]{xcolor}
\usepackage{enumitem}
\usepackage{thmtools}
\usepackage{thm-restate}
\usepackage{graphicx}
\usepackage{color}
\usepackage{bbm}
\usepackage{xifthen}
\usepackage{xspace}
\usepackage{mathtools}
\usepackage{titlesec}
\usepackage{setspace}
\usepackage{etoolbox}
\usepackage{xfrac}
\usepackage{esvect}
\usepackage{nicefrac}
\usepackage{cases}
\usepackage{empheq}
\usepackage{booktabs}
\usepackage{multirow}
\usepackage{xparse}
\usepackage{subfig}
\usepackage{caption}
\usepackage{cprotect}
\usepackage{bbm}
\usepackage{placeins}
\usepackage{url}    
\usepackage{nicefrac}  
\usepackage{microtype}   
\usepackage[hidelinks]{hyperref}
\hypersetup{
	colorlinks=true,
	linkcolor=blue!70!black,
	citecolor=blue!70!black,
	urlcolor=blue!70!black
}  

\newcommand{\notarxiv}[1]{foo}
\newcommand{\arxiv}[1]{ba}
\ifdefined \arXiv
\renewcommand{\arxiv}[1]{#1}%
\renewcommand{\notarxiv}[1]{\ignorespaces}%
\else%
\renewcommand{\arxiv}[1]{\ignorespaces}%
\renewcommand{\notarxiv}[1]{#1}%
\fi%

\arxiv{
	\usepackage[top=1in, right=1in, left=1in, bottom=1in]{geometry}
	\usepackage[numbers, square]{natbib}
}

\notarxiv{
	\usepackage[utf8]{inputenc} 
	\usepackage[T1]{fontenc}    
}

\usepackage[linesnumbered,ruled,vlined]{algorithm2e}
\usepackage{cleveref}

\theoremstyle{plain}
\newtheorem{theorem}{Theorem}
\newtheorem{lemma}{Lemma}
\newtheorem{corollary}[theorem]{Corollary}

\newtheorem*{claim*}{Claim}

\theoremstyle{definition}
\newtheorem{definition}{Definition}

\newtheoremstyle{remark}{0.5\topsep}{0.5\topsep}{}{}{\bf}{.}{5pt plus 1pt minus 1pt}{}
\theoremstyle{remark}
\newtheorem{remark}{Remark}

\DeclarePairedDelimiter{\abs}{\lvert}{\rvert} %
\DeclarePairedDelimiter{\brk}{[}{]}
\DeclarePairedDelimiter{\crl}{\{}{\}}
\DeclarePairedDelimiter{\prn}{(}{)}

\DeclarePairedDelimiter{\norm}{\|}{\|}

\DeclarePairedDelimiter{\ceil}{\lceil}{\rceil}

\DeclarePairedDelimiterXPP{\onenorm}[1]{}{\|}{\|}{_{1}}{#1}
\DeclarePairedDelimiterXPP{\twonorm}[1]{}{\|}{\|}{_{2}}{#1}
\DeclarePairedDelimiterXPP{\infnorm}[1]{}{\|}{\|}{_{\infty}}{#1}
\DeclarePairedDelimiterXPP{\pnorm}[1]{}{\|}{\|}{_{p}}{#1}
\DeclarePairedDelimiterXPP{\qnorm}[1]{}{\|}{\|}{_{q}}{#1}
\DeclarePairedDelimiterXPP{\opnorm}[1]{}{\|}{\|}{_{\mathrm{op}}}{#1}
\DeclarePairedDelimiterXPP{\dualnorm}[1]{}{\|}{\|}{_{*}}{#1}
\DeclarePairedDelimiterXPP{\gennorm}[2]{}{\|}{\|}{_{#1}}{#2}

\DeclarePairedDelimiterXPP{\inner}[2]{}{\langle}{\rangle}{}{#1,#2}

\renewcommand{\P}{\mathbb{P}} %
\undef\Pr
\DeclarePairedDelimiterXPP{\Pr}[1]{\P}{(}{)}{}{\activatebar#1}

\newcommand{\E}{\mathbb{E}} %
\DeclarePairedDelimiterXPP{\Ex}[1]{\E}{[}{]}{}{\activatebar#1}

\newcommand{\activatebar}{%
	\begingroup\lccode`~=`|
	\lowercase{\endgroup\def~}{\;\delimsize\vert\;}%
	\mathcode`|=\string"8000
}

\undef\O
\DeclarePairedDelimiterXPP{\O}[1]{O}{(}{)}{}{#1}
\DeclarePairedDelimiterXPP{\Otil}[1]{\widetilde{O}}{(}{)}{}{#1}
\DeclarePairedDelimiterXPP{\OMEGA}[1]{\Omega}{(}{)}{}{#1}
\DeclarePairedDelimiterXPP{\OMEGAtil}[1]{\widetilde{\Omega}}{(}{)}{}{#1}

\newcommand{\overeq}[1]{\overset{#1}{=}}

\newcommand{\overge}[1]{\overset{#1}{\ge}}

\newcommand{\mc}[1]{\mathcal{#1}}

\newcommand{\R}{\mathbb{R}}

\newcommand{\N}{\mathbb{N}}

\newcommand{\xset}{\mathcal{X}}
\newcommand{\yset}{\mathcal{Y}}

\newcommand{\simiid}{\stackrel{\rm iid}{\sim}}
\newcommand{\uniform}{\mathsf{Unif}}  %
\providecommand{\argmax}{\mathop{\rm argmax}} %
\providecommand{\argmin}{\mathop{\rm argmin}}
\providecommand{\dom}{\mathop{\rm dom}}

\providecommand{\interior}{\mathop{\rm int}}
\providecommand{\bd}{\mathop{\rm bd}} %

\newcommand{\ones}{\mathbf{1}}

\providecommand{\opt}{^\star}

\providecommand{\minimize}{\mathop{\rm minimize}}
\providecommand{\maximize}{\mathop{\rm maximize}}

\newcommand{\half}{\frac{1}{2}}

\newcommand{\defeq}{\coloneqq}

\newcommand{\eps}{\epsilon}

\newcommand{\grad}{\nabla}

\SetKwInput{Input}{Input}                %
\SetKwInput{Output}{Output}              %
\SetKwInput{Parameters}{Parameters}
\SetKwProg{Function}{function}{}{}
\SetKwComment{Comment}{$\triangleright$\ }{}
\SetKwRepeat{Do}{do}{while}
\SetKwBlock{Repeat}{Repeat}{}
\SetCommentSty{mycommfont}
\makeatletter
\newcommand\Block[2]{%
	#1%
	\algocf@group{#2}%
}
\makeatother

\usepackage[textsize=scriptsize,textwidth=2cm]{todonotes}

\usepackage{csquotes}

\makeatletter
\newcommand{\oset}[3][0ex]{%
  \mathrel{\mathop{#3}\limits^{
    \vbox to#1{\kern-2\ex@
    \hbox{$\scriptstyle#2$}\vss}}}}
\makeatother

\usepackage[thinc]{esdiff}

\newcommand{\epsargmin}[2]{\underset{#2}{{\argmin}^{#1}} \,}

\newcommand{\epsargmaxtxtstyl}[2]{{\argmax}_{#2}^{#1} \,}

\newcommand{\xopt}{x\opt}
\newcommand{\yopt}{y\opt}
\newcommand{\breg}[3]{V^{#1}_{#2} \left( #3 \right)}
\newcommand{\bregtilde}[3]{\widetilde{V}^{#1}_{#2} \left( #3 \right)}
\newcommand{\rconj}{r^*}
\newcommand{\inparen}[1]{\left(#1\right)}
\newcommand{\inangle}[1]{\left\langle#1\right\rangle}

\newcommand{\inbraces}[1]{\left\{#1\right\}}
\newcommand{\inbracess}[1]{\{#1\}}
\newcommand{\insquare}[1]{\left[#1\right]}
\newcommand{\epsprim}{\epsilon_{\mathrm{p}}} %
\newcommand{\epsp}{\epsilon} %

\newcommand{\DRPO}{\textsc{DRPO}}
\newcommand{\DRPOSP}{\textsc{DRPOSP}}
\newcommand{\DRBR}{\textsc{DRBR}}
\newcommand{\CGM}{\textsc{CGM}}

\newcommand{\zopt}{z\opt}
\renewcommand{\hbar}{\bar{h}}
\newcommand{\gap}{\Delta}
\newcommand{\simplex}{\Delta}
\newcommand{\Otilde}{\widetilde{O}}
\newcommand{\Omegatilde}{\widetilde{\Omega}}
\newcommand{\Thetatilde}{\widetilde{\Theta}}

\newcommand{\gconj}{g^*}
\newcommand{\uset}{\mathcal{U}}
\newcommand{\pset}{\mathcal{P}}
\newcommand{\rconjuset}{r^*_{\uset}}

\newcommand{\indc}{\mathbb{I}}

\newcommand{\qopt}{q\opt}

\newcommand{\sset}{\mathcal{S}}
\newcommand{\hopt}{\xopt_h}
\newcommand{\fopt}{\xopt_f}
\newcommand{\gopt}{\xopt_g}
\newcommand{\fconj}{f^*}
\newcommand{\overimp}[1]{\overset{#1}{\implies}}
\newcommand{\gtilde}{\tilde{g}}
\newcommand{\mufconj}{\mu_{\fconj}}
\newcommand{\T}{\mathcal{T}}
\providecommand{\ri}{\mathop{\rm ri}}

\newcommand{\barf}{\bar{f}} %
\newcommand{\fbar}{\barf} %

\newcommand{\event}{\mathcal{E}}
\newcommand{\poly}{\mathrm{poly}}
\newcommand{\ftilde}{\tilde{f}}
\newcommand{\ytrunc}{\yset_{\mathrm{trunc}}}
\renewcommand{\L}{\mathcal{L}}
\newcommand{\barL}{\overline{\L}} %
\newcommand{\Lbar}{\overline{\L}} %

\newcommand{\tildeM}{\widehat{M}} %

\newcommand{\allpoly}{\poly(\cdots)}

\makeatletter
\newcommand{\leqnomode}{\tagsleft@true}
\newcommand{\reqnomode}{\tagsleft@false}
\makeatother

\title{Extracting Dual Solutions via Primal Optimizers} 

\arxiv{
\renewcommand{\And}{~~}

}

\author{%
	Yair Carmon\thanks{Tel Aviv University, \texttt{ycarmon@tauex.tau.ac.il}} 
	\And
	Arun Jambulapati\thanks{University of Michigan, \texttt{jmblpati@gmail.com}} 
	\And
	Liam O'Carroll\thanks{Stanford University, \texttt{\string{ocarroll,sidford\string}@stanford.edu}} 
	\And
	Aaron Sidford\footnotemark[3] 
}

\arxiv{\date{}}  %

\begin{document}

\maketitle

\begin{abstract}
	We provide a general method to convert a ``primal'' black-box algorithm for solving regularized convex-concave minimax optimization problems into an algorithm for solving the associated dual maximin optimization problem. Our method adds recursive regularization over a logarithmic number of rounds where each round consists of an approximate regularized primal optimization followed by the computation of a dual best response. We apply this result to obtain new state-of-the-art runtimes for solving matrix games in specific parameter regimes, obtain improved query complexity for solving the dual of the CVaR distributionally robust optimization (DRO) problem, and recover the optimal query complexity for finding a stationary point of a convex function.
\end{abstract}

\newpage

\setcounter{tocdepth}{2} %
\tableofcontents

\newpage

\section{Introduction}\label{sec:introduction}

We consider the foundational problem of efficiently solving convex-concave games. For nonempty, closed, convex constraint sets $\xset \subseteq \R^d$ and $\yset \subseteq \R^n$ and differentiable convex-concave objective function $\psi : \R^d \times \R^n \rightarrow \R$ (namely, $\psi(\cdot,y)$ is convex for any fixed $y$ and $\psi(x,\cdot)$ is concave for any fixed $x$), we consider the following \emph{primal}, minimax optimization problem \eqref{eq:primal} and its associated  \emph{dual}, maximin optimization problem \eqref{eq:dual}:
\leqnomode
\begin{align}
    \tag{P} \label{eq:primal} &\minimize_{x \in \xset} f(x) \text{ for } f(x) \defeq \max_{y \in \yset} \psi(x,y), \text{ and} \\
    \tag{D} \label{eq:dual} &\maximize_{y \in \yset} \phi(y)\text{ for }\phi(y) \defeq \min_{x \in \xset} \psi(x,y) .
\end{align}
\reqnomode
If additionally $\xset$ and $\yset$ are bounded (which we assume for simplicity in the introduction but generalize later), every pair of primal and dual optimizers $\xopt \in \argmin_{x \in \xset} f(x)$ and $\yopt \in \argmax_{y \in \yset} \phi(y)$ satisfies the \emph{minimax principle}: $f(\xopt) = \phi(\yopt) = \psi(\xopt,\yopt)$.

Convex-concave games are pervasive in algorithm design, machine learning, data analysis, and optimization. For example, the games induced by bilinear objectives, i.e., $\psi(x,y) = x^\top A y + b^\top x + c^\top y$, where $\xset$ and $\yset$ are either the simplex, $\simplex^{k} \defeq \{x \in \R^k_{\geq 0} : \norm{x}_1 =1\}$, or the Euclidean ball, $B^{k} \defeq \{x \in \R^k : \norm{x}_2 \leq 1\}$, encompass zero-sum games, linear programming, hard-margin support vector machines (SVMs), and minimum enclosing/maximum inscribed ball \cite{dantzig1953linear, adler2013equivalence, minsky1988perceptrons,clarkson2012sublinear}. Additionally, the case when $\psi(x,y) = \sum_{i = 1}^n y_i f_i(x)$ for some functions $f_i : \R^d \to \R$ and $\yset$ is a subset of the simplex encompasses a variety of distributionally robust optimization (DRO) problems \cite{levy2020largescale, carmon2022distributionally} and (for $\yset=\simplex^{n}$) the problem of minimizing the maximum loss \cite{carmon2021thinking, carmon2023whole,asi2021stochastic}.

In this paper, we study the following question:

\vbox{ %
	\begin{displayquote}
		\textit{Given only a black-box oracle which solves (regularized versions of) \eqref{eq:primal} to $\epsilon$ accuracy, and a black-box oracle for computing an exact dual best response $y_x \defeq \argmax_{y \in \yset} \psi(x, y)$ to any primal point $x \in \xset$, can we extract an $\epsilon$-optimal solution to \eqref{eq:dual}?}
	\end{displayquote}
}

\noindent We develop a general \emph{dual-extraction framework} which answers this question in the affirmative. We show that as long as these oracles can be implemented as cheaply as obtaining an $\epsilon$-optimal point of \eqref{eq:primal}, then our framework can obtain an $\epsilon$-optimal point of \eqref{eq:dual} at the same cost as that of obtaining an $\epsilon$-optimal point of \eqref{eq:primal}, up to logarithmic factors. We then instantiate our framework to obtain new state-of-the-art results in the settings of bilinear matrix games and DRO. Finally, as evidence of its broader applicability, we show that our framework can be used to recover the optimal complexity for computing a stationary point of a smooth convex function.

In the remainder of the introduction we describe our results in greater detail (\Cref{subsec:our-results}), give an overview of the dual extraction framework and its analysis (\Cref{subsec:dual-extraction-overview}), discuss related work (\Cref{subsec:related-work}), and provide a roadmap for the remainder of the paper (\Cref{subsec:organization}).

\subsection{Our results}\label{subsec:our-results}

\paragraph{From primal algorithms to dual optimization.} We give a general framework which obtains an $\epsilon$-optimal solution to \eqref{eq:dual} via a sequence of calls to two black-box oracles: (i) an oracle for obtaining an $\epsilon$-optimal point of a regularized version of \eqref{eq:primal}, and (ii) an oracle for obtaining a dual best response $y_x \defeq \argmax_{y \in \yset} \psi(x, y)$ for a given $x \in \xset$. In particular, we show it is always possible to obtain an $\epsilon$-optimal point to \eqref{eq:dual} with at most a logarithmic number of calls to each of these oracles, where the regularized primal optimization oracle is always called to an accuracy of $\epsilon$ over a logarithmic factor. We also provide an alternate scheme (or more specifically choice of parameters) for applications where the cost of obtaining an $\epsilon$-optimal point of the regularized primal problem decreases sufficiently as the regularization level increases. In such cases, e.g., in our stationary point application, it is possible to avoid even logarithmic factor increases in computational complexity for approximately solving \eqref{eq:dual} relative to the complexity of approximately solving \eqref{eq:primal}.

\paragraph{Application 1: Bilinear matrix games.} In this application, $\psi(x, y) \defeq x^\top A y$ for a matrix $A \in \R^{d \times n}$, $\yset$ is the simplex $\simplex^n$, and $\xset$ is either the simplex $\simplex^d$ or the unit Euclidean ball $B^d$. Recently, \cite{carmon2023whole} gave a new state-of-the-art runtime in certain parameter regimes of $\Otilde(nd + n (d / \epsilon)^{2/3} + d \epsilon^{-2})$ for obtaining an expected $\epsilon$-optimal point for the primal problem \eqref{eq:primal} for this setup. However, unlike previous algorithms for bilinear matrix games (see Section \ref{subsec:related-work} for details), their algorithm does not return an $\epsilon$-optimal solution for the dual \eqref{eq:dual}, and their runtime is not symmetric in the dimensions $n$ and $d$. As a result, it was unclear whether the same runtime is achievable for obtaining an $\epsilon$-optimal solution of the dual \eqref{eq:dual}. We resolve this question by applying our general framework to achieve an expected $\epsilon$-optimal point of \eqref{eq:dual} with runtime $\Otilde(nd + n (d / \epsilon)^{2/3} + d \epsilon^{-2})$. We then observe (see Corollary~\ref{cor:primal-mat-guarantee} and Table \ref{table:runtimes-bilinear}) that in the setting where $\xset = \simplex^d$, our result can equivalently be viewed as a new state-of-the-art runtime of $\Otilde(nd + d (n / \epsilon)^{2/3} + n \epsilon^{-2})$ for obtaining an $\epsilon$-optimal point of the primal \eqref{eq:primal} due to the symmetry of $\psi$ and the constraint sets.

\paragraph{Application 2: CVaR at level $\alpha$ DRO.} In this application, $\psi(x, y) \defeq \sum_{i = 1}^n y_i f_i(x)$ for convex, bounded, and Lipschitz loss functions $f_i : \R^d \to \R$, $\xset$ is a convex, compact set, and $\yset \defeq \inbraces{y \in \simplex^n : \norm{y}_\infty \le \frac{1}{\alpha n}}$ is the CVaR at level $\alpha$ uncertainty set for $\alpha \in [1 / n, 1]$. The primal \eqref{eq:primal} is a canonical and well-studied DRO problem, and corresponds to the average of the top $\alpha$ fraction of the losses. We consider this problem given access to a first-order oracle that, when queried at $x \in \R^d$ and $i \in [n]$, outputs $(f_i(x), \grad f_i(x))$. Ignoring dependencies other than $\alpha$, the target accuracy $\epsilon > 0$, and the number of losses $n$ for brevity, \cite{levy2020largescale} gave a matching upper and lower bound (up to logarithmic factors) of $\Otilde(\alpha^{-1} \epsilon^{-2})$ queries to obtain an expected $\epsilon$-optimal point of the primal \eqref{eq:primal}. However, the best known query complexity for obtaining an expected $\epsilon$-optimal point of the dual \eqref{eq:dual} was $\Otilde(n \epsilon^{-2})$ prior to this paper (see Section \ref{subsec:related-work} for details). Applying our general framework to this setting, we obtain an algorithm with a new state-of-the-art query complexity of $\Otilde(\alpha^{-1} \epsilon^{-2} + n)$ for obtaining an expected $\epsilon$-optimal point of the dual \eqref{eq:dual}. In particular, note that this complexity is nearly linear in $n$ when $\epsilon \ge (\alpha n)^{-2}$.

\paragraph{Application 3: Obtaining stationary points of convex functions.}
In this application, we show that our framework yields an alternative optimal approach for computing an approximate critical point of a smooth convex function given a gradient oracle. Specifically, for $\gamma > 0$ and convex and $\beta$-smooth $h : \R^n \to \R$, in Section \ref{sec:convex-critical-point}, we give an algorithm which computes $x \in \R^n$ such that $\norm{\grad h(x)}_2 \le \gamma$ using $O \inparen{ \sqrt{\beta \gap} / \gamma }$ gradient queries, where $\gap \defeq h(x_0) - \inf_{x \in \R^n} h(x)$ is the initial suboptimalityf. While this optimal complexity has been achieved before \cite{kim2020optimizing,nesterov2019primaldual,diakonikolas2021potential,lee2021geometric,lan2023optimal}, that we achieve it is a consequence of our general framework illustrates its broad applicability.

For this application, we instantiate our framework with $\psi(x, y) \defeq \inangle{x, y} - h^*(y)$, where $h^* : \R^n \to \R$ denotes the convex conjugate of $h$. (For reasons discussed in Section \ref{sec:convex-critical-point}, we actually first substitute $h$ for an appropriately regularized version of $h$, call it $f$, before applying the framework, but the following discussion still holds with respect to $f$.) This objective function $\psi$ is known as the \textit{Fenchel game} and has been used in the past to recover classic convex optimization algorithms (e.g., the Frank-Wolfe algorithm and Nesterov's accelerated methods) via a minimax framework \cite{abernethy2017equilibrium,wang2023noregretdynamicsfenchelgame,cohen2021relative,jin2022sharper}. In the Fenchel game, a dual best response corresponds to a gradient evaluation:
\begin{align*}
    \argmax_{y \in \R^n} \inbraces{\inangle{x, y} - h^*(y)} = \grad h(x),
\end{align*}
and we show that approximately optimal points for the dual objective \eqref{eq:dual} must have small norm. As a result, obtaining an approximately optimal dual point $y$ as a best response to a primal point $x$ yields a bound on the norm of $y = \grad f(x)$. Furthermore, we note that in this setting, adding regularization to $\psi$ with respect to an appropriate choice of distance-generating function (namely $h^*$) is equivalent to rescaling and recentering the primal function $f$, as well as the point at which a gradient is taken in the dual best response computation (cf. Lemma~\ref{lem:Fenchel-game-with-reg}). Thus, the properties of the Fenchel game extend naturally to appropriately regularized versions of $\psi$.

\subsection{Overview of the framework and analysis}
\label{subsec:dual-extraction-overview}

We now give an overview of the dual-extraction framework. Our framework applies generally to a set of assumptions given in Section \ref{subsec:main-framework-assumptions-preliminaries} (cf. Definition~\ref{def:dual-extraction-setup}), but for now we specialize to the assumptions given above, namely: (i) the constraint sets $\xset$ and $\yset$ are nonempty, compact, and convex; and (ii) $\psi$ is differentiable and convex-concave.
Throughout this section, let $\norm{\cdot}$ denote any norm on $\R^n$ and assume that the dual function, $\phi$, is $L$-Lipschitz with respect to $\norm{\cdot}$.\footnote{This is a weak assumption since we ensure at most a logarithmic dependence on $L$; see Remark \ref{rem:example-schedules-wlog}.} Let $r : \R^n \to \R$ denote a differentiable distance-generating function (dgf) which is $\mu_r$-strongly convex with respect to $\norm{\cdot}$ for $\mu_r > 0$,\footnote{\Cref{sec:notation-and-assumptions} gives the general setup for a distance-generating function which also covers the case where $\dom r \ne \R^n$.} and let $\breg{}{u}{v} \defeq r(v) - r(u) - \inangle{\grad r(u), v - u}$ denote the associated Bregman divergence. For the sake of illustration, it may be helpful to consider the choices $\norm{\cdot} \defeq \norm{\cdot}_2$, $r(u) \defeq \frac{1}{2} \norm{u}_2^2$, $\mu_r = 1$, and $\breg{}{u}{v} = \frac{1}{2} \norm{u - v}_2^2$ in the following, in which case relative strong convexity with respect to $r$ is equivalent to the standard notion of strong convexity with respect to $\norm{\cdot}_2$.

How should we obtain an $\epsilon$-optimal point for \eqref{eq:dual} using the two oracles discussed previously, namely: (i) an oracle for approximately solving a regularized primal objective, and (ii) an oracle for computing a dual best response? 
We call (i) a dual-regularized primal optimization (DRPO) oracle and (ii) a dual-regularized best response (DRBR) oracle; their formal definitions are given in Section \ref{subsec:main-framework-assumptions-preliminaries}. Note that to solve \eqref{eq:dual}, one cannot simply solve the primal problem \eqref{eq:primal} to high accuracy and then compute a dual best response. Consider $\psi(x, y) = xy$ with $\xset = \yset = [-1, 1]$; clearly $\xopt = \yopt = 0$, but for any $x$ arbitrarily close to $\xopt$, the dual best response is either $-1$ or $1$.

The key observation underlying our framework is that if $\psi(x, \cdot)$ is strongly concave for a given $x \in \xset$, it is possible to upper bound the distance between the best response $y_x \defeq \argmax_{y \in \yset} \psi(x, y)$ and the dual optimum $\yopt$ in terms of the primal suboptimality of $x$. Figure \ref{fig:key-lemma-intuition} illustrates why this should be the case when subtracting a quadratic regularizer in $y$ (so that $\psi(x, \cdot)$ is strongly concave) to the preceding example of $\psi(x, y) = xy$. We generalize this intuition in the following lemma (replacing strong concavity with relative strong concavity and a distance bound with a divergence bound), which is itself generalized further and proven in Section \ref{sec:dual-extraction-framework}:

\begin{figure}[h] %
	\centering
	\subfloat[]{\includegraphics[width=0.45\textwidth]{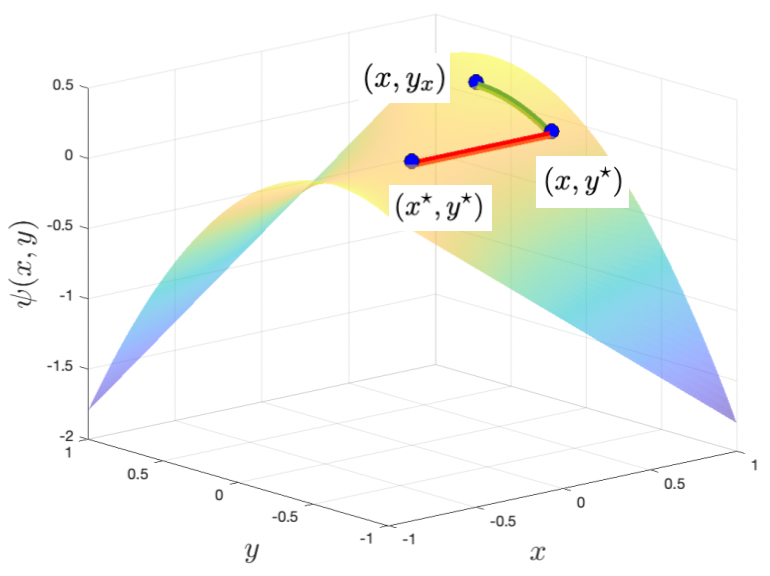}\label{fig:side}}
	\hfill
	\subfloat[]{\includegraphics[width=0.45\textwidth]{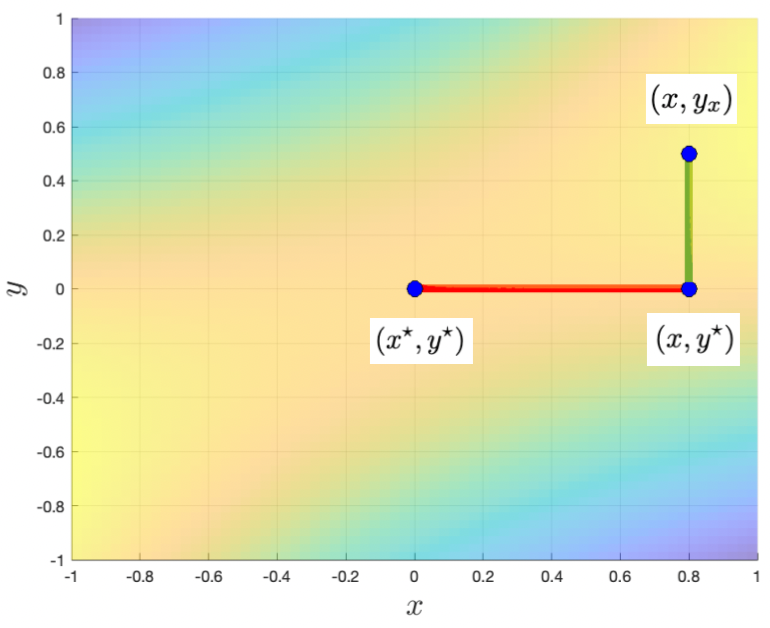}\label{fig:top-down}}
	\cprotect\caption{\label{fig:key-lemma-intuition}
    \small
		An example to give intuition behind Lemma~\ref{lem:bound-dual-div-primal-subopt-informal}. Here, $\psi(x, y) = xy - 0.8 y^2$, $(\xopt, \yopt) = (0, 0)$, $x = 0.8$, and $y_x = 0.5$. To see why it is possible to bound $|\yopt - y_x|$ in terms of the primal suboptimality $f(x) - f(\xopt)$, note that by the strong concavity of $\psi(x, \cdot)$ and the fact that $y_x$ is the maximizer of $\psi(x, \cdot)$ over $\yset$, we can upper bound $|\yopt - y_x|$ in terms of $\psi(x, y_x) - \psi(x, \yopt)$ (the vertical drop over the green line) via a standard strong-concavity inequality. In turn, $\psi(x, y_x) - \psi(x, \yopt)$ can be upper bounded by $\psi(x, y_x) - \psi(\xopt, \yopt) = f(x) - f(\xopt)$ (the vertical drop over the green line plus the vertical drop over the red line) due to the fact that $\psi(\xopt, \yopt) \le \psi(x, \yopt)$ by the optimality of $\xopt$.
	}
\end{figure}

\begin{lemma}[Lemma~\ref{lem:dual-div-bound-exact-best-response} specialized]
    \label{lem:bound-dual-div-primal-subopt-informal}
    For a given $x \in \xset$, suppose $-\psi(x, \cdot)$ is $\mu$-strongly convex relative to the dgf $r$ for some $\mu > 0$. Then $y_x \defeq \argmax_{y \in \yset} \psi(x, y)$ satisfies
\begin{align*}
    \breg{}{y_x}{\yopt} \le \frac{f(x) - f(\xopt)}{\mu}.
\end{align*}
\end{lemma}

\paragraph{A first try.} In particular, Lemma~\ref{lem:bound-dual-div-primal-subopt-informal} suggests the following approach: Define ``dual-regularized'' versions of $\psi, \phi, f$ as follows for $\lambda > 0$ and $y_0 \in \yset$:
\begin{align*}
    \psi_1 (x, y) &\defeq \psi(x, y) - \lambda V_{y_0}(y), \\
    f_1(x) &\defeq \max_{y \in \yset} \psi_1 (x, y), \\
    \phi_1 (y) &\defeq \min_{x \in \xset} \psi_1 (x, y)\,.
\end{align*}
(Here, the subscript 1 denotes one level of regularization and will be extended later.)
For any $x \in \xset$, note that $- \psi_1(x, \cdot)$ is $\lambda$-strongly convex relative to $r$, in which case Lemma~\ref{lem:bound-dual-div-primal-subopt-informal} applied to $\psi_1$ yields 
\begin{align}
    \label{eq:first-try-div-bound}
    \breg{}{y_x}{\yopt_1} \le \frac{f_1(x) - f_1(\xopt_1)}{\lambda},
\end{align}
for $\yopt_1 \defeq \argmax_{y \in \yset} \phi_1(y)$, $\xopt_1 \in \argmin_{x \in \xset} f_1(x)$, and $y_x \defeq \argmax_{y \in \yset} \psi_1(x, y)$. Then note
\begin{align}
    \label{eq:first-try-per-on-original}
    \phi(\yopt_1) \ge \phi_1(\yopt_1) \ge \phi_1(\yopt) = \min_{x \in \xset} \inbraces{\psi(x, \yopt) - \lambda \breg{}{y_0}{\yopt}} = \phi(\yopt) - \lambda \breg{}{y_0}{\yopt},
\end{align}
where the first inequality follows since $\phi \ge \phi_1$ pointwise. Then by the $L$-Lipschitzness of $\phi$ and $\mu_r$-strong convexity of $r$, it is straightforward to bound the suboptimality of $y_x$ as 
\begin{align}
    \label{eq:first-try-subopt}
    \phi(\yopt) - \phi(y_x)  \le  \lambda \breg{}{y_0}{\yopt} + L \sqrt{\frac{2 (f_1(x) - f_1(\xopt_1))}{\mu_r \lambda}}.
\end{align}
Consequently, an $\epsilon$-optimal point for \eqref{eq:dual} can be obtained via our oracles as follows: Set $\lambda \gets \frac{\epsilon}{2 \breg{}{y_0}{\yopt}}$, and use the $\DRPO$ oracle 
on the regularized primal problem
to obtain $x \in \xset$ such that
\begin{equation}
	\label{eq:framework_derive:1}
f_1(x) - f_1(\xopt_1) \le \frac{\epsilon^3 \mu_r}{16 L^2 \cdot \breg{}{y_0}{\yopt}}\,.
\end{equation}
Then the best response to $x$ with respect to $\psi_1$, namely $y_x \defeq \argmax_{y \in \yset} \psi_1(x, y)$, is $\epsilon$-optimal by \eqref{eq:first-try-subopt}. However, a typical setting in our applications is $\breg{}{y_0}{\yopt} = \Omega(1)$, $\mu_r = 1$, and $L \ge 1$, in which case ensuring \eqref{eq:framework_derive:1}  requires solving the regularized primal problem to $O(\epsilon^3)$ error.

\paragraph{Recursive regularization and the dual-extraction framework.} To lower the accuracy requirements, we apply dual regularization recursively. A key issue with the preceding argument is that it required a nontrivial bound on $\breg{}{y_0}{\yopt}$. However, it provided us with a nontrivial bound $\eqref{eq:first-try-div-bound}$ on $\breg{}{y_x}{\yopt_1}$, the ``level-one equivalent'' of $\breg{}{y_0}{\yopt}$. This suggests solving $f_1$ to lower accuracy while still obtaining a bound on $\breg{}{y_x}{\yopt_1}$ due to \eqref{eq:first-try-div-bound}, and then adding regularization centered at $y_x$ with a larger value of $\lambda$. Indeed, our framework recursively repeats this process until the total regularization is large enough so that (a term similar to) the right-hand side of \eqref{eq:first-try-subopt} can be bounded by $\epsilon$, despite never needing to solve a regularized primal problem to high accuracy.

To more precisely describe our approach, let $\psi_0 \defeq \psi, f_0 \defeq f, \phi_0 \defeq \phi$. Over iterations $k = 1, 2, \dots, K$, our framework implicitly constructs a sequence of convex-concave games $\psi_k : \R^d \times \R^n \to \R$, along with corresponding primal and dual functions $f_k : \xset \to \R$ and $\phi_k : \yset \to \R$ respectively, as follows:
\begin{align}
    \label{eq:recursive-reg-functions}
    \begin{split}
    \psi_k(x, y) &\defeq \psi_{k - 1}(x, y) - \lambda_{k - 1} \breg{}{y_{k - 1}}{y}, \\
    f_k(x) &\defeq \max_{y \in \yset} \psi_k(x, y), \\
    \phi_k(y) &\defeq \min_{x \in \xset} \psi_k(x, y). \\
    \end{split}
\end{align}
Here, $\inparen{\lambda_k \in \R_{>0}}_{k = 0}^{K - 1}$ is a dual-regularization schedule given as input to the framework, and $\inparen{y_k \in \yset}_{i = 0}^K$ is a sequence of dual-regularization ``centers'' generated by the algorithm, with $y_0$ given as input. For $k \in \inbraces{0} \cup [K]$, it will be useful to let $\yopt_k$ denote a maximizer of $\phi_k$ over $\yset$ and $\xopt_k$ denote a minimizer of $f_k$ over $\xset$, with $\yopt_0 \defeq \yopt$ and $\xopt_0 \defeq \xopt$ in particular.

Over the $K$ rounds of recursive dual regularization, we aim to balance two goals:
\begin{itemize}
    \item On the one hand, we want $\lambda_k$ to increase quickly so that $- \psi_k(x, \cdot)$ is very strongly convex relative to $r$, thereby allowing us to apply Lemma~\ref{lem:bound-dual-div-primal-subopt-informal} with a larger strong convexity constant.

    \item On the other hand, we want to maintain the invariant that, roughly speaking, $\yopt_k$ is always $\epsilon / 2$-optimal for the original dual $\phi$. Indeed, we were constrained in choosing $\lambda$ in \eqref{eq:first-try-per-on-original} to be on the order of $\epsilon / \breg{}{y_0}{\yopt}$ to ensure $\yopt_1$ is $\epsilon / 2$-optimal for $\phi$. A similar ``constraint'' on the dual-regularization schedule $(\lambda_k)_{k = 0}^{K - 1}$ appears when \eqref{eq:first-try-per-on-original} is extended to additional levels of regularization. This prevents us from increasing $\lambda_k$ too quickly.
\end{itemize}
In all the applications in this paper we choose $\lambda_k \approx 2 \lambda_{k - 1}$. $\lambda_0$ typically must remain on the order of $\epsilon / \breg{}{y_0}{\yopt}$ due to the second point.

Pseudocode of the framework is given in Algorithm~\ref{alg:dual-extraction-framework-sketch}.
Each successive dual-regularization center $y_k$ is computed via the $\DRBR$ oracle (Line~\ref{algline:y_k-informal}) as a best response to a primal point $x_k$ obtained via the $\DRPO$ oracle (Line~\ref{algline:x_k-informal}). In Section \ref{sec:dual-extraction-framework}, we generalize Algorithm~\ref{alg:dual-extraction-framework-sketch} (cf. Algorithm~\ref{alg:dual-extraction-framework}) in several ways: (i) we allow for stochasticity in the $\DRPO$ oracle; (ii) we allow for distance-generating functions $r$ such that $\dom r \ne \R^n$; (iii) we give different but equivalent characterizations of $x_k$ and $y_k$ which facilitate the derivation of explicit expressions for the $\DRPO$ and $\DRBR$ oracles in applications.

\begin{algorithm}[h] %
	\setstretch{1.1}
	\caption{Dual-extraction framework (Algorithm~\ref{alg:dual-extraction-framework} specialized)}
	\label{alg:dual-extraction-framework-sketch}
	\LinesNumbered
	\DontPrintSemicolon
	\Input{
		Initial dual-regularization center $y_0 \in \yset$, iteration count $K \in \N$, dual-regularization schedule $(\lambda_k \in \R_{>0})_{k = 0}^{K - 1}$, primal-accuracy schedule $(\epsp_k \in \R_{>0})_{k = 1}^K$, $\DRPO$ and $\DRBR$ oracles
	}

    $\psi_0 \defeq \psi$, $f_0 \defeq f$, and $\phi_0 \defeq \phi$\;

    \For{$k = 1, 2, \ldots, K$ }{ 
        Define $\psi_k$, $f_k$, and $\phi_k$ as in \eqref{eq:recursive-reg-functions}\;
        Let $x_k \in \xset$ be such that $f_k(x_k) - f_k(\xopt_k) \le \epsilon_k$ \tcp*{Computed via the $\DRPO$ oracle} \label{algline:x_k-informal}
        $y_k = \argmax_{y \in \yset} \psi_k(x_k, y) $ \tcp*{Computed via the $\DRBR$ oracle} \label{algline:y_k-informal}
    }

    \Return $y_K$\;
\end{algorithm}

\paragraph{Analysis of Algorithm~\ref{alg:dual-extraction-framework-sketch}.} Theorem~\ref{thm:main-framework-result-informal} is our main result for Algorithm~\ref{alg:dual-extraction-framework-sketch}. We then instantiate Theorem~\ref{thm:main-framework-result-informal} with two illustrative choices of parameters in Corollaries \ref{rem:example-schedules-wlog} and \ref{rem:example-schedules-no-log}, and defer the proofs of the latter to their general versions in Section \ref{sec:dual-extraction-framework}.
All of the remarks below (Remarks \ref{rem:picking-parameters-main-thm}, \ref{rem:example-schedules-wlog}, \ref{rem:example-schedules-no-log}) are stated with reference to the specialized results in this section (Theorem~\ref{thm:main-framework-result-informal} and Corollaries \ref{cor:example-schedules-wlog-informal}, \ref{cor:example-schedules-no-log-informal} resp.), but extend immediately to the corresponding general versions (Theorem~\ref{thm:main-framework-guarantee} and Corollaries \ref{cor:example-schedules-wlog}, \ref{cor:example-schedules-no-log} resp.).

\begin{theorem}[Theorem~\ref{thm:main-framework-guarantee} specialized]
    \label{thm:main-framework-result-informal}
    Algorithm~\ref{alg:dual-extraction-framework-sketch} returns $y_K$ satisfying
\begin{align}
    \label{eq:informal-dist-bound}
    \breg{}{y_K}{u} \le \frac{\epsilon_K}{\Lambda_K}
    \text{ where } 
    \Lambda_k \defeq \sum_{j = 0}^{k - 1} \lambda_j
    \text{ for }
    k \in [K]
\end{align}
and $u \in \yset$ is a point with dual suboptimality bounded as
\begin{align}
    \label{eq:informal-subopt-to-u}
    \phi(\yopt) - \phi(u) \le \lambda_0 \breg{}{y_0}{\yopt} + \sum_{k = 1}^{K - 1} \frac{\lambda_k}{\Lambda_k} \epsp_k.
\end{align}
If we additionally assume that $\phi$ is $L$-Lipschitz with respect to $\norm{\cdot}$, we can directly bound the suboptimality of $y_K$ as
\begin{align}
    \label{eq:informal-subopt}
    \phi(\yopt) - \phi(y_K) \le  \lambda_0 \breg{}{y_0}{\yopt} + \sum_{k = 1}^{K - 1} \frac{\lambda_k}{\Lambda_k} \epsp_k + L \sqrt{\frac{2}{\mu_r} \frac{\epsp_K}{\Lambda_K}}.
\end{align}
\end{theorem}

\begin{proof}
We claim the first half of Theorem~\ref{thm:main-framework-result-informal} holds with $u \gets \yopt_K$. To see this, note that we can bound the suboptimality of $\yopt_K$ as
\begin{align*}
    \phi(\yopt_K) \overge{(i)} \phi_K(\yopt_K) \ge \phi_K(\yopt_{K - 1}) &= \max_{x \in \xset} \inbraces{\psi_{K - 1}(x, \yopt_{K - 1}) - \lambda_{K - 1} \breg{}{y_{K - 1}}{\yopt_{K - 1}}} \\
    &= \phi_{K - 1}(\yopt_{K - 1}) - \lambda_{K - 1} \breg{}{y_{K - 1}}{\yopt_{K - 1}} \\
    &\overge{(ii)} \phi_0(\yopt_0) - \lambda_0 \breg{}{y_0}{\yopt_0} - \sum_{k = 1}^{K - 1} \lambda_k \breg{}{y_k}{\yopt_k} \\
    &\overge{(iii)} \phi(\yopt) - \lambda_0 \breg{}{y_0}{\yopt} - \sum_{k = 1}^{K - 1} \frac{\lambda_k}{\Lambda_k} \epsilon_k,
\end{align*}
where $(i)$ follows since $\phi \ge \phi_K$ pointwise, $(ii)$ follows from repeating the argument in the previous lines recursively (starting by lower bounding $\phi_{K - 1}(\yopt_{K - 1})$, etc.), and $(iii)$ uses Lemma~\ref{lem:bound-dual-div-primal-subopt-informal} applied to $\psi_k$, which yields by Lines \ref{algline:x_k-informal} and \ref{algline:y_k-informal} in Algorithm~\ref{alg:dual-extraction-framework-sketch}:
\begin{align*}
    \breg{}{y_k}{\yopt_k} \le \frac{f_k(x_k) - f_k(\xopt_k)}{\Lambda_k} \le \frac{\epsilon_k}{\Lambda_k},
\end{align*}
since $\psi_k(x, \cdot) = \psi(x, \cdot) + \sum_{j = 0}^{k - 1} \lambda_j \breg{}{y_j}{\cdot}$ is $\Lambda_k$-strongly concave relative to $-r$. Thus, we have proven Equation~\ref{eq:informal-subopt-to-u}, and Equation~\ref{eq:informal-dist-bound} follows again from Lemma~\ref{lem:bound-dual-div-primal-subopt-informal} applied to $\psi_K$. Equation~\ref{eq:informal-subopt} then follows since the fact that $r$ is $\mu_r$-strongly convex with respect to $\norm{\cdot}$ and Equation~\ref{eq:informal-dist-bound} imply
\begin{align*}
    \norm{y_K - \yopt_K} \le \sqrt{ \frac{2}{\mu_r} \breg{}{y_K}{\yopt_K} } \le \sqrt{\frac{2}{\mu_r} \frac{\epsilon_K}{\Lambda_K}}.
\end{align*}
\end{proof}

We give a remark regarding how to pick the parameters $(\lambda_k)_{k = 0}^{K - 1}$ and $(\epsilon_k)_{k = 1}^K$ when applying Theorem~\ref{thm:main-framework-result-informal}:

\begin{remark}[Picking the parameters for Theorem~\ref{thm:main-framework-result-informal}]
    \label{rem:picking-parameters-main-thm}
    Equation~\ref{eq:informal-subopt} can be interpreted as follows: To ensure $y_K$ is $\epsilon$-optimal for $\phi$, it suffices to choose the sequences $(\lambda_k)_{k = 0}^{K - 1}$ and $(\epsilon_k)_{k = 1}^K$ so that the right side of \eqref{eq:informal-subopt} is at most $\epsilon$. Then the first term, $\lambda_0 \breg{}{y_0}{\yopt}$, constrains $\lambda_0$ to be on the order of $\epsilon / \breg{}{y_0}{\yopt}$. Skipping ahead, the third term, $L \sqrt{\frac{2}{\mu_r} \frac{\epsp_K}{\Lambda_K}}$, is the reason we always choose $\lambda_k \approx 2 \lambda_{k - 1}$ in our applications, as this ensures $\Lambda_K$ is large enough to handle this term with $K$ only needing to be logarithmic in the problem parameters. Then the second term, $\sum_{k = 1}^{K - 1} \frac{\lambda_k}{\Lambda_k} \epsp_k$, effectively constrains roughly $\sum_{k = 1}^{K - 1} \epsilon_k \le \epsilon$, as $\lambda_k / \Lambda_k \approx 1$.
\end{remark}

\begin{corollary}[Corollary~\ref{cor:example-schedules-wlog} specialized] \label{cor:example-schedules-wlog-informal}
    Suppose $\phi$ is $L$-Lipschitz with respect to $\norm{\cdot}$, and let $B > 0$ be such that $\breg{}{y_0}{\yopt} \le B$. Then for any $\epsilon > 0$,  and $K \ge \max \inbraces{\log_2 \frac{L^2 B}{\mu_r \epsilon^2}, 1} + 10$, the output of Algorithm~\ref{alg:dual-extraction-framework-sketch} with dual-regularization and primal-accuracy schedules of
    \begin{align*}
        \lambda_k = 2^k \frac{\epsilon}{4 B} \text{ for $k \in \inbraces{0} \cup [K - 1]$}
        \text{ and }
        \epsp_k = \frac{\epsilon}{4 K} \text{ for $k \in [K]$}
    \end{align*}
    satisfies $\phi(\yopt) - \phi(y_K) \le \epsilon$.
    \end{corollary}

    \begin{remark}
        \label{rem:example-schedules-wlog}
    Corollary~\ref{cor:example-schedules-wlog-informal} achieves the stated goal of obtaining an $\epsilon$-optimal point for \eqref{eq:dual} by running for a number of iterations which depends logarithmically on the problem parameters, and solving each dual-regularized primal subproblem to an accuracy of $\epsilon$ divided by a logarithmic factor. Note in particular the logarithmic dependence on the dual divergence bound $B$ and dual Lipschitz constant $L$, meaning these are weak assumptions. Furthermore, it is clear from the proof of Theorem~\ref{thm:main-framework-result-informal} that $\phi$ only need be $L$-Lipschitz on a set containing $y_K$ and $\yopt_K$.
    \end{remark}
    
\begin{corollary}[Corollary~\ref{cor:example-schedules-no-log} specialized]
    \label{cor:example-schedules-no-log-informal}
    Let $B > 0$ be such that $\breg{}{y_0}{\yopt} \le B$. Then for any $\epsilon > 0$ and $K \in \N$, the output of Algorithm~\ref{alg:dual-extraction-framework-sketch} with dual-regularization and primal-accuracy schedules of
    \begin{align*}
        \lambda_k = 2^k \frac{\epsilon}{4 B} \text{ for $k \in \inbraces{0} \cup [K - 1]$}
        \text{ and }
        \epsp_k = \frac{\epsilon}{8 \cdot 1.5^k} \text{ for $k \in [K]$}
    \end{align*}
    satisfies
    \begin{align*}
        \norm{y_K - u} \le  \frac{1}{1.5^K} \sqrt{\frac{2 B}{\mu_r}},
    \end{align*}
    where $u \in \yset$ is a point whose suboptimality is at most $\epsilon$, i.e., $\phi(\yopt) - \phi(u) \le \epsilon$. 
\end{corollary}

\begin{remark}
    \label{rem:example-schedules-no-log}
Later calls to the $\DRPO$ oracle during the run of Algorithm~\ref{alg:dual-extraction-framework-sketch} may be cheaper since there will be a significant amount of dual regularization at that point (namely, $\Lambda_k = \sum_{j = 0}^{k - 1} \lambda_j$ is large). One can sometimes take advantage of this (in particular, if the cost of a $\DRPO$ oracle call scales inverse polynomially with the regularization) to design schedules that avoid even the typically additional multiplicative logarithmic cost of Corollary~\ref{cor:example-schedules-wlog-informal} over the cost of a single $\DRPO$ oracle call. In such cases, a choice of schedules similar to those of Corollary~\ref{cor:example-schedules-no-log-informal} is often appropriate. With this choice of schedules, later rounds require very high accuracy. However, if one can argue that the increasing dual regularization $\Lambda_k$ makes the $\DRPO$ oracle call cheaper at a faster rate than the decreasing error $\epsilon_k$ makes it more expensive (as we do in Section \ref{sec:convex-critical-point}), the total cost of applying the framework may collapse geometrically to the cost of a single $\DRPO$ oracle call made with target error approximately $\epsilon$.

We purposely state Corollary~\ref{cor:example-schedules-no-log-informal} without the assumption that $\phi$ is Lipschitz because that is the form we will use in Section~\ref{sec:convex-critical-point}. However, it is straightforward to reformulate a version of Corollary~\ref{cor:example-schedules-no-log-informal} with the Lipschitz assumption. Here the focus was to illustrate different possible choices of schedules.
\end{remark}

\subsection{Related work}\label{subsec:related-work}

\paragraph{Black-box reductions.} Our main contribution can be viewed as a black-box reduction from (regularized) primal optimization to dual optimization. Similar black-box reductions exist in the optimization literature. For example, \cite{allenzhu2016optimalblackbox} develops reductions between various fundamental classes of optimization problems, e.g., strongly convex optimization and smooth optimization. In a similar vein, the line of work \cite{lin2015catalyst,frostig2015unregularizing,carmon2022RECAPP} reduces convex optimization to approximate proximal point computation (i.e., regularized minimization).

\paragraph{Bilinear matrix games.} Consider the bilinear objective $\psi(x,y) = x^\top A y$ where $\xset$ and $\yset$ are either the simplex, $\simplex^{k} \defeq \{x \in \R^k_{\geq 0} : \norm{x}_1 =1\}$, or the Euclidean ball, $B^{k} \defeq \{x \in \R^k : \norm{x}_2 \leq 1\}$. State-of-the-art methods in regard to runtime for obtaining an approximately optimal primal and/or dual solution can be divided into second-order interior point methods \cite{cohen2021solving,brand2021minimum} and stochastic first-order methods \cite{grigoriadis1995sublinear,clarkson2012sublinear,carmon2019variance,carmon2023whole}; see Table 2 in \cite{carmon2023whole} for a summary of the best known runtimes as well as other references. Of importance to this paper, all state-of-the-art algorithms other than that of \cite{carmon2023whole} are either (i) primal-dual algorithms which return both an $\epsilon$-optimal primal and dual solution simultaneously, and/or (ii) achieve runtimes which are symmetric in the primal dimension $d$ and dual dimension $n$, meaning the cost of obtaining an $\epsilon$-optimal dual solution is the same as that of obtaining an $\epsilon$-optimal primal solution. The algorithm of \cite{carmon2023whole}, on the other hand, only returns an $\epsilon$-optimal primal point and further has a runtime which is not symmetric in $n$ and $d$ (see the footnote on the first page of that paper). As a result, solving the dual by simply swapping the roles of the primal and dual variables may be more expensive than solving the primal. (In fact, swapping the variables in this way may not even always be possible without further modifications due to restrictions on the constraint sets.)

\paragraph{CVaR at level $\alpha$ distributionally robust optimization (DRO).} The DRO objectives we study are of the form $\psi(x, y) = \sum_{i = 1}^n y_i f_i(x)$, where the functions $f_i : \R^d \to \R$ are convex, bounded, and Lipschitz, and $\yset$, known as the uncertainty set, is a subset of the simplex. This objective corresponds to a robust version of the empirical risk minimization (ERM) objective where instead of taking an average over the losses (namely, $y_i$ is fixed at $1 / n$), larger losses may be given more weight. In particular, in this paper we focus on a canonical DRO setting, \emph{CVaR at level $\alpha$,} where the uncertainty set is given by $\yset \defeq \inbraces{y \in \simplex^n : \norm{y}_\infty \le \frac{1}{\alpha n}}$ for a choice of $\alpha \in [1 / n, 1]$. CVaR DRO, along with its generalization $f$-divergence DRO, has been of significant interest over the past decade; see \cite{levy2020largescale,carmon2022distributionally,curi2020adaptivesampling,namkoong2016fdivergences,duchi2021learningmodels} and the references therein. \cite{levy2020largescale} is the most relevant to this paper---omitting parameters other than $\alpha$, the number of losses $n$, and the target accuracy $\epsilon > 0$, they give a matching upper and lower bound (up to logarithmic factors) of $\Otilde(\alpha^{-1} \epsilon^{-2})$ first-order queries of the form $(\grad f_i(x), f_i(x))$ to obtain an expected $\epsilon$-optimal point of the primal objective. Their upper bound is achieved by a stochastic gradient method where the gradient estimator is based on a multilevel Monte Carlo (MLMC) scheme \cite{giles2008multilevel,giles2015multilevel}. However, the best known complexity for obtaining an expected $\epsilon$-optimal point of the dual of CVaR at level $\alpha$ is $O(n \epsilon^{-2})$ via a primal-dual method based on \cite{nemirovski2009robust}; see also \cite{curi2020adaptivesampling,namkoong2016fdivergences} as well as \cite[Appendix A.1]{carmon2022distributionally}, the last of which obtains complexity $\Otilde (n \epsilon^{-2})$ in the more general setting of the uncertainty set being an $f$-divergence ball.

\paragraph{Stationary point computation.} For $\gamma > 0$, convex and $\beta$-smooth $h : \R^n \to \R$ with global minimum $\zopt$, and initialization point $z_0$, consider the problem of computing a point $z$ such that $\norm{\grad h(z)}_2 \le \gamma$. Two worst-case optimal gradient query complexities for this problem exist in the literature: $O \inparen{ \sqrt{\beta (h(z_0) - h(\zopt))} / \gamma }$ and $O \inparen{\sqrt{\beta \norm{z_0 - \zopt}_2 / \gamma}}$. An algorithm (the OGM-G method) which achieves the former complexity was given in \cite{kim2020optimizing}, and \cite{nesterov2019primaldual} pointed out that any algorithm which achieves the former complexity can achieve the latter complexity. This is obtainable by running $N$ iterations of any optimal gradient method for reducing the function value, followed by $N$ iterations of a method which achieves the former complexity for reducing the gradient magnitude. In what may be of independent interest, we observe in Section \ref{subsec:critical-point-prelims} that a reduction in the opposite direction is also possible. More broadly, algorithms and frameworks for reducing the gradient magnitude of convex functions have been of much recent interest, and further algorithms and related work for this problem include \cite{kim2023timereversed, lan2023optimal, kim2024mirror, lee2021geometric, diakonikolas2021potential, nesterov2019primaldual,grapiglia2022tensormethodsstationary}, with lower bounds given in \cite{nemirovsky1991optimality, nemirovsky1992information}.

\subsection{Paper organization}\label{subsec:organization}

In Section \ref{sec:notation-and-assumptions}, we go over notation and conventions for the rest of the paper. 
We give our general dual-extraction framework and its guarantees in Section \ref{sec:dual-extraction-framework}. In Section \ref{sec:maximin-algorithms}, we apply our framework to bilinear matrix games and the CVaR at level $\alpha$ DRO problem. Finally, in Section \ref{sec:convex-critical-point} we give an optimal algorithm (in terms of query complexity) for computing an approximate stationary point of a convex and $\beta$-smooth function.
\section{Notation and conventions}
\label{sec:notation-and-assumptions}

\paragraph{General notation and conventions.}
For $\psi : \R^d \times \R^n \to \R$, we let $\grad_x \psi(x, y)$ (resp.\ $\grad_y \psi(x, y)$) denote the partial gradient of $\psi$ with respect to the first (resp.\ second) variable, evaluated at $(x, y) \in \R^d \times \R^n$. We use the notation $\psi(\cdot, y) : \R^d \to \R$ for a fixed $y \in \R^n$ to denote the map $x \mapsto \psi(x, y)$ (and define $\psi(x, \cdot)$ analogously). When we say $\psi(\cdot, y)$ satisfies a property, we mean it satisfies that property for any fixed $y \in \R^n$ (and analogously for $\psi(x, \cdot)$).

We let $[K] \defeq \inbraces{1, 2, \dots, K}$, $\simplex^n \defeq \inbraces{x \in \R_{\ge 0}^n : \norm{x}_1 = 1}$, and $B_r^n(x) \defeq \inbraces{x \in \R^n : \norm{x}_2 \le r}$. In the latter two definitions, we may drop the superscript $n$ if it is clear from context,  the argument $x$ if it is 0, and the subscript $r$ if it is 1. We use the notation $\Otilde(\cdot)$, $\Omegatilde(\cdot)$, and $\Thetatilde(\cdot)$ to hide polylogarithmic factors. $\bd U$, $\interior U$, and $\ri U$ denote the boundary, interior, and relative interior of the set $U$ respectively. For $y \in \R^n$, we may use either the notation $y_i$ or $[y]_i$ to denote its $i$-th entry. (In particular, the latter may be used in the form $[y_k]_i$ to denote the $i$-th entry of the $k$-th vector in a sequence of vectors $y_1, y_2, \dots.$) $\ones$ denotes the all-ones vector. For a function $f$ which depends on some inputs $x_1, \dots, x_k \in \R$, we write $f \le \poly(x_1, \dots, x_k)$ to denote the fact that $f$ is uniformly bounded above by a polynomial in $x_1, \dots, x_k$ as $x_1, \dots, x_k$ vary.

\paragraph{Approximate optimizers.} For $\epsilon > 0$,  $U \subseteq \R^n$, and $f : U \to \R$ we say that $u \in U$ is an $\epsilon$-minimizer of $f$ (or $\epsilon$-(sub)optimal if the fact that we are minimizing is clear from context) if $f(v) \ge f(u) - \epsilon$ for all $v \in U$. Similarly, we say a random point $u \in U$ is an expected $\epsilon$-minimizer (or expected $\epsilon$-(sub)optimal) if $f(v) \ge \E f(u) - \epsilon$ for all $v \in U$. $\epsilon$-maximizers and expected $\epsilon$-maximizers are defined analogously, and we may drop the qualifier ``expected'' if it is clear from context. We use the notation $\argmin^\epsilon$ and $\argmax^{\epsilon}$ to denote the set of $\epsilon$-minimizers and $\epsilon$-maximizers respectively.

\paragraph{Strong convexity/concavity, smoothness, and Lipschitzness.} For $\mu > 0$ and convex $U \subseteq V \subseteq \R^n$, we say $f : V \to \R$ is $\mu$-strongly convex with respect to a norm $\norm{\cdot}$ over $U$ if $f(tu + (1 - t)v) \le t f(u) + (1 - t)f(v) - \frac{\mu}{2} \cdot t (1 - t) \norm{u - v}^2$ for all $u, v \in U$. If $f$ is differentiable, then $f$ is $\mu$-strongly convex with respect to $\norm{\cdot}$ over $U$ if and only if $f(v) \ge f(u) + \inangle{\grad f(u), v - u} + \frac{\mu}{2} \norm{v - u}^2$ for all $u, v \in U$. We omit the qualification ``over $U$'' when $U = V$ or $U$ is clear from context, and we say a function is strongly convex if it is strongly convex for some $\mu > 0$. We say that differentiable $g : V \to \R$ is $\mu$-strongly convex relative to (differentiable) $f$ over $U$ if and only if $g(v) - g(u) - \inangle{\grad g(u), v - u} \ge \mu (f(v) - f(u) - \inangle{\grad f(u), v - u})$ for all $u, v \in U$. $f$ is $\mu$-strongly concave (relative to $g$) if and only if $-f$ is $\mu$-strongly convex (respectively relative to $-g$). Letting $\dualnorm{\cdot}$ denote the dual norm, $f$ is $\beta$-smooth over $U$ with respect to a norm $\norm{\cdot}$ if $\dualnorm{\grad f(u) - \grad f(v)} \le \beta \norm{u - v}$ for all $u, v \in U$. $f$ is $L$-Lipschitz over $U$ with respect to $\norm{\cdot}$ if $|f(u) - f(v)| \le L \norm{u - v}$ for all $u, v \in U$.

\paragraph{Distance-generating functions (dgfs) and Bregman divergences.} 
Following \cite[Sec. 6.4]{orabona2023modern}, we encapsulate the setup for a dgf as follows:

\begin{definition}[dgf setup]
    \label{def:dgf-setup}
We say $(\uset, \pset, \norm{\cdot}, r)$ is a \emph{dgf setup} over $\R^n$ for closed and convex sets $\uset \subseteq \pset \subseteq \R^n$ with $\uset \cap \interior \pset \ne \emptyset$ if: (i) the distance-generating function (dgf) $r : \pset \to \R$ is convex and continuous over $\pset$, differentiable on $\interior \pset$, and $\mu_r$-strongly convex with respect to the chosen norm $\norm{\cdot}$ on $\uset \cap \interior \pset$ for some $\mu_r > 0$; and (ii) either $\lim_{u \to \bd \pset} \norm{\grad r(u)}_2 = \infty$ or $\uset \subseteq \interior \pset$. 
\end{definition}

The assumptions in \Cref{def:dgf-setup} are common assumptions to ensure well-definedness of mirror descent or similar proximal-point-based methods (e.g., \cite[Sec. 3]{beck2003mirrordescent}, \cite[Thm. 3.12]{bauschke1997convexanalysis}, or \cite[Thm. 6.7]{orabona2023modern}); for example, two standard choices are: (i) $\uset \subseteq \R^n$, $\pset = \R^n$, $\norm{\cdot} = \norm{\cdot}_2$, and $r(u) = \frac{1}{2} \norm{u}_2^2$; and (ii) $\uset = \simplex^n$, $\pset = \R^n_{\ge 0}$, $\norm{\cdot} = \norm{\cdot}_1$, and $r(u) = \sum_{i = 1}^n u_i \ln u_i$. ($\mu_r = 1$ in both cases.) For a given dgf setup, we define its induced Bregman divergence $\breg{r}{u}{v} \defeq r(v) - r(u) - \inner{\grad r(u)}{v - u}$ for $u \in \interior \pset, v \in \pset$, and drop the superscript $r$ when it is clear from context.

\paragraph{Conjugation, indicators, and restrictions.} For an extended-real-valued convex function $f : \R^n \to [- \infty, \infty]$, its effective domain is $\dom f \defeq \inbraces{x \in \R^n : f(x) < \infty}$, and its convex or Fenchel conjugate $\fconj : \R^n \to [- \infty, \infty]$ is given by $\fconj(\theta) = \sup_{x \in \dom f} \inbraces{ \inangle{\theta, x} - f(x) }$. We collect some standard properties of the Fenchel conjugate and other useful convex-analytic facts in Appendix \ref{app:convex-analysis-facts}. For $S \subseteq \R^n$, we let $\indc_S$ denote the infinite indicator of $S$, namely $\indc_S(x) = 0$ if $x \in S$ and $\indc_S(x) = \infty$ if $x \notin S$. For a function $f : S \to [- \infty, \infty]$ initially defined on a strict subset $S \subset \R^n$, we may implicitly extend the domain of $f$ to all of $\R^n$ via its indicator as $f + \indc_S$ without additional comment. For a function $f : U \to [- \infty, \infty]$ with $S \subseteq U \subseteq \R^n$, we let $f_S \defeq f + \indc_S$ denote the restriction of $f$ to $S$. We note that $f^*_S$ denotes the convex conjugate of $f_S$ (and not $\fconj$ restricted to $S$).

\section{Dual-extraction framework}\label{sec:dual-extraction-framework}

In this section, we provide our general dual-extraction framework and its guarantees. In Section \ref{subsec:main-framework-assumptions-preliminaries}, we give the general setup, oracle definitions, and assumptions with which we apply and analyze the framework. Section \ref{subsec:main-framework} contains the statement and guarantees of the framework and Section~\ref{subsec:main-framework-proofs} contains the associated proofs.

\subsection{Preliminaries}
\label{subsec:main-framework-assumptions-preliminaries}

We bundle all of the inputs to our framework into what we call a \emph{dual-extraction setup}, defined below. Recall that when we say $\psi(x, \cdot)$ satisfies a property, we mean it satisfies that property for any fixed $x \in \R^d$ (and analogously for $\psi(\cdot, y)$).

\begin{definition}[Dual-extraction setup]
    \label{def:dual-extraction-setup}
A \emph{dual-extraction setup} is a tuple $(\psi, \xset, \yset, \uset, \pset, \norm{\cdot}, r)$ where: 
\begin{enumerate*}[series = tobecont, itemjoin =, label=(\roman*)]
    \item $\psi(x, \cdot)$ is differentiable; \label{as:dif}
    \item $\psi(\cdot, y)$ and $\psi(x, \cdot)$ are convex and concave respectively; \label{as:convex-concave}
    \item $(\uset, \pset, \norm{\cdot}, r)$ is a dgf setup over $\R^n$ per Definition~\ref{def:dgf-setup}; \label{as:dgf}
    \item the constraint sets $\xset \subseteq \R^d$ and $\yset \subseteq \R^n$ are nonempty, closed, and convex with $\yset \subseteq \uset$ and $\yset \cap \interior \pset \ne \emptyset$; \label{as:constraint-sets}
    \item $\xset$ is bounded or $\psi(\cdot, y)$ is strongly convex; \label{as:bounded-strongly-convex}
    \item $\yset$ is bounded or $\psi(x, \cdot)$ is strongly concave; \label{as:bounded-strongly-concave}
    \item over all $p \in \uset \cap \interior \pset$ and $w \in \partial \indc_\uset(p)$, the map $y \mapsto \inangle{w, y}$ is constant over $\yset$.\footnote{In all of our applications, this map will in fact be constant over $\uset$.} \label{as:technical-dual}
\end{enumerate*}
\end{definition}
 
Assumption~\ref{as:dif} is only used in the proofs of Lemma~\ref{lem:dual-div-bound-exact-best-response} (the general version of Lemma~\ref{lem:bound-dual-div-primal-subopt-informal} from Section \ref{subsec:dual-extraction-overview}) and Corollary~\ref{cor:main-framework-well-defined} (used to show the framework is well-defined when $\dom r \ne \R^n$).
Assumptions \ref{as:convex-concave}, \ref{as:bounded-strongly-convex}, and \ref{as:bounded-strongly-concave} ensure that the minimax optimization problem with objective $\psi$ and constraint sets $\xset$ and $\yset$ satisfies the \emph{minimax principle}; see below. Regarding Assumptions \ref{as:dgf}, \ref{as:constraint-sets}, and \ref{as:technical-dual}, the fact that $\yset$ is potentially a strict subset of $\uset$ as well as the necessity of the technical assumption \ref{as:technical-dual} is discussed in Remark \ref{rem:technical-dual-assumption-and-uset}. In particular, Assumption~\ref{as:technical-dual} is only used to derive an equivalent formulation of the framework to Algorithm~\ref{alg:dual-extraction-framework-sketch} which often allows for easier instantiations in applications, but is not strictly necessary to obtain our guarantees.

While our main results are stated in the full generality of Definition~\ref{def:dual-extraction-setup}, in our applications we only particularize to 
\Cref{ex:unbounded} and \Cref{ex:simplex} introduced below.

\begin{definition}[Unbounded setup]
    \label{ex:unbounded}
    \sloppy
    A \emph{$(\psi, \xset, \yset, r)$-unbounded setup} is a $(\psi, \xset, \yset, \R^n, \R^n, \norm{\cdot}_2, r)$-dual-extraction setup.
\end{definition}

In other words, in an unbounded setup we choose $\uset = \pset = \R^n$ and the Euclidean norm, in which case the dgf $r$ can be any differentiable and strongly convex function with respect to $\norm{\cdot}_2$. Note that Assumption~\ref{as:technical-dual} is trivial as $\partial \indc_\uset (p) = \inbraces{0}$ for all $p \in \R^n$.

\begin{definition}[Simplex setup]
    \label{ex:simplex}
    A \emph{$(\psi, \xset, \yset)$-simplex setup} is a $(\psi, \xset, \yset, \simplex^n, \R^n_{\ge 0}, \norm{\cdot}_1, r)$-dual-extraction setup where $r(u) \defeq \sum_{i = 1}^n u_i \ln u_i$ (with $0 \ln 0 \defeq 0$).
\end{definition}

In other words, in a simplex setup we choose $\uset = \simplex^n$, $\pset = \R^n_{\ge 0}$, we are using the $\ell_1$-norm, and the dgf is negative entropy when restricted to the simplex. It is a standard result known as Pinsker's inequality that $r$ is 1-strongly convex over $\simplex_{>0}^n$ with respect to $\norm{\cdot}_1$, and the associated Bregman divergence is given by the Kullback-Leibler (KL) divergence $\breg{}{u}{w} = \sum_{i = 1}^n w_i \ln \frac{w_i}{u_i}$ for $u \in \simplex^n_{>0}$ and $w \in \simplex^n$.
     We verify that Assumption~\ref{as:technical-dual} holds in Appendix \ref{subapp:verify-simplex-setting}.

\paragraph{Notation associated with a setup.} Whenever we instantiate a dual-extraction setup (Definition~\ref{def:dual-extraction-setup}), we use the following notation and oracles associated with that setup without additional comment. We define the associated primal $f : \xset \to \R$ and dual $\phi : \yset \to \R$ functions, along with their corresponding primal and dual optimization problems, as they were introduced above in \eqref{eq:primal} and \eqref{eq:dual}. We let $\xopt \in \argmin_{x \in \xset} f(x)$ and $\yopt \in \argmax_{y \in \yset} \phi(y)$ denote arbitrary primal and dual optima. 
To facilitate the discussion of dual-regularized problems, we define $f_{\lambda, q}(x) : \xset \to \R$ as follows
\begin{align*}
    f_{\lambda, q}(x) \defeq \max_{y \in \yset} \inbraces{\psi(x, y) - \lambda V_q(y)}
    \text{ for }
    \lambda > 0
    \text{ and }
    q \in \uset \cap \interior \pset\,.
\end{align*}

\paragraph{The minimax principle.} Assumptions \ref{as:convex-concave}, \ref{as:bounded-strongly-convex}, and \ref{as:bounded-strongly-concave} in Definition~\ref{def:dual-extraction-setup} guarantee $f(\xopt) = \psi(\xopt, \yopt) = \phi(\yopt)$, which we refer to as the \textit{minimax principle}. See, e.g., \cite{sion1958general,neumann1928zur} as well as Propositions 1.2 and 2.4 in \cite[Ch. VI]{ekeland1999convex}.

\paragraph{Oracle definitions.} Our framework assumes black-box access to $\psi$, $\xset$, and $\yset$ via a dual-regularized primal optimization ($\DRPO$) oracle and a dual-regularized dual best response ($\DRBR$) oracle defined below. Note that we generalize the setting of Section \ref{subsec:dual-extraction-overview} by allowing the $\DRPO$ oracle to return an expected $\epsilon$-optimal point; this is used in our applications in Section~\ref{sec:maximin-algorithms}.

\begin{definition}[$\DRPO$ oracle]
    \label{def:DRPO}
    A \emph{$(q \in \uset \cap \interior \pset, \lambda > 0, \epsprim > 0)$-dual-regularized primal optimization oracle,} $\DRPO(q, \lambda, \epsprim)$, returns an expected $\epsprim$-minimizer of $f_{\lambda, q}$, i.e., a point $x \in \xset$ such that $\E f_{\lambda, q}(x) \le \inf_{x' \in \xset} f_{\lambda, q}(x') + \epsprim$, where the expectation is over the internal randomness of the oracle.
\end{definition}

\begin{definition}[$\DRBR$ oracle]
    \label{def:DRBR}
    A \emph{$(q \in \uset \cap \interior \pset, \lambda > 0, x \in \xset)$-dual-regularized best response oracle,} $\DRBR(q, \lambda, x)$, returns $\argmax_{y \in \yset} \inbraces{\psi(x, y) - \lambda \breg{}{q}{y}}$.
\end{definition}

We also define a version of the $\DRPO$ oracle, called the $\DRPOSP$ oracle, which allows for a failure probability. We include this definition here due to its generality and broad applicability, but it is only used in Section~\ref{subsec:matrix-games} since the external result we cite to obtain an expected $\epsprim$-minimizer of $f_{\lambda, q}$ in that application has a failure probability. We also show in Appendix \ref{subapp:boosting-DRPOSP} how to boost the success probability of a $\DRPOSP$ oracle.

\begin{definition}[$\DRPOSP$ oracle]
    \label{def:DRPOSP}
A \emph{$(q \in \uset \cap \interior \pset, \lambda > 0, \epsprim > 0, \delta \in [0, 1))$-dual-regularized primal optimization oracle with success probability,} $\DRPOSP(q, \lambda, \epsprim, \delta)$, returns an expected $\epsprim$-minimizer of $f_{\lambda, q}$ with success probability at least $1 - \delta$, where the expectation and success probability are over the internal randomness of the oracle.
\end{definition}

\subsection{The framework and its guarantees}\label{subsec:main-framework}

\begin{algorithm}[h] %
	\setstretch{1.1}
	\caption{Dual-extraction framework}
	\label{alg:dual-extraction-framework}
	\LinesNumbered
	\DontPrintSemicolon
	\Input{
		$(\psi, \xset, \yset, \uset, \pset, \norm{\cdot}, r)$-dual extraction setup (Definition~\ref{def:dual-extraction-setup}), initial dual-regularization center $y_0 \in \yset \cap \interior \pset$, iteration count $K \in \N$, dual-regularization schedule $(\lambda_k \in \R_{> 0})_{k = 0}^{K - 1}$, primal-accuracy schedule $(\epsp_k \in \R_{> 0})_{k = 1}^K$, $\DRPO$ and $\DRBR$ oracles (Definitions \ref{def:DRPO} and \ref{def:DRBR})
	}

    \For{$k = 1, 2, \ldots, K$ }{ 
        $\Lambda_k = \sum_{j = 0}^{k - 1} \lambda_j$\;
        $q_k = \argmin_{q \in \uset}\frac{1}{\Lambda_k} \sum_{j = 0}^{k - 1} \lambda_j \breg{}{y_j}{q}$ \tcp*{Or, $q_k = \grad \rconjuset \inparen{\frac{1}{\Lambda_k} \sum_{j = 0}^{k - 1} \lambda_j \grad r(y_j)}$; see Appendix \ref{subapp:equiv-expression-q_k}} \label{algline:q_k}
        $x_k = \DRPO(q_k, \Lambda_k, \epsp_k)$ \tcp*{$\E [f_{\Lambda_k, q_k}(x_k) \mid x_{k - 1}] \le \inf_{x \in \xset} f_{\Lambda_k, q_k}(x) + \epsilon_k$} \label{algline:x_k}
        $y_k = \DRBR(q_k, \Lambda_k, x_k)$ \tcp*{$y_k = \underset{y \in \yset}{\argmax} \inbraces{\psi(x_k, y) - \Lambda_k \breg{}{q_k}{y}}$} \label{algline:y_k}
    }

    \Return $y_K$\;
\end{algorithm}

We now state the general dual-extraction framework, Algorithm~\ref{alg:dual-extraction-framework}, and its guarantees, with proofs in the next section. As mentioned in Section \ref{subsec:dual-extraction-overview}, Algorithm~\ref{alg:dual-extraction-framework} generalizes Algorithm~\ref{alg:dual-extraction-framework-sketch} in three major ways: (i) we allow for stochasticity in the $\DRPO$ oracle; (ii) we allow for distance-generating functions $r$ where $\dom r \ne \R^n$; and (iii) we give different but equivalent characterizations of $x_k$ and $y_k$ which often allow for easier instantiations of the framework.

Regarding (iii), consider the case where the $\DRPO$ oracle is deterministic and $\dom r = \R^n$ for the sake of discussion. Note that in this case, the definitions of $x_k$ and $y_k$ in Lines \ref{algline:x_k} and \ref{algline:y_k} of Algorithm~\ref{alg:dual-extraction-framework} may seem different than those in Lines \ref{algline:x_k-informal} and \ref{algline:y_k-informal} of Algorithm~\ref{alg:dual-extraction-framework-sketch} at first glance. In particular, $x_k$ in Line~\ref{algline:x_k} of Algorithm~\ref{alg:dual-extraction-framework} is an $\epsilon_k$-minimizer of $x \mapsto \max_{y \in \yset} \inbracess{\psi(x, y) - \Lambda_k \breg{}{q_k}{y}}$ over $\xset$, whereas $x_k$ in Line~\ref{algline:x_k-informal} of Algorithm~\ref{alg:dual-extraction-framework-sketch} is an $\epsilon_k$-minimizer of $x \mapsto \max_{y \in \yset} \inbracess{\psi(x, y) - \sum_{j = 0}^{k - 1} \lambda_j \breg{}{y_j}{y}}$ over $\xset$. Similarly, $y_k = \argmax_{y \in \yset} \inbracess{\psi(x_k, y) - \Lambda_k \breg{}{q_k}{y}}$ in Line~\ref{algline:y_k} of Algorithm~\ref{alg:dual-extraction-framework}, whereas $y_k = \argmax_{y \in \yset} \inbracess{\psi(x, y) - \sum_{j = 0}^{k - 1} \lambda_j \breg{}{y_j}{y}}$ in Line~\ref{algline:y_k-informal} of Algorithm~\ref{alg:dual-extraction-framework-sketch}. In fact, we show in Section \ref{subsec:main-framework-proofs} that these are equivalent; see Lemma~\ref{lem:connecting-to-overview} and Remark \ref{rem:technical-dual-assumption-and-uset}. The potential advantage of the expressions in Algorithm~\ref{alg:dual-extraction-framework} compared to those in Algorithm~\ref{alg:dual-extraction-framework-sketch} is that they involve only a single regularization term.

Note also that Line~\ref{algline:q_k} of Algorithm~\ref{alg:dual-extraction-framework} gives two equivalent expressions for the iterate $q_k$; their equivalence is proven in Appendix \ref{subapp:equiv-expression-q_k}. Also, note that Line~\ref{algline:x_k} is the only potential source of randomness in Algorithm~\ref{alg:dual-extraction-framework}; in particular, $y_k$ and $q_{k + 1}$ are deterministic upon conditioning on $x_k$. Finally, we show that Algorithm~\ref{alg:dual-extraction-framework} is well-defined in Appendix \ref{subapp:supporting-lemmas-well-defined}; in particular, whenever a Bregman divergence $\breg{}{u}{w}$ is written in Algorithm~\ref{alg:dual-extraction-framework}, it is the case that $u \in \uset \cap \interior \pset$. For example, in the context of a simplex setup per Definition~\ref{ex:simplex}, this corresponds to $u \in \simplex^n_{>0}$.

We now give the main guarantee for Algorithm~\ref{alg:dual-extraction-framework}. See Remark \ref{rem:picking-parameters-main-thm} for additional explanation.

\begin{restatable}[Algorithm~\ref{alg:dual-extraction-framework} guarantee]{theorem}{restateThmMainFrameworkThm}
    \label{thm:main-framework-guarantee}
    With $K$ calls to a $\DRPO$ oracle and $K$ calls to a $\DRBR$ oracle,
    Algorithm~\ref{alg:dual-extraction-framework} returns $y_K$ satisfying
    \begin{align*}
        \E \breg{}{y_K}{u} \le \frac{\epsilon_K}{\Lambda_K},
    \end{align*}
    where $u \in \yset$ is a point with expected suboptimality bounded as
    \begin{align*}
        \phi(\yopt) - \E \phi(u) \le \lambda_0 \breg{}{y_0}{\yopt} + \sum_{k = 1}^{K - 1} \frac{\lambda_k}{\Lambda_k} \epsp_k.
    \end{align*}
    If we additionally assume that $\phi$ is $L$-Lipschitz with respect to $\norm{\cdot}$, the expected suboptimality of $y_K$ can be directly bounded as
    \begin{align}
        \label{eq:Lipschitz-main-guarantee}
        \phi(\yopt) - \E \phi(y_K) \le  \lambda_0 \breg{}{y_0}{\yopt} + \sum_{k = 1}^{K - 1} \frac{\lambda_k}{\Lambda_k} \epsp_k + L \sqrt{\frac{2}{\mu_r} \frac{\epsp_K}{\Lambda_K}}.
    \end{align}
\end{restatable}

We now particularize Theorem~\ref{thm:main-framework-guarantee} using two exemplary choices of the dual-regularization and primal-accuracy schedules. See Remarks \ref{rem:example-schedules-wlog} and \ref{rem:example-schedules-no-log} for additional comments.

\begin{restatable}{corollary}{restateCorLogCost} \label{cor:example-schedules-wlog}
Suppose $\phi$ is $L$-Lipschitz with respect to $\norm{\cdot}$, and let $B > 0$ be such that $\breg{}{y_0}{\yopt} \le B$. Then for any $\epsilon > 0$, and $K \ge \max \inbraces{\log_2 \frac{L^2 B}{\mu_r \epsilon^2}, 1} + 10$, the output of Algorithm~\ref{alg:dual-extraction-framework} with dual-regularization and primal-accuracy schedules given by
\begin{align*}
    \lambda_k = 2^k \frac{\epsilon}{4 B} \text{ for $k \in \inbraces{0} \cup [K - 1]$}
    \text{ and }
    \epsp_k = \frac{\epsilon}{4 K} \text{ for $k \in [K]$}
\end{align*}
satisfies $\phi(\yopt) - \E \phi(y_K) \le \epsilon$.
\end{restatable}
    
\begin{restatable}{corollary}{restateCorNoLogCost}
    \label{cor:example-schedules-no-log}
    Let $B > 0$ be such that $\breg{}{y_0}{\yopt} \le B$. Then for any $\epsilon > 0$ and $K \in \N$, the output of Algorithm~\ref{alg:dual-extraction-framework} with dual-regularization and primal-accuracy schedules given by
    \begin{align}
        \label{eq:schedules-no-log}
        \lambda_k = 2^k \frac{\epsilon}{4 B} \text{ for $k \in \inbraces{0} \cup [K - 1]$}
        \text{ and }
        \epsp_k = \frac{\epsilon}{8 \cdot 1.5^k} \text{ for $k \in [K]$}
    \end{align}
    satisfies
    \begin{align*}
        \E \norm{y_K - u} \le  \frac{1}{1.5^K} \sqrt{\frac{2 B}{\mu_r}},
    \end{align*}
    where $u \in \yset$ is a point whose expected suboptimality is at most $\epsilon$, i.e., $\phi(\yopt) - \E \phi(u) \le \epsilon$. 
\end{restatable}

\subsection{Proofs of the guarantees}\label{subsec:main-framework-proofs}

We now prove Theorem~\ref{thm:main-framework-guarantee} and Corollaries \ref{cor:example-schedules-wlog} and \ref{cor:example-schedules-no-log}. With the sole exception of Lemma~\ref{lem:dual-div-bound-exact-best-response} (which is stated independently of Algorithm~\ref{alg:dual-extraction-framework} so that it is easier to cite), all of the corollaries, lemmas, and theorems in this section are stated in the context of Algorithm~\ref{alg:dual-extraction-framework}; in particular, we have fixed a $(\psi, \xset, \yset, \uset, \pset, \norm{\cdot}, r)$-dual extraction setup per Definition~\ref{def:dual-extraction-setup} which has been given as input to Algorithm~\ref{alg:dual-extraction-framework}, and we will use the additional notation associated with this setup introduced in Section \ref{subsec:main-framework-assumptions-preliminaries} (as well as the other notation of Algorithm~\ref{alg:dual-extraction-framework} itself).

Also, for brevity, we may condition on $x_0$ in this section to avoid having to distinguish between the $k = 1$ and $k > 1$ cases when indexing (for example, when writing $\E[Z \mid x_{k - 1}]$ for some random quantity $Z$), even though $x_0$ is not defined in Algorithm~\ref{alg:dual-extraction-framework}. Full formality can be restored by, e.g., arbitrarily fixing a value for $x_0$.

First, let us compare Algorithm~\ref{alg:dual-extraction-framework} to Algorithm~\ref{alg:dual-extraction-framework-sketch}. For the rest of this section, define $\psi_k$, $f_k$, $\phi_k$, $\xopt_k$, and $\yopt_k$ for $k \in \inbraces{0} \cup [K]$ as in \eqref{eq:recursive-reg-functions} and the surrounding text. Then Lemma~\ref{lem:connecting-to-overview} below says that the sequence of iterates $(x_k, y_k)$ generated by Algorithm~\ref{alg:dual-extraction-framework-sketch} is equivalent to that generated by Algorithm~\ref{alg:dual-extraction-framework} (up to stochasticity). In particular, Algorithm~\ref{alg:dual-extraction-framework} combines the sequence of dual-regularization terms which appear in $\psi_k(x, y) = \psi(x, y) - \sum_{j = 0}^{k - 1} \lambda_j \breg{}{y_j}{y}$ into a single dual-regularization term centered at $q_k$, thereby yielding a form appropriate for the $\DRPO$ and $\DRBR$ oracles as specified in Definitions \ref{def:DRPO} and \ref{def:DRBR}.

\begin{lemma}[Connecting Algorithm~\ref{alg:dual-extraction-framework} to Algorithm~\ref{alg:dual-extraction-framework-sketch}]
    \label{lem:connecting-to-overview}
    For all $k \in [K]$, we have that $x_k$ given by Line~\ref{algline:x_k} of Algorithm~\ref{alg:dual-extraction-framework} is an expected $\epsilon_k$-minimizer of $f_k$ conditioned on $x_{k - 1}$ (formally, $\E \insquare{ f_k(x_k) \mid x_{k - 1}} \le f_k(x) + \epsilon_k$ for all $x \in \xset$), and $y_k$ given by Line~\ref{algline:y_k} is the dual best response to $x_k$ with respect to $\psi_k$, i.e., $y_k = \argmax_{y \in \yset} \psi_k(x_k, y)$. 
\end{lemma}

\begin{proof}
    To prove this claim, it suffices to show that for any $k \in [K]$, we have that over all $(x, y) \in \xset \times \yset$, it is the case that $\psi_k(x, y) = \psi(x, y) - \sum_{j = 0}^{k - 1} \lambda_j \breg{}{y_j}{y}$ and $\psi(x, y) - \Lambda_k \breg{}{q_k}{y}$ only differ by a quantity with no dependence on $x$ or $y$. In turn, it suffices to show that over $y \in \yset$, we have that $\sum_{j = 0}^{k - 1} \lambda_j \breg{}{y_j}{y}$ and $\Lambda_k \breg{}{q_k}{y}$ differ by a quantity with no dependence on $y$.
    Indeed, we claim
    \begin{align*}
        \sum_{j = 0}^{k - 1} \lambda_j \breg{}{y_j}{y} &= \Lambda_k \inparen{r(y) - \inangle{ \frac{1}{\Lambda_k} \sum_{j = 0}^{k - 1} \lambda_j \grad r(y_j), y}} + C \\
        &= \Lambda_k \breg{}{q_k}{y} + C',
    \end{align*}
    where $C, C'$ are quantities which don't depend on $y$. The first equality is straightforward algebra, and the second equality follows because we can equivalently express Line~\ref{algline:q_k} in Algorithm~\ref{alg:dual-extraction-framework} as
    \begin{align*}
        q_k = \argmin_{q \in \R^n} \inbraces{ \frac{1}{\Lambda_k} \sum_{j = 0}^{k - 1} \lambda_j \breg{}{y_j}{q}  + \indc_\uset(q)},
    \end{align*}
    in which case Lemma~\ref{lem:subdiff-optimality-cond} followed by Lemma~\ref{lem:subdiff-sum-of-funcs} imply
    \begin{align}
        \label{eq:q_k-opt-condition}
        0 \in \partial \inparen{ \frac{1}{\Lambda_k} \sum_{j = 0}^{k - 1} \lambda_j \breg{}{y_j}{\cdot}  + \indc_\uset}(q_k) = \grad r(q_k) - \frac{1}{\Lambda_k} \sum_{j = 0}^{k - 1} \lambda_j \grad r(y_j) + \partial \indc_{\uset}(q_k).
    \end{align}
    (Applying Lemma~\ref{lem:subdiff-sum-of-funcs} with $f_m \gets \indc_\uset$ reduces the condition to $\uset \cap \interior \pset = \emptyset$, which is true by assumption. We also used the fact that $q_k \in \uset \cap \interior \pset$ due to Corollary~\ref{cor:main-framework-well-defined}, so that $r$ is differentiable at $q_k$.) Then rearranging, we have
    \begin{align*}
        \grad r(q_k) = \frac{1}{\Lambda_k} \sum_{j = 0}^{k - 1} \lambda_j \grad r(y_j) - w
    \end{align*}
    for some $w \in \partial \indc_{\uset}(q_k)$. We conclude by applying Assumption~\ref{as:technical-dual} in Definition~\ref{def:dual-extraction-setup}.
\end{proof}

Before continuing, we give a remark regarding Lemma~\ref{lem:connecting-to-overview} and Assumption~\ref{as:technical-dual} in Definition~\ref{def:dual-extraction-setup}:

\begin{remark}[Regarding Assumption~\ref{as:technical-dual} and $\uset$ in Definition~\ref{def:dual-extraction-setup}]
    \label{rem:technical-dual-assumption-and-uset}
    The (end of the) proof of Lemma~\ref{lem:connecting-to-overview} is the only place where we use Assumption~\ref{as:technical-dual} in Definition~\ref{def:dual-extraction-setup}. In particular, it is due to Lemma~\ref{lem:connecting-to-overview} that the dgf is defined over an ``intermediate set'' $\uset$ in Definition~\ref{def:dual-extraction-setup} as opposed to being defined directly on $\yset$. In applications where there is no gain to be had by combining all of the regularization terms in $\psi_k(x, y) = \psi(x, y) - \sum_{j = 0}^{k - 1} \lambda_j \breg{}{y_j}{y}$ into a single regularization term, Assumption~\ref{as:technical-dual} and the presence of $\uset$ can be removed from Definition~\ref{def:dual-extraction-setup} (the dgf can be defined directly on $\yset$), and all of the results of Section \ref{subsec:main-framework} still hold if Algorithm~\ref{alg:dual-extraction-framework} is altered so that $x_k$ and $y_k$ are defined analogously to their definitions in Algorithm~\ref{alg:dual-extraction-framework-sketch}. Indeed, having connected the iterates of Algorithm~\ref{alg:dual-extraction-framework} to those of Algorithm~\ref{alg:dual-extraction-framework-sketch} (up to stochasticity) via Lemma~\ref{lem:connecting-to-overview}, the rest of the proofs in this section only use the fact that $x_k$ is an expected $\epsilon_k$-minimizer of $f_k$ conditioned on $x_{k - 1}$, and $y_k$ is the dual best response to $x_k$ with respect to $\psi_k$.
\end{remark}

We now give the general version of Lemma~\ref{lem:bound-dual-div-primal-subopt-informal} from Section \ref{subsec:dual-extraction-overview}. Recall that Lemma~\ref{lem:dual-div-bound-exact-best-response} is the one result in Section \ref{subsec:main-framework} that is not stated in the context of Algorithm~\ref{alg:dual-extraction-framework} (for ease of citation), although we still use the notation associated with a dual extraction setup per Section \ref{subsec:main-framework-assumptions-preliminaries}.

\begin{restatable}[Bounding dual divergence with primal suboptimality]{lemma}{restateLemDualDivBoundBestResponse}
    \label{lem:dual-div-bound-exact-best-response}
    Let $(\psi, \xset, \yset, \uset, \pset, \norm{\cdot}, r)$ be any dual-extraction setup (Definition~\ref{def:dual-extraction-setup}), and suppose for some $x \in \xset$ that $- \psi(x, \cdot)$ is $\mu$-strongly convex over $\yset \cap \interior \pset$ relative to the dgf $r$ for some $\mu > 0$. Then letting $y_x \defeq \argmax_{y \in \yset} \psi(x, y)$, we have
    \begin{align*}
        \breg{}{y_x}{\yopt} \le \frac{f(x) - f(\xopt)}{\mu}.
    \end{align*}
\end{restatable}

\begin{proof}
    Note that $y_x \in \yset \cap \interior \pset$ by Lemma~\ref{lem:min-in-int} with $\sset \gets \yset$.
    Then we have
    \begin{align*}
        f\left(x\right)-f\left(x^{\star}\right) = \psi\left(x,y_{x}\right)-\psi\left(x^{\star},y_{x^{\star}}\right) 
        &\overeq{(i)} \psi\left(x,y_{x}\right)-\psi\left(x^{\star}, \yopt \right)                              
        \overge{(ii)} \psi\left(x,y_{x}\right)-\psi\left(x, \yopt \right)\\
        &\overge{(iii)}   \psi\left(x,y_{x}\right) + \inner{\grad_y \psi (x, y_x)}{ \yopt - y_x} -\psi\left(x, \yopt \right)           \\
         & \overge{(iv)} \mu \breg{}{y_x}{ \yopt}.
    \end{align*}
    Here 
    $(i)$ follows because $y_{\xopt} = \yopt$ by the minimax principle (note that $y_{\xopt}$ and $\yopt$ are unique by strong concavity), $(ii)$ uses the fact that $\xopt$ minimizes $\psi(\cdot, \yopt)$ over $\xset$ (again by the minimax principle), $(iii)$ uses the following standard optimality condition (cf. Lemma~\ref{lem:first-order-optimality}), evaluated at $\yopt$, for the fact that $y_x$ maximizes $\psi(x, \cdot)$ over $\yset$:
    \begin{align*}
        \inner{\grad_y \psi (x, y_x)}{y - y_x} \le 0 \text{ for all $y \in \yset$},
    \end{align*}
    and $(iv)$ follows from the relative strong convexity assumption.
\end{proof}

Next, we bound the the expected divergence between $y_k$ and $\yopt_k$:

\begin{lemma}[Divergence bound between $y_k$ and $\yopt_k$]
    \label{lem:divergence-bound-y_k-yopt_k}
    For $y_k$ as defined in Line~\ref{algline:y_k} in Algorithm~\ref{alg:dual-extraction-framework}, we have $\E \insquare{ \breg{}{y_k}{\yopt_k} \mid x_{k - 1} } \le \frac{\epsp_k}{\Lambda_k}$ for $k \in [K]$.
\end{lemma}

\begin{proof}
    Note that $(\psi_k, \xset, \yset, \uset, \pset, \norm{\cdot}, r)$ is a valid dual-extraction setup (Definition~\ref{def:dual-extraction-setup}), and $- \psi_k(x, \cdot)$ is $\Lambda_k$-strongly convex relative to $r$ over $\yset \cap \interior \pset$ for all $x \in \xset$. Applying Lemma~\ref{lem:dual-div-bound-exact-best-response} to the dual-extraction setup $(\psi_k, \xset, \yset, \uset, \pset, \norm{\cdot}, r)$ and taking expectations yields
\begin{align*}
    \E \insquare{ \breg{}{y_k}{\yopt_k} \mid x_{k - 1}} \le \frac{\E \insquare{f_k(x_k) - f_k(\xopt_k) \mid x_{k - 1}}}{\Lambda_k} \le \frac{\epsp_k}{\Lambda_k}.
\end{align*}
\end{proof}

We now bound the expected suboptimality of $\yopt_K$ for \eqref{eq:dual}:

\begin{lemma}[Suboptimality bound for $\yopt_K$]
    \label{lem:yopt_K-performs-well-on-phi}
We have 
\begin{align*}
    \phi(\yopt) - \E \phi(\yopt_K) \le  \lambda_0 \breg{}{y_0}{\yopt} + \sum_{k = 1}^{K - 1} \frac{\lambda_k}{\Lambda_k} \epsp_k.
\end{align*}
\end{lemma}

\begin{proof}
Note that for any $k \in [K]$, we can write 
\begin{align*}
    \phi_k(y) = \min_{x \in \xset} \inbraces{\psi_{k - 1}(x, y) - \lambda_{k - 1} \breg{}{y_{k - 1}}{y}} = \phi_{k - 1}(y) - \lambda_{k - 1} \breg{}{y_{k - 1}}{y}.
\end{align*}
As a result, the following holds for all $k \in [K]$:
\[
\phi_k(\yopt_k) \ge \phi_{k}(\yopt_{k - 1})
=  \phi_{k - 1}(\yopt_{k - 1}) - \lambda_{k - 1} \breg{}{y_{k - 1}}{\yopt_{k - 1}}\,.
\]
Repeatedly applying this inequality, taking expectations, and applying Lemma~\ref{lem:divergence-bound-y_k-yopt_k} yields
\begin{align*}
    \phi_K(\yopt_K)
&\geq  \phi_0 (\yopt_0) - \sum_{k = 0}^{K - 1} \lambda_k \breg{}{y_k}{\yopt_k}, \\
\implies \E \phi_K(\yopt_K) &\ge \phi(\yopt) - \lambda_0 \breg{}{y_0}{\yopt} - \sum_{k = 1}^{K - 1} \E \insquare{ \E \insquare{ \lambda_k \breg{}{y_k}{\yopt_k} \mid x_{k - 1}} } \\
& \ge \phi(\yopt) - \lambda_0 \breg{}{y_0}{\yopt} - \sum_{k = 1}^{K - 1} \frac{\lambda_k}{\Lambda_k} \epsp_k.
\end{align*}
To conclude, note that $\E \phi(\yopt_K) \ge \E \phi_k(\yopt_K)$ since $\phi \ge \phi_K$ pointwise.
\end{proof}

We now prove the main result, restated here for convenience:

\restateThmMainFrameworkThm*

\begin{proof}
    The first part is immediate from Lemma~\ref{lem:divergence-bound-y_k-yopt_k} and Lemma~\ref{lem:yopt_K-performs-well-on-phi} with $u \gets \yopt_K$. As for the second part, using the Lipschitzness of $\phi$ and strong convexity of $r$, we have
    \begin{align*}
        \phi(y_K) \ge \phi(\yopt_K) - L \norm{y_K - \yopt_K} 
                    \ge \phi(\yopt_K) - L \sqrt{\frac{2}{\mu_r} \breg{}{y_K}{\yopt_K}},
    \end{align*}
    so taking expectations, applying Jensen's inequality for concave functions, and an application of Lemma~\ref{lem:divergence-bound-y_k-yopt_k} yields
    \begin{align*}
        \E \phi(y_K) \ge \E \phi(\yopt_K) - L \sqrt{\frac{2}{\mu_r} \E \insquare{\E \insquare{\breg{}{y_K}{\yopt_K} \mid x_{K - 1}}}} \ge 
        \E \phi(\yopt_K)
        - L \sqrt{\frac{2}{\mu_r} \frac{\epsilon_K}{\Lambda_K}}.
    \end{align*}
    Finally, combining this with Lemma~\ref{lem:yopt_K-performs-well-on-phi} yields the result.
\end{proof}

Finally, we prove the resulting corollaries, restated here for convenience:

\restateCorLogCost*

\begin{proof}
    Note that $\Lambda_k = \frac{\epsilon}{4B} \sum_{j = 0}^{k - 1} 2^j = \frac{\epsilon}{4B} \inparen{2^k - 1}$, implying $\lambda_k / \Lambda_k = 2^k / (2^k - 1) \le 2$ for $k \ge 1$. Then Equation~\ref{eq:Lipschitz-main-guarantee} becomes
    \begin{align*}
        \phi(\yopt) - \E \phi(y_K) \le  \frac{\epsilon}{4} + \sum_{k = 1}^{K - 1} 2 \cdot \frac{\epsilon}{4 K } + L \sqrt{\frac{2}{\mu_r} \frac{\epsilon}{4 K } \frac{4 B}{\epsilon \cdot 2^{K - 1}}} \le \epsilon.
    \end{align*}
\end{proof}

\restateCorNoLogCost*

\begin{proof}
    As in the proof of Corollary~\ref{cor:example-schedules-wlog}, we have $\Lambda_k = \frac{\epsilon}{4 B} (2^k - 1)$. Both claims then follow directly from the first part of Theorem~\ref{thm:main-framework-guarantee}, which posits the existence of $u \in \yset$ such that 
\begin{align*}
    \E \breg{}{y_K}{u} \le \frac{\epsilon}{8 \cdot 1.5^K} \cdot \frac{4B}{\epsilon \cdot (2^{K - 1})} = \frac{B}{3^K}.
\end{align*}
Then, the fact that $r$ is $\mu_r$-strongly convex and Jensen's inequality for concave functions implies
\begin{align*}
    \norm{y_K - u} \le \sqrt{\frac{2}{\mu_r} \breg{}{y_K}{u}} 
    \implies \E \norm{y_K - u} \le \frac{1}{1.5^K} \sqrt{\frac{2 B}{\mu_r}}.
\end{align*}
Furthermore, since $\lambda_k / \Lambda_k \le 2$ for $k \ge 1$, we have
\begin{align*}
    \phi(\yopt) - \E \phi(u) \le \frac{\epsilon}{4} + \frac{\epsilon}{4} \sum_{k = 1}^{K - 1} \frac{1}{1.5^k} \le \epsilon.
\end{align*}
\end{proof}

\section{Efficient maximin algorithms}\label{sec:maximin-algorithms}

In this section, we obtain new state-of-the-art runtimes for solving bilinear matrix games in certain parameter regimes (Section \ref{subsec:matrix-games}), as well as an improved query complexity for solving the dual of the CVaR at level $\alpha$ distributionally robust optimization (DRO) problem (Section \ref{subsec:alpha-CVaR}).
In each application, we apply Corollary~\ref{cor:example-schedules-wlog} to compute an $\epsilon$-optimal point for the dual problem at approximately the same cost as computing an $\epsilon$-optimal point for the primal problem (up to logarithmic factors and the cost of representing a dual vector when it comes to CVaR at level $\alpha$).

\subsection{Bilinear matrix games}\label{subsec:matrix-games}

In this section, we instantiate $\psi(x, y) \defeq x^\top A y$ for a matrix $A \in \R^{d \times n}$. Given $p, q \ge 1$, we write $\norm{A}_{p \to q} \defeq \max_{v \in \R^d, v \ne 0} \frac{\norm{Av}_q}{\norm{v}_p}$, and use the notation $A_{ij}$, $A_{i:}$, and $A_{:j}$ for the $(i, j)$ entry, $i$-th row as a row vector, and $j$-th column as a column vector.
We consider two setups: 

\begin{definition}[Matrix games ball setup]
    \label{def:mat-games-ball-in-body}
In the \emph{matrix games ball setup,} we set $\xset \defeq B^d$ (the unit Euclidean ball in $\R^d$), $\yset \defeq \simplex^n$, and fix a $(\psi, \xset, \yset)$-simplex setup (Definition~\ref{ex:simplex}).
We assume $\norm{A^\top}_{2 \to \infty} = \max_{i \in [n]} \norm{A_{:i}}_2 \le 1$.
\end{definition}

\begin{definition}[Matrix games simplex setup]
    \label{def:mat-games-simplex-in-body}
In the \emph{matrix games simplex setup,} we set $\xset \defeq \simplex^d$, $\yset \defeq \simplex^n$, and fix a $(\psi, \xset, \yset)$-simplex setup (Definition~\ref{ex:simplex}). 
We assume $\norm{A^\top}_{1 \to \infty} = \max_{i, j} |A_{ij}| \le 1$.
\end{definition}

Throughout Section \ref{subsec:matrix-games}, any theorem, statement, or equation which does not make reference to a specific choice of Definition~\ref{def:mat-games-ball-in-body} or \ref{def:mat-games-simplex-in-body} applies to both setups.
Specializing the primal \eqref{eq:primal} and dual \eqref{eq:dual} to this application gives
\leqnomode
\begin{align}
    \tag{P-MG} \label{eq:primal-MG} &\minimize_{x \in \xset} f(x) \text{ for } f(x) \defeq \max_{y \in \simplex^n} x^\top A y, \text{ and} \\
    \tag{D-MG} \label{eq:dual-MG} &\maximize_{y \in \simplex^n} \phi(y)\text{ for }\phi(y) \defeq \min_{x \in \xset} x^\top A y .
\end{align}
\reqnomode

Regarding the assumptions on the norm of the matrix $A$ in Definitions \ref{def:mat-games-ball-in-body} and \ref{def:mat-games-simplex-in-body}, note that we can equivalently write $f(x) = \max_{y \in \simplex^n} \sum_{i = 1}^n y_i f_i(x)$ with $f_i(x) \defeq [A^\top x]_i$. Then the assumptions on the norm of $A$ correspond to ensuring $f_i$ is 1-Lipschitz with respect to the $\ell_2$-norm in Definition~\ref{def:mat-games-ball-in-body} and $\ell_1$-norm in Definition~\ref{def:mat-games-simplex-in-body} (which in turn implies $f$ is 1-Lipschitz in the respective norms). This normalization is performed to simplify expressions as in \cite{carmon2023whole}. (In particular, \cite{carmon2023whole} also considers the more general problem where each $f_i$ can be any smooth, Lipschitz, convex function.)

Recently, \cite[Cor. 8.2]{carmon2023whole} achieved a state-of-the-art runtime in certain parameter regimes of $\Otilde(nd + n (d / \epsilon)^{2/3} + d \epsilon^{-2})$ for obtaining an $\epsilon$-optimal point for \eqref{eq:primal-MG}. However, unlike previous algorithms for  \eqref{eq:primal-MG} (see Section \ref{subsec:related-work} for an extended discussion), their algorithm does not yield an $\epsilon$-optimal point for \eqref{eq:dual-MG} with the same runtime.

\begin{algorithm}[h] %
	\setstretch{1.1}
	\caption{Dual extraction for matrix games}
	\label{alg:dual-extraction-mat-games}
	\LinesNumbered
	\DontPrintSemicolon
	\Input{
		$(\psi, \xset, \simplex^n)$-simplex setup (Definition~\ref{ex:simplex}), iteration count $K \in \N$, dual-regularization schedule $(\lambda_k \in \R_{> 0})_{k = 0}^{K - 1}$, primal-accuracy schedule $(\epsp_k \in \R_{> 0})_{k = 1}^K$, $\DRPOSP$ oracle (Definition~\ref{def:DRPOSP})
	}

    $y_0 \defeq \frac{1}{n} \ones$\;

    \For{$k = 1, 2, \ldots, K$ }{ 
        $\Lambda_k = \sum_{j = 0}^{k - 1} \lambda_j$\;
        $[q_k]_i \propto \prod_{j = 0}^{k - 1} [y_j]_i^{\lambda_j / \Lambda_k}, ~ \forall i \in [n]$ \tcp*{Note: $q_k \in \simplex^n$} \label{algline:q_k-mat-games}
        $x_k = \DRPOSP(q_k, \Lambda_k, \epsp_k, \frac{1}{10 K})$\; \label{algline:x_k-mat-games}
        $[y_k]_i \propto [q_k]_i \cdot \exp (\Lambda_k^{-1} \cdot [A^\top x_k]_i), ~ \forall i \in [n]$ \tcp*{$y_k = \underset{y \in \simplex^n}{\argmax} \inbraces{x_k^\top A y - \Lambda_k \breg{}{q_k}{y}}$} \label{algline:y_k-mat-games}
    }

    \Return $y_K$\;
\end{algorithm}

Our instantiation of the dual-extraction framework in Algorithm~\ref{alg:dual-extraction-mat-games} and the accompanying guarantee Theorem~\ref{thm:guarantee-mat-games} resolves this asymmetry between the complexity of obtaining a primal versus dual $\epsilon$-optimal point by obtaining an $\epsilon$-optimal point of \eqref{eq:dual-MG} with the same runtime of $\Otilde(nd + n (d / \epsilon)^{2/3} + d \epsilon^{-2})$. At the end of Section~\ref{subsec:matrix-games}, we observe that Theorem~\ref{thm:guarantee-mat-games} also yields a new state-of-the-art runtime for the primal \eqref{eq:primal-MG} in the setting of Definition~\ref{def:mat-games-simplex-in-body} due to the symmetry of the constraint sets and $\psi$.

Before giving the guarantee Theorem~\ref{thm:guarantee-mat-games} for Algorithm~\ref{alg:dual-extraction-mat-games}, the following lemma provides a runtime bound for the $\DRPOSP$ oracle when the success probability is $9/10$ (see Appendix \ref{subapp:deferred-matrix-games-proofs} for the proof). In particular, Lemma~\ref{lem:DRPOSP-matrix-games} shows that adding dual regularization to \eqref{eq:primal-MG} does not increase the complexity of obtaining an $\epsilon$-optimal point over the guarantee of \cite[Cor. 8.2]{carmon2023whole} discussed above.

\begin{restatable}[$\DRPOSP$ oracle for matrix games]{lemma}{restateLemDRPOSPMatrixGames}
    \label{lem:DRPOSP-matrix-games}
    In the settings of Definitions \ref{def:mat-games-ball-in-body} and \ref{def:mat-games-simplex-in-body},
    for any $q \in \simplex^n_{>0}$, $\epsprim > 0$, and $\lambda > 0$, with success probability at least $9/10$, there exists an algorithm which returns an expected $\epsprim$-optimal point of $f_{\lambda, q}$ with runtime $\Otilde(nd + n (d / \epsprim)^{2/3} + d \epsprim^{-2})$. (Equivalently, per Definition~\ref{def:DRPOSP}, we have that $\DRPOSP(q, \lambda, \epsprim, 1/10)$ can be implemented with this runtime.)
\end{restatable}

Now for the main guarantee:

\begin{theorem}[Guarantee for Algorithm~\ref{alg:dual-extraction-mat-games}]
    \label{thm:guarantee-mat-games}
    In the settings of Definitions \ref{def:mat-games-ball-in-body} and \ref{def:mat-games-simplex-in-body}, given target error $\epsilon > 0$ and with success probability at least $9/10$, Algorithm~\ref{alg:dual-extraction-mat-games} with dual-regularization and primal-accuracy schedules given by
    \begin{align*}
        \lambda_k = 2^k \frac{\epsilon}{4 \ln n} \text{ for $k \in \inbraces{0} \cup [K - 1]$}
        \text{ and }
        \epsp_k = \frac{\epsilon}{4 K} \text{ for $k \in [K]$}
    \end{align*}
    for $K = \ceil*{\max \inbraces{\log_2 \frac{\ln n}{\epsilon^2}, 1}} + 10$ returns an expected $\epsilon$-optimal point for \eqref{eq:dual-MG}, and can be implemented with runtime $\Otilde(nd + n (d / \epsilon)^{2/3} + d \epsilon^{-2})$.
\end{theorem}

\begin{proof}
Line~\ref{algline:x_k-mat-games} in Algorithm~\ref{alg:dual-extraction-mat-games}, which clearly dominates the runtime of a single iteration of the for loop, can be implemented in $\Otilde(nd + n (d / \epsilon)^{2/3} + d \epsilon^{-2})$ by combining Lemma~\ref{lem:DRPOSP-matrix-games} with Lemma~\ref{lem:boost-DRPOSP}. This yields the desired final runtime since $K$ is at most logarithmic in all parameters.

As for correctness, we apply Corollary~\ref{cor:example-schedules-wlog} with $\mu_r \gets 1$, $B \gets \ln n$ (see Lemma~\ref{lem:bound-on-KL-from-uniform}), and $L \gets 1$. Regarding the latter, it is straightforward to bound the Lipschitzness of $\phi$ with respect to $\norm{\cdot}_1$ by $1$ in the setting of Definition~\ref{def:mat-games-simplex-in-body}, and we show the same is true in the setting of Definition~\ref{def:mat-games-ball-in-body} in Lemma~\ref{lem:bounding-Lipschitz-dual-matrix-games}.
Also, note that the success probability follows from a union bound over all of the $\DRPOSP$ oracle calls in Line~\ref{algline:x_k-mat-games} of Algorithm~\ref{alg:dual-extraction-mat-games}. 
All that is left to check is that Lines \ref{algline:q_k-mat-games} and \ref{algline:y_k-mat-games} in Algorithm~\ref{alg:dual-extraction-mat-games} are indeed the appropriate particularizations of Lines \ref{algline:q_k} and \ref{algline:y_k} in Algorithm~\ref{alg:dual-extraction-framework} respectively. The correspondence between Line~\ref{algline:q_k-mat-games} in Algorithm~\ref{alg:dual-extraction-mat-games} and Line~\ref{algline:q_k} in Algorithm~\ref{alg:dual-extraction-framework} (see the alternate expression for $q_k$) follows immediately from the last part of Lemma~\ref{lem:convex-conjugate-neg-entropy}. As for the correspondence between Line~\ref{algline:y_k-mat-games} in Algorithm~\ref{alg:dual-extraction-mat-games} and Line~\ref{algline:y_k} in Algorithm~\ref{alg:dual-extraction-framework}, note that we have
\begin{align*}
    y_k = \argmax_{y \in \simplex^n} \inbraces{x_k^\top A y - \Lambda_k \breg{}{q_k}{y}} = \argmax_{y \in \simplex^n} 
        \sum_{i = 1}^n  \insquare{ \inparen{
        \Lambda_k^{-1} \cdot [A^\top x_k]_i + \ln ([q_k]_i)
        } y_i - y_i \ln y_i
        },
\end{align*}
in which case we conclude by applying Lemma~\ref{lem:subdiff-properties-of-conjugate} and Lemma~\ref{lem:convex-conjugate-neg-entropy}.
\end{proof}

\paragraph{The primal perspective.} As alluded to above, the guarantee of Theorem~\ref{thm:guarantee-mat-games} also implies a new state-of-the-art runtime for the primal \eqref{eq:primal-MG} in the setting of Definition~\ref{def:mat-games-simplex-in-body}. This follows because in the matrix games simplex setup, \eqref{eq:primal-MG} and \eqref{eq:dual-MG} are symmetric in terms of their constraint sets, so we can obtain an expected $\epsilon$-optimal point for \eqref{eq:primal-MG} via Theorem~\ref{thm:guarantee-mat-games} by negating and treating \eqref{eq:primal-MG} as if it were the dual problem. Formally:

\begin{corollary}[Guarantee for {\eqref{eq:primal-MG}} in the matrix games simplex setup]
    \label{cor:primal-mat-guarantee}
In the setting of Definition~\ref{def:mat-games-simplex-in-body}, there exists an algorithm which, given target error $\epsilon > 0$ and with success probability at least $9/10$, returns an expected $\epsilon$-optimal point for \eqref{eq:primal-MG} with runtime $\Otilde(nd + d (n / \epsilon)^{2/3} + n \epsilon^{-2})$.
\end{corollary}

\begin{proof}
Consider
\leqnomode
\begin{align}
    \tag{P-MG'} \label{eq:primal-MG-prime} &\minimize_{y \in \simplex^n} \max_{x \in \simplex^d} - y^\top A^\top x, \text{ and} \\
    \tag{D-MG'} \label{eq:dual-MG-prime} &\maximize_{x \in \simplex^d} \min_{y \in \simplex^n} - y^\top A^\top x .
\end{align}
\reqnomode
Observe that \eqref{eq:dual-MG-prime} is equivalent to \eqref{eq:primal-MG} in the sense that any $\epsilon$-maximizer of \eqref{eq:dual-MG-prime} is an $\epsilon$-minimizer of \eqref{eq:primal-MG}. Then we apply Theorem~\ref{thm:guarantee-mat-games} to obtain an $\epsilon$-expected maximizer of \eqref{eq:dual-MG-prime} with success probability at least $9/10$ and runtime $\Otilde(nd + d(n / \epsilon)^{2/3} + n \epsilon^{-2})$, since the dimensions $d$ and $n$ have switched places compared to in \eqref{eq:primal-MG} and \eqref{eq:dual-MG}.
\end{proof}

See Table \ref{table:runtimes-bilinear} for a summary of how Corollary~\ref{cor:primal-mat-guarantee} fits into the current literature for obtaining an expected $\epsilon$-optimal point of \eqref{eq:primal-MG} in the setting of Definition~\ref{def:mat-games-simplex-in-body}. In particular, as discussed in \cite[Sec. 1]{carmon2023whole}, the runtime $\Otilde(nd + n (d / \epsilon)^{2/3} + d \epsilon^{-2})$ achieved by \cite[Cor. 8.2]{carmon2023whole} is the state of the art in certain subregimes under the broader restriction of $n \ge d$. (Note that when $n \le d$, the runtime $nd + n (d / \epsilon)^{2/3} + d \epsilon^{-2}$ is never better than $(n + d) \epsilon^{-2}$ achieved by \cite{grigoriadis1995sublinear}.) Thus, the runtime $\Otilde(nd + d(n / \epsilon)^{2/3} + n \epsilon^{-2})$ achieved in Corollary~\ref{cor:primal-mat-guarantee} mirrors the improvements of \cite[Cor. 8.2]{carmon2023whole} over prior algorithms except with $n$ and $d$ flipped (in particular, under the broader restriction of $d \ge n$).

\begin{table}[h!]
    \centering
    \begin{tabular}{@{}ll@{}}
        \toprule
        Method                                                                            & \begin{tabular}[c]{@{}l@{}}Runtime\end{tabular}  \\ \midrule
        Stochastic primal-dual \cite{grigoriadis1995sublinear,clarkson2012sublinear}      & $(n+d)\epsilon^{-2}$                                                                                           \\
        Exact gradient primal-dual \cite{nemirovski2004prox,nesterov2007dual}             & $nd\epsilon^{-1}$                                                                              \\
        Variance-reduced primal-dual \cite{carmon2019variance}                            & $nd + \sqrt{nd(n+d)}\epsilon^{-1}$                                                                   \\
        Ball acceleration (primal-only)     \cite{carmon2023whole}                                                           & $nd + n(d/\epsilon)^{2/3} + d\epsilon^{-2}$                                                                                                 \\
        \rowcolor[HTML]{EFEFEF} 
        Our method (Corollary~\ref{cor:primal-mat-guarantee})                                                                & $nd + d(n/\epsilon)^{2/3} + n\epsilon^{-2}$                                                                             \\                                                                                      & $\max\crl{n,d}^\omega$                                                                                                 \\
        \multirow{-2}{*}{Interior point \cite[resp.,][]{cohen2021solving,brand2021minimum}} & $nd + \min\crl{n,d}^{5/2}$                                                                             \\ \bottomrule
        \end{tabular}
    \caption{\label{table:runtimes-bilinear}
    Runtime bounds for solving the problem~\eqref{eq:primal-MG} to $\epsilon$ accuracy, omitting constant and polylogarithmic factors, in the setting of Definition~\ref{def:mat-games-simplex-in-body}. This table is based on Table 2 in \cite{carmon2023whole}.
    }
\end{table}

\subsection{CVaR at level $\alpha$ DRO}\label{subsec:alpha-CVaR}

In this section, we instantiate $\psi(x, y) \defeq \sum_{i = 1}^n y_i f_i(x)$ for convex, bounded, and $G$-Lipschitz (with respect to the Euclidean norm) functions $f_i : \R^d \to \R$.\footnote{Note that we do not require the functions $f_i$ to be differentiable. Here, it is important that Definition~\ref{def:dual-extraction-setup} only requires $\psi(x, \cdot)$ to be differentiable.} Given a compact, convex set $\xset$ and $\alpha \in [1 / n, 1]$, the primal and dual problem for CVaR at level $\alpha$ are as follows (we will explain the reason for the notation $\fbar$ as opposed to $f$ shortly):
\leqnomode
\begin{align}
    \tag{P-CVaR} \label{eq:primal-CVaR} &\qquad \qquad \qquad \minimize_{x \in \xset} \barf(x) \text{ \; \;\;\; for } \barf(x) \defeq \max_{y \in \simplex^n, \norm{y}_\infty \le \frac{1}{\alpha n}} \sum_{i = 1}^n y_i f_i(x), \text{ and} \\
    \tag{D-CVaR} \label{eq:dual-CVaR} &\qquad \qquad \qquad \maximize_{y \in \simplex^n, \norm{y}_\infty \le \frac{1}{\alpha n}} \phi(y)\text{ for }\phi(y) \defeq \min_{x \in \xset} \sum_{i = 1}^n y_i f_i(x) .
\end{align}
\reqnomode
Our complexity model in this section is the number of computations of the form $(f_i(x), \grad f_i(x))$ for $x \in \xset$ and $i \in [n]$. We refer to the evaluation of $(f_i(x), \grad f_i(x))$ for a given $x \in \xset$ and $i \in [n]$ as a single \emph{first-order query}.
Omitting the Lipschitz constant $G$ and bounds on the range of the $f_i$'s and size of $\xset$ for clarity, \cite[Sec. 4]{levy2020largescale} gave an algorithm which returns an expected\footnote{To be precise, \cite{levy2020largescale} gives a $\Otilde(\alpha^{-1} \epsilon^{-2})$-complexity high probability bound in Theorem 2. They do not state a $\Otilde(\alpha^{-1} \epsilon^{-2})$-complexity expected suboptimality bound explicitly in a theorem, but they note in the text above Theorem 2 that such a bound follows immediately from Propositions 3 and 4 in their paper.} $\epsilon$-optimal point of \eqref{eq:primal-CVaR} with $\Otilde(\alpha^{-1} \epsilon^{-2})$ first-order queries, and also proved a matching lower bound up to logarithmic factors when $n$ is sufficiently large. However, to the best of our knowledge, the best known complexity for obtaining an expected $\epsilon$-optimal point of \eqref{eq:dual-CVaR} is $\Otilde (n \epsilon^{-2})$ via a primal-dual method based on \cite{nemirovski2009robust}; see also \cite{curi2020adaptivesampling,namkoong2016fdivergences, carmon2022distributionally}. In our main guarantee for this section, Theorem~\ref{thm:dual-CVaR-guarantee}, we apply Algorithm~\ref{alg:dual-extraction-framework} to obtain an expected $\epsilon$-optimal point of \eqref{eq:dual-CVaR} with complexity $\Otilde(\alpha^{-1} \epsilon^{-2} + n)$, which always improves upon or matches $\Otilde (n \epsilon^{-2})$ since $\alpha \in [1 / n, 1]$.

Toward stating our main guarantee, we encapsulate the formal assumptions of \cite[Sec. 2]{levy2020largescale}, as well as the relevant setup for Algorithm~\ref{alg:dual-extraction-framework}, in the following definition:

\begin{definition}[CVaR at level $\alpha$ setup]
    \label{def:CVaR-setup}
 We assume $\xset$ is nonempty, closed, convex, and satisfies $\norm{x - y}_2 \le R$ for all $x, y \in \xset$. We also assume, for all $i \in [n]$, that $f_i$ is convex, $G$-Lipschitz with respect to $\norm{\cdot}_2$, and satisfies $f_i(x) \in [0, M]$ for all $x \in \xset$. For a given $\alpha \in [1 / n, 1]$ and target error $\epsilon \in (0, 4M)$, we set $\ytrunc \defeq \inbraces{y \in \simplex^n : \frac{\epsilon}{4 n M} \le y_i \le \frac{1}{\alpha n}, ~ \forall i \in [n]}$ and fix a $(\psi, \xset, \ytrunc)$-simplex setup (Definition~\ref{ex:simplex}).
\end{definition}

Regarding the notation $\fbar$ as opposed to the usual $f$ in \eqref{eq:primal-CVaR}, as well as the fact that we fix a $(\psi, \xset, \ytrunc)$-simplex setup in Definition~\ref{def:CVaR-setup}, where $\ytrunc$ is a truncated version of the dual constraint set which appears in \eqref{eq:primal-CVaR} and \eqref{eq:dual-CVaR} (namely, the CVaR uncertainty set $\inbraces{y \in \simplex^n : \norm{y}_\infty \le \frac{1}{\alpha n}}$): For technical reasons related to the implementation of the $\DRPO$ oracle, we do not apply the dual-extraction framework Algorithm~\ref{alg:dual-extraction-framework} directly to the CVaR primal-dual pair \eqref{eq:primal-CVaR} and \eqref{eq:dual-CVaR}. Instead, per the $(\psi, \xset, \ytrunc)$-simplex setup chosen in Definition~\ref{def:CVaR-setup}, we instantiate Algorithm~\ref{alg:dual-extraction-framework} with the following primal-dual pair (the ``T'' is for truncation):\footnote{We will treat $\phi$ as a function with domain $\simplex^n$ in this section as opposed to having domain $\ytrunc$, as it would if specified exactly as in Section \ref{subsec:main-framework-assumptions-preliminaries}, so that the definition \eqref{eq:dual-CVaR} makes sense. When we refer to an $\epsilon$-maximizer of $\phi$, we will always make it clear whether we are referring to \eqref{eq:dual-CVaR} or \eqref{eq:dual-CVaR-trunc}.}
\leqnomode
\begin{align}
    \tag{P-CVaR-T} \label{eq:primal-CVaR-trunc} &\qquad \qquad \qquad \minimize_{x \in \xset} f(x) \text{ for } f(x) \defeq \max_{y \in \ytrunc} \sum_{i = 1}^n y_i f_i(x), \text{ and} \\
    \tag{D-CVaR-T} \label{eq:dual-CVaR-trunc} &\qquad \qquad \qquad \maximize_{y \in \ytrunc} \phi(y) \text{ for } \phi(y) = \min_{x \in \xset} \sum_{i = 1}^n y_i f_i(x).
\end{align}
\reqnomode

Note that $\phi$ in \eqref{eq:dual-CVaR-trunc} is defined as in \eqref{eq:dual-CVaR}, but $f$ in \eqref{eq:primal-CVaR-trunc} differs from $\fbar$ in \eqref{eq:primal-CVaR} due to the maximization in the former being over the truncated set $\ytrunc$ as opposed to the CVaR uncertainty set. 
We argue below that approximately solving \eqref{eq:dual-CVaR-trunc} suffices to approximately solve \eqref{eq:dual-CVaR} due to the Lipschitzness of $\phi$. 

Next, we state our guarantee for the $\DRPO$ oracle (proven in Section \ref{subapp:alpha-CVaR}), which shows that the $\Otilde(\alpha \epsilon^{-2})$ complexity for \eqref{eq:primal-CVaR} achieved by \cite{levy2020largescale} extends to the case where additional regularization is added (up to a certain amount and with an additional condition on the center of regularization $q$). For the rest of this section, if $h$ is a quantity which depends on (some subset of) the parameters in Definition~\ref{def:CVaR-setup}, we write $h \le \allpoly$ to denote the fact that $h$ is bounded above by a polynomial in the problem parameters, i.e., $h \le \poly(n, d, R, G, M, \epsilon^{-1})$.

\begin{restatable}[$\DRPO$ oracle for CVaR]{lemma}{restateLemDRPOCVaR}
    \label{lem:DRPO-CVaR}
    In the setting of Definition~\ref{def:CVaR-setup},
    for any $\epsprim > 0$, $0 < \lambda \le \allpoly$, and $q \in \simplex^n_{>0}$ with $\max_{i \in [n]} q_i^{-1} \le \allpoly$, there exists an algorithm which returns an expected $\epsprim$-optimal point of $f_{\lambda, q}(x) = \max_{y \in \ytrunc} \inbraces{\sum_{i = 1}^n y_i f_i(x) - \lambda \breg{}{q}{y}}$ with complexity $\Otilde(G^2 R^2 \alpha^{-1} \epsprim^{-2})$. (Equivalently, per Definition~\ref{def:DRPOSP}, we have that $\DRPO(q, \lambda, \epsprim)$ can be implemented with this complexity.)
\end{restatable}

The reason we apply Algorithm~\ref{alg:dual-extraction-framework} with a truncated dual contraint set $\ytrunc$ is precisely to ensure $q_k$ in Line~\ref{algline:q_k} of Algorithm~\ref{alg:dual-extraction-framework} satisfies $\max_{i \in [n]} [q_k]_i^{-1} \le \poly(\cdots)$ for all iterations $k \in [K]$, so that the oracle call $\DRPO(q_k, \Lambda_k, \epsilon_k)$ in Line~\ref{algline:x_k} of Algorithm~\ref{alg:dual-extraction-framework} can be implemented per Lemma~\ref{lem:DRPO-CVaR}. Note that it is not a priori clear that $q_k \in \ytrunc$, but it is straightforward to show that in the simplex setup of Definition~\ref{ex:simplex}, we have that each entry $[q_k]_i$ for $i \in [n]$ is a weighted geometric mean of $[y_0]_i, \dots, [y_{k - 1}]_i$, normalized to the probability simplex. Thus, the fact that $y_k \in \ytrunc$ for $k \in \inbraces{0} \cup [K]$ yields $\max_{i \in [n]} [q_k]_i^{-1} \le \poly(\cdots)$ as a result. We discuss the reason for the conditions $\max_{i \in [n]} q_i^{-1} \le \allpoly$ and $\lambda \le \allpoly$ in Lemma~\ref{lem:DRPO-CVaR} in Appendix \ref{subapp:alpha-CVaR}.

Applying Lemma~\ref{lem:DRPO-CVaR}, we obtain the following guarantee via Algorithm~\ref{alg:dual-extraction-framework}. Note that the upper bound on $\epsilon$ in Theorem~\ref{thm:dual-CVaR-guarantee} is without loss of generality since if $\epsilon \ge M$, any feasible point is $\epsilon$-optimal.

\begin{theorem}[Guarantee for \eqref{eq:dual-CVaR}]
    \label{thm:dual-CVaR-guarantee}
In the setting of Definition~\ref{def:CVaR-setup} with target error $\epsilon \in (0, 4M)$ and $\alpha \in [1 / n, 1]$, there exists an algorithm which computes an expected $\epsilon$-optimal point of \eqref{eq:dual-CVaR} with complexity $\Otilde(n + G^2 R^2 \alpha^{-1} \epsilon^{-2})$. 
\end{theorem}

\begin{proof}
    Consider Algorithm~\ref{alg:dual-extraction-framework} with the simplex setup fixed in Definition~\ref{def:CVaR-setup}, 
    $y_0 \defeq \frac{1}{n} \ones$, and schedules given by Corollary~\ref{cor:example-schedules-wlog} with $\mu_r \gets 1$, $B \gets \ln n$ (see Lemma~\ref{lem:bound-on-KL-from-uniform}), and $L \gets M$. Regarding the latter, it is immediate that $\phi$ is $M$-Lipschitz with respect to $\norm{\cdot}_1$ due to the assumption on the range of each $f_i$ in Definition~\ref{def:CVaR-setup} (see, e.g., \cite[Sec. D.4.4]{urruty2004convexanalysistextbook}). Then by Corollary~\ref{cor:example-schedules-wlog}, we have that $K = O (\log_2 \frac{M^2 \ln n}{\epsilon^2})$ suffices to obtain an $\epsilon / 2$-optimal point for \eqref{eq:dual-CVaR-trunc}, which is $\epsilon$-optimal for \eqref{eq:dual-CVaR} since the Lipschitzness of $\phi$ implies
    \begin{align*}
        \max_{y \in \ytrunc} \phi(y)  \le \max_{y \in \simplex^n, \norm{y}_\infty \le \frac{1}{\alpha n}} \phi(y) \le \max_{y \in \ytrunc} \phi(y) + \epsilon / 4.
    \end{align*}
    
    As for the complexity, we claim the $\DRPO$ oracle call in Line~\ref{algline:x_k} in Algorithm~\ref{alg:dual-extraction-framework} can be implemented with complexity $\Otilde(G^2 R^2 \alpha^{-1} \epsilon^{-2})$ via Lemma~\ref{lem:DRPO-CVaR}, and it is immediate that the $\DRBR$ oracle call in Line~\ref{algline:y_k} in Algorithm~\ref{alg:dual-extraction-framework} can be implemented with complexity $n$ by querying $f_1(x_k), \dots, f_n(x_k)$. Thus, we obtain the desired complexity since $K = \Otilde(1)$.

    It remains to verify that the stipulations of Lemma~\ref{lem:DRPO-CVaR} hold when the $\DRPO$ oracle is called in Line~\ref{algline:x_k} in Algorithm~\ref{alg:dual-extraction-framework}; namely, that $\max_{i \in [n]} [q_k]_i^{-1} \le \poly(\cdots)$ and $\Lambda_k \le \poly(\cdots)$ for all $k \in [K]$. The latter is immediate from the choice of dual-regularization schedule per Corollary~\ref{cor:example-schedules-wlog} and the fact that $K = O (\log_2 \frac{M^2 \ln n}{\epsilon^2})$. As for the former, this follows because Lemma~\ref{lem:convex-conjugate-neg-entropy} and the alternate expression for $q_k$ in Line~\ref{algline:q_k} of Algorithm~\ref{alg:dual-extraction-framework} imply
    \begin{align*}
        [q_k]_i = \frac{\prod_{j = 0}^{k - 1} [y_j]_i^{\lambda_j / \Lambda_k}}{\sum_{\ell = 1}^n \prod_{j = 0}^{k - 1} [y_j]_\ell^{\lambda_j / \Lambda_k}}
    \end{align*}
    for all $k \in [K]$ and $i \in [n]$. Then the claimed bound on $[q_k]_i^{-1}$ follows from the fact that $y_k \in \ytrunc$ for all $k \in [K] \cup \inbraces{0}$. Indeed, note
    \begin{align*}
        1 \le [y_j]_i^{- \lambda_j / \Lambda_k} \le [y_j]_i^{-1} \le \allpoly \implies 1 \le \prod_{j = 0}^{k - 1} [y_j]_i^{- \lambda_j / \Lambda_k} \le \allpoly.
    \end{align*}
\end{proof}

\section{Obtaining critical points of convex functions}\label{sec:convex-critical-point}

In this section, our goal is to obtain an approximate critical point of a convex, $\beta$-smooth function $h : \R^n \to \R$, given access to a gradient oracle for $h$. We show that our general framework yields an algorithm with the optimal query complexity for this problem.
In Section \ref{subsec:critical-point-prelims}, we give the formal problem definition and some important preliminaries. In Section \ref{subsec:critical-point-extraction-framework}, we give the setup for applying our main framework Algorithm~\ref{alg:dual-extraction-framework} to this problem and a sketch of why the resulting algorithm works. In Section \ref{subsec:convex-critical-resulting-alg}, we formally state the resulting algorithm for obtaining an approximate critical point of $h$ and prove that it achieves the optimal rate using the guarantees associated with Algorithm~\ref{alg:dual-extraction-framework}.

\subsection{Preliminaries for Section \ref{sec:convex-critical-point}}\label{subsec:critical-point-prelims}

Throughout Section \ref{sec:convex-critical-point}, we fix $\norm{\cdot}$ to be the standard Euclidean norm over $\R^n$. We assume $h : \R^n \to \R$ is convex, $\beta$-smooth with respect to $\norm{\cdot}$, and $\gap \defeq h(x_0) - \inf_{x \in \R^n} h(x) < \infty$ for an arbitrary initialization point $x_0 \in \R^n$.
We access $h$ through a gradient oracle. For $\gamma > 0$, our goal will be to obtain a $\gamma$-critical point of $h$, i.e., a point $x \in \R^n$ such that $\norm{\grad h(x)} \le \gamma$.
Instead of operating on $h$ itself, our algorithm will operate on a regularized version of $h$:
\begin{align}
    \label{eq:reg-func-f}
    f(x) \defeq h(x) + \frac{\gamma^2}{16 \gap} \norm{x - x_0}^2.
\end{align}

This notation was chosen to mirror the notation of Section \ref{subsec:main-framework-assumptions-preliminaries}; $f$ will be the primal function when we apply the framework. Let $\fopt$ denote the unique global minimum of $f$. The following corollary of \Cref{lem:general-reg-func-properties} in Appendix \ref{app:deferred-convex-critical-proofs} summarizes the key properties of $f$:

\begin{corollary}[Properties of the regularized function $f$]
    \label{cor:prop-of-reg-func-f}
We have
\begin{enumerate}
    \item $\norm{\fopt - x_0} \le 4 \gap / \gamma$.

    \item If $u \in \R^n$ is such that $\norm{\grad f(u)} \le \gamma / 4$, then $\norm{\grad h(u)} \le \gamma$.
\end{enumerate}
\end{corollary}

\begin{proof}
This follows immediately from  \Cref{lem:general-reg-func-properties} with $\alpha \gets \frac{\gamma^2}{8 \gap}$ and $\nu \gets \gamma / 4$.
\end{proof}

The second part of Corollary~\ref{cor:prop-of-reg-func-f} says that to find a $\gamma$-critical point of $h$, it suffices to find a $(\gamma / 4)$-critical point of $f$. Furthermore, clearly a single query to $\grad h$ suffices to obtain $\grad f$ at a point. As a result, we will focus on finding a $(\gamma / 4)$-critical point of $f$. Furthermore, Corollary~\ref{cor:prop-of-reg-func-f} may be of independent interest since it trivially allows one to achieve a gradient query complexity of $O \inparen{ \sqrt{\beta \gap} / \gamma }$ via a method which achieves query complexity $O \inparen{\sqrt{\beta \norm{x_0 - \hopt} / \gamma}}$ (for $\hopt$ defined as some minimizer of $h$ over $\R^n$, assuming one exists); see Section \ref{subsec:related-work}.

The reason we perform this regularization before applying our framework is it enables us to obtain a sufficiently tight bound on $\breg{}{y_0}{\yopt}$ (equivalently, a small enough value of $B$ when we ultimately apply Corollary~\ref{cor:example-schedules-no-log}). It is possible to apply the framework more directly to $h$, but it is not clear how to do so in a way that achieves an optimal complexity.

Finally, we provide a notation guide for Section \ref{sec:convex-critical-point} in Table \ref{table:convex-critical-notation}, which may be useful to reference as additional notation is introduced in Sections \ref{subsec:critical-point-extraction-framework} and \ref{subsec:convex-critical-resulting-alg}.

\begin{table}[]
    \centering
    \small
    \begin{tabular}{@{}lll@{}}
    \toprule
    \multicolumn{1}{l}{Notation} & \multicolumn{1}{l}{Description}                      & \multicolumn{1}{l}{Section}                       \\ \midrule
    $\norm{\cdot}$                 & Euclidean norm                               & $\ref{subsec:critical-point-prelims}$              \\
    $h$                            & Convex, $\beta$-smooth function                       &                                                    \\
    $\gamma$                       & Target critical point error for $h$                &                                                    \\
    $x_0$                          & Arbitrary initialization point                    &                                                    \\
    $\gap$                         & $h(x_0) - \inf_{x \in \R^n} h(x) < \infty$                                  &                                                    \\
    $f(x)$                         & $h(x) + \frac{\gamma^2}{16 \gap} \norm{x - x_0}^2$      &                                                    \\
    $\fopt$                        & The global minimizer of $f$                                 &                                                    \\
    $\psi(x, y)$                   & $\inangle{x, y} - \fconj(y)$                          &     $\ref{subsec:critical-point-extraction-framework}$                                               \\
    $R$                            & $5 \gap / \gamma$                                     &                                                    \\
    $\xset$                        & $B_R^n(x_0)$                                          &                                                    \\
    $\yset$                        & $\R^n$                                                &  \\
    dgf $r$                            & $\fconj$                                     &                                                    \\
    $\phi(y)$                      & $\inangle{x_0, y} - R \norm{y} - \fconj(y)$           &                                                    \\
    $\lambda_k$                    & $2^k / 32$                                            & $\ref{subsec:convex-critical-resulting-alg}$       \\
    $\epsilon_k$                   & $\gap / (64 \cdot 1.5^k)$                             &                                                    \\
    $\CGM$              & Fast composite gradient method oracle & \\ \bottomrule                                                  
    \end{tabular}
    \notarxiv{\vspace{8pt}}
	\cprotect\caption{\label{table:convex-critical-notation}
		Notation guide for Section \ref{sec:convex-critical-point}
	} 
    \end{table}

\subsection{Instantiating the framework}\label{subsec:critical-point-extraction-framework}

For this application, we instantiate
\begin{align*}
    \psi(x, y) &\defeq \inangle{x, y} - \fconj(y).
\end{align*}
Recall that $\psi$ is the \textit{Fenchel game} \cite{abernethy2017equilibrium,wang2023noregretdynamicsfenchelgame,cohen2021relative,jin2022sharper}; see Section \ref{subsec:our-results} for a discussion of why it is a natural choice in this setting. For the rest of Section \ref{sec:convex-critical-point}, we fix a $(\psi, \xset \defeq B_R^n(x_0), \yset \defeq \R^n,\fconj)$-unbounded setup (Definition~\ref{ex:unbounded}) with $R \defeq 5 \gap / \gamma$.
$\fconj$ is a valid choice for the dgf because $\fconj$ is differentiable and $\inparen{\beta + \frac{\gamma^2}{8 \gap}}^{-1}$-strongly convex \cite[Thm. 6.11]{orabona2023modern}. The strong convexity of $\fconj$ also implies that Assumption~\ref{as:bounded-strongly-concave} holds. Note that the associated primal function $x \mapsto \max_{y \in \R^n} \psi(x, y)$ is precisely $f^{**} = f$ (hence the choice of notation in \eqref{eq:reg-func-f}), and the dual function is given by
\begin{align*}
    \phi(y) = \min_{x \in B_R^n(x_0)} \inbraces{\inangle{x, y} - \fconj(y)} 
    = \inangle{x_0 -R \frac{y}{\norm{y}}, y} - \fconj(y) 
    = \inangle{x_0, y} - R \norm{y} - \fconj(y).
\end{align*}
Next, the following lemma fulfills part of the outline given in Section \ref{subsec:our-results} by showing that approximately optimal points for the dual objective \eqref{eq:dual} must have small norm.

\begin{lemma}[Bounding the norm by dual suboptimality]
    \label{lem:convex-crit-bounding-norm-dual-point}
   If $y \in \R^n$ is $\epsilon$-optimal for \eqref{eq:dual} for some $\epsilon > 0$, then $\norm{y} \le \epsilon \gamma / \Delta$.
\end{lemma}

\begin{proof}
By Corollary~\ref{cor:prop-of-reg-func-f} and the choice of $R$, note that $\fopt \in \xset$, implying $\fopt$ is primal optimal even after restricting the search space to $\xset$, i.e., to use the notation of Section \ref{subsec:main-framework-assumptions-preliminaries}, $\fopt = \xopt \defeq \argmin_{x \in \xset} f(x)$. Let $\yopt \defeq \argmax_{y \in \R^n} \phi(y)$ denote the dual optimum as in Section \ref{subsec:main-framework-assumptions-preliminaries}, which is unique by the strong convexity of $\fconj$. Since $f(\fopt) = \phi(\yopt)$ by the minimax principle and $y$ is $\epsilon$-optimal, we have 
\begin{align*}
    \phi(\yopt) - \phi(y) \le \epsilon \iff f(\fopt) - \phi(y) \le \epsilon.
\end{align*}
Expanding the latter gives
\begin{align*}
    &f(\fopt) - \inangle{x_0, y} + R \norm{y} + \fconj(y) \le \epsilon, \\
    \overimp{(i)}    &  R \norm{y} + \inangle{\fopt - x_0, y} \le \epsilon, \\
    \overimp{(ii)} & R \norm{y} - \norm{\fopt - x_0} \norm{y} \le \epsilon, \\
    \overimp{(iii)} & \norm{y} \le \frac{\epsilon}{R - \norm{\fopt - x_0}} \le \frac{\epsilon \gamma}{\Delta},
\end{align*}
where we used $(i)$ the Fenchel-Young inequality, $(ii)$ Cauchy-Schwarz, and $(iii)$ Corollary~\ref{cor:prop-of-reg-func-f}. 
\end{proof}

We now derive the oracles of Definitions \ref{def:DRPO} and \ref{def:DRBR}. Regarding Definition~\ref{def:DRPO}, for the rest of Section \ref{sec:convex-critical-point} we restrict $\DRPO(\cdot)$ to denote a deterministic implementation of the $\DRPO$ oracle, since we can always obtain a deterministic implementation in this application.
Then the following corollary is an immediate consequence of a more general lemma given in Appendix \ref{app:deferred-convex-critical-proofs} which characterizes the properties of the Fenchel game with added dual regularization; see also Section \ref{subsec:our-results}.

\begin{corollary}
    \label{cor:oracles-for-convex-critical}
The set of valid output points of $\DRPO(q \in \R^n, \lambda > 0, \epsprim > 0)$ is precisely
\begin{align*}
     &\epsargmin{\epsprim}{x \in B_R^n(x_0)} (1 + \lambda) \cdot f \inparen{\frac{x + \lambda \grad \fconj (q)}{1 + \lambda}}, \text{ and} \\
     &\DRBR(q \in \R^n, \lambda > 0, x \in B_R^n(x_0)) = \grad f \inparen{\frac{x + \lambda \grad \fconj (q)}{1 + \lambda}}.
\end{align*}
\end{corollary}

\begin{proof}
Apply Lemma~\ref{lem:Fenchel-game-with-reg} with $g \gets f$.
\end{proof}

Taken together, Lemma~\ref{lem:convex-crit-bounding-norm-dual-point} and Corollary~\ref{cor:oracles-for-convex-critical} nearly immediately imply that Algorithm~\ref{alg:dual-extraction-framework} can be applied to the above setup to obtain a $(\gamma / 4)$-critical point of $f$ (and therefore a $\gamma$-critical point of $h$). In particular, we will apply the schedules of Corollary~\ref{cor:example-schedules-no-log} to certify that the output $y_K$ of Algorithm~\ref{alg:dual-extraction-framework} is close in distance to an $\epsilon$-optimal point for \eqref{eq:dual} for an appropriate choice of $\epsilon > 0$. Then Lemma~\ref{lem:convex-crit-bounding-norm-dual-point} and a triangle inequality yield a bound on $\norm{y_K}$. Finally, since 
\begin{align*}
    y_K \defeq \DRBR(q_K, \Lambda_K, x_K) = \grad f \inparen{\frac{x_K + \Lambda_K \grad \fconj (q_K)}{1 + \Lambda_K}}
\end{align*}
by Corollary~\ref{cor:oracles-for-convex-critical}, we have that $\frac{x_K + \Lambda_K \grad \fconj (q_K)}{1 + \Lambda_K}$ is an approximate critical point of $f$ (and therefore $h$).
One may worry about the presence of $\grad \fconj(q_K)$ here and, more generally, the presence of $\grad \fconj(q)$ in the expressions for the oracles in Corollary~\ref{cor:oracles-for-convex-critical}. However, $\grad \fconj(\cdot)$ never needs to be evaluated explicitly since per the alternate expression for $q_k$ given in Line~\ref{algline:q_k} of Algorithm~\ref{alg:dual-extraction-framework}, note that $q_k$ was itself computed as the gradient of $f$ at a point (recall the dgf is $\fconj$ and $f = f^{**}$), in which case $\grad \fconj$ simply undoes this operation by Lemma~\ref{lem:subdiff-properties-of-conjugate}.

We formalize this sketch and provide a complexity guarantee in the next section. We also reframe this sketch and treat the sequence of $\frac{x_k + \Lambda_k \grad \fconj (q_k)}{1 + \Lambda_k}$ terms as our iterates (as opposed to the sequence of $x_k$'s), as this leads to a simpler statement and interpretation of the resulting algorithm.

\subsection{The resulting algorithm and guarantee}
\label{subsec:convex-critical-resulting-alg}

We now formalize the sketch given at the end of the previous section, state the resulting algorithm, and provide a complexity guarantee. But first, we define a subroutine which will be used by the algorithm to implement the $\DRPO$ oracle:

\begin{definition}[$\CGM$ oracle {\cite{tseng2008accelerated, nesterov2018textbook}}]
    \label{def:CGM}
    A \emph{$(\zeta > 0, w \in \R^n, \epsilon > 0)$-fast composite gradient method oracle}, $\CGM(\zeta, w, \epsilon)$, returns an $\epsilon$-minimizer of $f$ over $x \in B^n_\zeta(w)$, i.e., an element of $\epsargmaxtxtstyl{\epsilon}{x \in B^n_\zeta(w)} f(x)$, using at most $O\inparen{ 1 + \sqrt{\frac{\beta \zeta^2}{\epsilon}}}$ queries to $\grad f$.
\end{definition}

For example, implementations with a small constant can be found in \cite{tseng2008accelerated} or \cite[Sec. 6.1.3]{nesterov2018textbook}. The implementation of the $\CGM$ oracle falls under fast gradient methods for composite minimization, where letting $g$ denote a convex, $\beta$-smooth function and $\Psi$ a ``simple regularizer'' (a quadratic in our case), the goal is to minimize $\gtilde(x) \defeq g(x) + \Psi(x)$ with the same complexity as it takes to minimize $g$. The domain constraint can also be baked into the regularizer $\Psi$ by adding an indicator.

Our method for computing a $\gamma$-critical point of $h$ is given in Algorithm~\ref{alg:critical-point-algorithm}, with the associated guarantee in Theorem~\ref{thm:convex-critical-main-guarantee}. We note that the decision to introduce the equivalent notation $z_0$ for $x_0$ in Line~\ref{algline:z_0} is aesthetic (to make Line~\ref{algline:CGM} simpler to state and interpret). Furthermore, we state Algorithm~\ref{alg:critical-point-algorithm} for general schedules $(\lambda_k)_{k = 0}^{K - 1}$ and $(\epsp_k)_{k = 1}^K$ for clarity, but ultimately we will choose the schedules given in Theorem~\ref{thm:convex-critical-main-guarantee}, which correspond to particularizing the schedules of Corollary~\ref{cor:example-schedules-no-log} to this setting. With this choice of schedules, $\Lambda_k \approx 2^k$ and $\epsilon_k \approx \gap / 1.5^k$ so that $\frac{\epsilon_k}{1 + \Lambda_k} \approx \gap / 3^k$. As a result, Algorithm~\ref{alg:critical-point-algorithm} can be interpreted as optimizing $f$ in a sequence of balls where the radius and target error are both decreasing geometrically, and the center is a convex combination of the past iterates. While we choose the iteration count $K$ to be logarithmic in the problem parameters, we avoid multiplicative logarithmic factors in the total complexity because the ratio $\zeta^2 / \epsilon$ in the complexity of the $\CGM$ oracle call (to borrow the notation of Definition~\ref{def:CGM}) in Line~\ref{algline:CGM} of Algorithm~\ref{alg:critical-point-algorithm} is $\approx \frac{R^2}{4^k} \cdot \frac{3^k}{\gap}$ at the $k$-th iteration, meaning it is collapsing geometrically.

\begin{algorithm}[h] %
	\setstretch{1.1}
	\caption{Algorithm for obtaining a $\gamma$-critical point of $h$}
	\label{alg:critical-point-algorithm}
	\LinesNumbered
	\DontPrintSemicolon
	\Input{
		Sequences $(\lambda_k)_{k = 0}^{K - 1}$ and $(\epsp_k)_{k = 1}^K$ specified in Theorem~\ref{thm:convex-critical-main-guarantee}, iteration count $K \in \N$, $\CGM$ oracle (Definition~\ref{def:CGM})
	}

    $z_0 \defeq x_0$\; \label{algline:z_0}

    \For{$k = 1, 2, \ldots, K$ }{ 
        $\Lambda_k = \sum_{j = 0}^{k - 1} \lambda_j$\;
        $w_k = \frac{z_0 + \sum_{j = 0}^{k - 1} \lambda_j z_j}{1 + \Lambda_k}$\;
        $z_k = \CGM\inparen{\frac{R}{1 + \Lambda_k}, w_k, \frac{\epsilon_k}{1 + \Lambda_k}}$ \tcp*{$z_k \in \epsargmin{\epsilon_k / (1 + \Lambda_k)}{z \in B^n_{R / (1 + \Lambda_k)}\inparen{w_k}} f(z)$} \label{algline:CGM}
    }
    \Return $z_K$\;
\end{algorithm}

Toward analyzing Algorithm~\ref{alg:critical-point-algorithm}, we first connect the sequence of iterates $z_k$ produced by Algorithm~\ref{alg:critical-point-algorithm} to the sequence of iterates $x_k, y_k, q_k$ produced by Algorithm~\ref{alg:dual-extraction-framework} with the same input parameters. Namely, we are formalizing the comment made at the end of Section \ref{subsec:critical-point-extraction-framework} about reframing the sequence of iterates to achieve a more interpretable algorithm.

\begin{lemma}[Connecting Algorithm~\ref{alg:critical-point-algorithm} to Algorithm~\ref{alg:dual-extraction-framework}]
    \label{lem:connecting-convex-crit-to-framework}
    Consider Algorithm~\ref{alg:dual-extraction-framework} with input given by a $(\psi, B_R^n(x_0), \R^n,\fconj)$-unbounded setup (Definition~\ref{ex:unbounded}); $y_0 \defeq \grad f(x_0)$; and $K$, $(\epsilon_k)_{k = 1}^K$, and $(\lambda_k)_{k = 0}^{K - 1}$ as in Algorithm~\ref{alg:critical-point-algorithm}. Then letting $(z_k)_{k = 0}^K$
    denote the sequence of iterates generated by Algorithm~\ref{alg:critical-point-algorithm}, the following are valid sequences of iterates for Algorithm~\ref{alg:dual-extraction-framework}:
\begin{align}
    q_k &= \grad f \inparen{\frac{1}{\Lambda_k} \sum_{j = 0}^{k - 1} \lambda_j z_j} &&\text{ for $k \in [K]$}, \label{eq:q_k-expr} \\
    x_k &= (1 + \Lambda_k) z_k - \sum_{j = 0}^{k - 1} \lambda_j z_j &&\text{ for $k \in [K]$, and} \label{eq:x_k-expr} \\
    y_k &= \grad f(z_k) &&\text{ for $k \in \inbraces{0} \cup [K]$}. \label{eq:y_k-expr}
\end{align}
\end{lemma}

\begin{proof}
We proceed by induction. For $k = 0$, we have $z_0 = x_0$ by definition and thus $y_0 = \grad f(z_0)$. Now for $k > 0$, suppose $y_j = \grad f(z_j)$ for all $j \in \inbraces{0, 1, \dots, k - 1}$. Then recalling that the dgf $r$ is $\fconj$ and $f^{**} = f$, it is immediate from the alternate expression for $q_k$ given in Line~\ref{algline:q_k} of Algorithm~\ref{alg:dual-extraction-framework} that
\begin{align*}
    q_k = \grad f \inparen{\frac{1}{\Lambda_k} \sum_{j = 0}^{k - 1} \lambda_j \grad  \fconj(y_j)} =
    \grad f \inparen{\frac{1}{\Lambda_k} \sum_{j = 0}^{k - 1} \lambda_j z_j} .
\end{align*}
Thus, we have proven the expression \eqref{eq:q_k-expr}, and now aim to prove the expression \eqref{eq:x_k-expr}. By Corollary~\ref{cor:oracles-for-convex-critical}, the expression \eqref{eq:x_k-expr} is a valid choice for $x_k$ if and only if 
\begin{align}
    x_k \in \epsargmin{\epsilon_k}{x \in B_R^n(x_0)} (1 + \Lambda_k) \cdot f \inparen{\frac{x + \Lambda_k \grad \fconj(q_k)}{1 + \Lambda_k}}
    &= \epsargmin{\epsilon_k}{x \in B_R^n(x_0)} (1 + \Lambda_k) \cdot f \inparen{\frac{x + \sum_{j = 0}^{k - 1} \lambda_j z_j}{1 + \Lambda_k}} \nonumber \\
    &= \epsargmin{\epsilon_k / (1 + \Lambda_k)}{x \in B_R^n(x_0)}  f \inparen{\frac{x + \sum_{j = 0}^{k - 1} \lambda_j z_j}{1 + \Lambda_k}}. \label{eq:transformed-set-for-x_k}
\end{align}
Next, rearranging \eqref{eq:x_k-expr}, note
\begin{align}
    \label{eq:z_k-in-terms-of-x_k}
    z_k = \frac{x_k + \sum_{j = 0}^{k - 1} \lambda_j z_j}{1 + \Lambda_k}.
\end{align}
The image of $B^n_R(x_0)$ under the map $x \mapsto \frac{x + \sum_{j = 0}^{k - 1} \lambda_j z_j}{1 + \Lambda_k}$ is precisely $B^n_{R / (1 + \Lambda_k)}\inparen{\frac{x_0 + \sum_{j = 0}^{k - 1} \lambda_j z_j}{1 + \Lambda_k}}$. This, combined with the fact that $z_0 = x_0$ and the expression \eqref{eq:z_k-in-terms-of-x_k} for $z_k$, implies $x_k$ lies in the set \eqref{eq:transformed-set-for-x_k} if and only if\footnote{Here, we are using the fact that for any $\epsilon > 0$, functions $f : \R^n \to \R$, $g : \R^n \to \R^n$, and set $U \subseteq \R^n$, it is the case that a point $x \in U$ satisfies $f(g(x)) \le f(g(x')) + \epsilon$ for all $x' \in U$ if and only if $z \defeq g(x) \in g(U)$ satisfies $f(z) \le f(z') + \epsilon$ for all $z' \in g(U)$.} 
\begin{align*}
    z_k \in  \epsargmin{\epsilon_k / (1 + \Lambda_k)}{z \in B^n_{R / (1 + \Lambda_k)}\inparen{\frac{z_0 + \sum_{j = 0}^{k - 1} \lambda_j z_j}{1 + \Lambda_k}}} f(z),
\end{align*}
which is exactly Line~\ref{algline:CGM} in Algorithm~\ref{alg:critical-point-algorithm}. Then we have proven the expression \eqref{eq:x_k-expr}, and to conclude, \eqref{eq:z_k-in-terms-of-x_k} and Corollary~\ref{cor:oracles-for-convex-critical} imply $y_k = \grad f(z_k)$. 
\end{proof}

Having connected Algorithm~\ref{alg:critical-point-algorithm} to Algorithm~\ref{alg:dual-extraction-framework}, we can apply the schedules given in Corollary~\ref{cor:example-schedules-no-log} to show that Algorithm~\ref{alg:critical-point-algorithm} returns a $\gamma$-critical point of $h$ with an optimal complexity.

\begin{theorem}[Guarantee for Algorithm~\ref{alg:critical-point-algorithm}]
    \label{thm:convex-critical-main-guarantee}
For any\footnote{The restriction on $\gamma$ is without loss of generality since $\norm{\grad h(x_0)} \le \sqrt{2 \beta \gap}$ by smoothness. We add it because it simplifies the analysis.} $\gamma \in (0, \sqrt{2 \beta \gap})$ and with $K = O(\log (\beta \Delta / \gamma))$,
the output of Algorithm~\ref{alg:critical-point-algorithm} with schedules given by
\begin{align}
    \label{eq:schedules-for-convex-crit}
    \lambda_k = \frac{2^k}{32} \text{ for $k \in \inbraces{0} \cup [K - 1]$}
    \text{ and }
    \epsp_k = \frac{\gap}{64 \cdot 1.5^k} \text{ for $k \in [K]$}
\end{align}
satisfies $\norm{\grad h(z_K)} \le \gamma$, and the algorithm makes at most $O \inparen{ \frac{\sqrt{\beta \gap}}{\gamma} }$ gradient queries to $h$.
\end{theorem}

\begin{proof}
    We first prove $\norm{\grad h(z_k)} \le \gamma$ and then bound the total complexity.

    \paragraph{Correctness.} By Corollary~\ref{cor:prop-of-reg-func-f}, it suffices to show $\norm{\grad f(z_K)} \le \gamma / 4$. Toward this goal, let $(q_k)_{k = 1}^K$, $(x_k)_{k = 1}^K$, and $(y_k)_{k = 0}^K$ be defined as in Lemma~\ref{lem:connecting-convex-crit-to-framework} as a function of the iterates $(z_k)_{k = 0}^K$ produced by Algorithm~\ref{alg:critical-point-algorithm}. Then the former are valid sequences of iterates for Algorithm~\ref{alg:dual-extraction-framework} run with input as in Lemma~\ref{lem:connecting-convex-crit-to-framework}, so they satisfy the guarantee of Corollary~\ref{cor:example-schedules-no-log}. Indeed, note that \eqref{eq:schedules-for-convex-crit} is precisely \eqref{eq:schedules-no-log} with $B \gets \gap$ and $\epsilon \gets \gap / 8$. 

    We claim $B \gets \gap$ is indeed a valid choice with respect to the input to Algorithm~\ref{alg:dual-extraction-framework} specified in Lemma~\ref{lem:connecting-convex-crit-to-framework}, namely, $\breg{f^*}{y_0}{\yopt} \le \Delta$. To see this, note that Lemma~\ref{lem:convex-crit-bounding-norm-dual-point} implies $\yopt = 0$ by sending $\epsilon \to 0$. 
    Then
    \begin{align*}
        \breg{\fconj}{y_0}{\yopt} = \breg{f}{\grad \fconj(\yopt)}{\grad \fconj(y_0)} = \breg{f}{\grad \fconj(\grad f(\fopt))}{\grad \fconj(\grad f(x_0))} = \breg{f}{\fopt}{x_0} &= f(x_0) - f(\fopt) \\
        &= h(x_0) - f(\fopt) \\
        &\le \gap 
    \end{align*}
    since $f(\fopt) \ge \inf_{x \in \R^n} h(x)$ as $f \ge h$ pointwise.

    Then the guarantee of Corollary~\ref{cor:example-schedules-no-log} for the parallel sequences of iterates $(q_k)_{k = 1}^K, (x_k)_{k = 1}^K, (y_k)_{k = 0}^K$ with $B \gets \gap$, $\epsilon \gets \gap / 8$ and resulting choice of schedules \eqref{eq:schedules-for-convex-crit} is 
    \begin{align*}
        \norm{y_K - u} \le \frac{1}{1.5^K} \sqrt{\frac{2 \gap}{\mufconj}},
    \end{align*}
    where $\mufconj \defeq \inparen{\beta + \frac{\gamma^2}{8 \gap}}^{-1} \ge 1 / (2 \beta)$ (recall $\gamma < \sqrt{2 \beta \Delta}$ by assumption) is the strong convexity constant of $\fconj$ \cite[Thm. 6.11]{orabona2023modern}, and $u \in \R^n$ is a point which is $(\gap / 8)$-optimal for the dual \eqref{eq:dual}. Then Lemma~\ref{lem:convex-crit-bounding-norm-dual-point} yields $\norm{u} \le \gamma / 8$, so a triangle inequality implies
    \begin{align*}
        \norm{y_K} \le  \frac{1}{1.5^K} \sqrt{\frac{2 \gap}{\mufconj}} + \frac{\gamma}{8} \le 
        \frac{1}{1.5^K} \cdot 2 \sqrt{\beta \Delta} + \frac{\gamma}{8}  
        \le \gamma / 4
    \end{align*}
    by the choice of $K$. Finally, $y_K = \grad f(z_K)$ by definition.

    \paragraph{Complexity.}     
    Note that we can clearly obtain the gradient of $f$ at a point via a single gradient query to $h$, so we will not bother distinguishing between the two. For $k \in [K]$, let $\T_k$ denote the number of gradient queries made during the $k$-th iteration of Algorithm~\ref{alg:critical-point-algorithm}. Then from Definition~\ref{def:CGM} and Line~\ref{algline:CGM} of Algorithm~\ref{alg:critical-point-algorithm}, clearly
    \begin{align*}
        \T_k = O \inparen{
            1 + \sqrt{\beta \cdot \frac{1 + \Lambda_k}{\epsilon_k} \cdot \frac{R^2}{(1 + \Lambda_k)^2}}
        } = O \inparen{1 + \sqrt{\frac{\beta R^2}{\epsilon_k(1 + \Lambda_k)}}}
    \end{align*}
    Note that
    \begin{align*}
        &\Lambda_k = \sum_{j = 0}^{k - 1} \lambda_j = \frac{1}{32} \sum_{j = 0}^{k - 1} 2^j = \frac{1}{32} (2^k - 1), \\
        \implies & \epsilon_k (1 + \Lambda_k) = \frac{\gap}{64 \cdot 1.5^k}  \cdot \inparen{ \frac{1}{32} (2^k - 1) + 1}  \ge C  \cdot \gap \cdot (4/3)^k
    \end{align*}
    for a universal constant $C$.
    Then recalling $R = 5 \Delta / \gamma$ by definition, we can bound the total complexity as
    \begin{align*}
        \sum_{k = 1}^K \T_k \le \sum_{k = 1}^K O \inparen{1 + \sqrt{\frac{\beta \Delta}{\gamma^2}} \cdot (3/4)^{k / 2}} = O \inparen{\frac{\sqrt{\beta \Delta}}{\gamma} + K} = O \inparen{ \frac{\sqrt{\beta \Delta}}{\gamma} }.
    \end{align*}

\end{proof}

 \section*{Acknowledgments}
 YC acknowledges support from the Israeli Science Foundation (ISF) grant no. 2486/21, and the Alon Fellowship. LO acknowledges support from NSF Grant CCF-1955039. AS acknowledges support from a Microsoft Research Faculty Fellowship, NSF CAREER Grant CCF-1844855, NSF Grant CCF-1955039, and a PayPal research award.

\bibliographystyle{abbrvnat}

\newpage
\appendix

\section{Deferred proofs from Section \ref{sec:dual-extraction-framework}}
\label{app:deferred-main-framework-proofs}

In this section, we give additional lemmas referenced in Section \ref{sec:dual-extraction-framework}. In Section \ref{subapp:verify-simplex-setting}, we verify that the simplex setting of Definition~\ref{ex:simplex} satisfies Assumption~\ref{as:technical-dual} in Definition~\ref{def:dual-extraction-setup}. In Section \ref{subapp:equiv-expression-q_k}, we verify the equivalent expression for $q_k$ given in Line~\ref{algline:q_k} of Algorithm~\ref{alg:dual-extraction-framework}. 
In Section \ref{subapp:supporting-lemmas-well-defined}, we prove that Algorithm~\ref{alg:dual-extraction-framework} is well-defined.
In Section \ref{subapp:boosting-DRPOSP}, we show how to boost the success probability of a $\DRPOSP$ (Definition~\ref{def:DRPOSP}).

\subsection{Verifying the assumptions for the simplex setting (Definition~\ref{ex:simplex})}
\label{subapp:verify-simplex-setting}

In this section, we verify that the simplex setting of Definition~\ref{ex:simplex} satisfies Assumption~\ref{as:technical-dual} in Definition~\ref{def:dual-extraction-setup}. To begin, fix any $p \in \uset \cap \interior \pset = \simplex^n \cap \R_{>0}$. Letting $\ones \in \R^n$ denotes the vector consisting of all ones, we claim
\begin{align*}
    \partial \indc_\uset (p) = \inbraces{c \cdot \ones : c \in \R},
\end{align*}
which clearly suffices. The fact that $\inbraces{c \cdot \ones : c \in \R} \subseteq \partial \indc_\uset (p)$ is trivial. As for the other direction, fix any $z \notin \inbraces{c \cdot \ones : c \in \R}$, and let $j^\star \in \argmin_{j \in [n]} z_j$ denote some index where $z$ is smallest. Letting $e_{j^\star}$ denote the vector with a one in its $j^\star$-th index and zeros elsewhere, the fact that $\inangle{z, e_{j^\star} - p} < 0$ implies $z \notin \partial \indc_\uset (p)$.

\subsection{Equivalent expression for $q_k$ in Algorithm~\ref{alg:dual-extraction-framework}}
\label{subapp:equiv-expression-q_k}

The following lemma shows that the two expressions for $q_k$ in Line~\ref{algline:q_k} of Algorithm~\ref{alg:dual-extraction-framework} are equivalent.

\begin{lemma}[Equivalent expression for $q_k$]
With the notation of Line~\ref{algline:q_k} in Algorithm~\ref{algline:q_k}, we have
\begin{align*}
    \grad \rconjuset \inparen{\frac{1}{\Lambda_k} \sum_{j = 0}^{k - 1} \lambda_j \grad r(y_j)} = \argmin_{q \in \uset}\frac{1}{\Lambda_k} \sum_{j = 0}^{k - 1} \lambda_j \breg{}{y_j}{q}.
\end{align*}
\end{lemma}

\begin{proof}
A rearrangement of Equation~\ref{eq:q_k-opt-condition} from the proof of Lemma~\ref{lem:connecting-to-overview} implies
\begin{align*}
    \frac{1}{\Lambda_k} \sum_{j = 0}^{k - 1} \lambda_j \grad r(y_j)  \in \grad r(q_k) + \partial \indc_{\uset}(q_k) = \partial r_{\uset}(q_k),
\end{align*}
where we again applied Lemma~\ref{lem:subdiff-sum-of-funcs}. Then Lemma~\ref{lem:subdiff-properties-of-conjugate} implies
\begin{align*}
    q_k \in \partial \rconjuset \inparen{\frac{1}{\Lambda_k} \sum_{j = 0}^{k - 1} \lambda_j \grad r(y_j)}.
\end{align*}
To conclude, the strong convexity of $r_{\uset}$ implies the conjugate $\rconjuset$ is smooth (in the sense that its gradient is Lipschitz) over all of $\R^n$ by \cite[Thm. 6.11]{orabona2023modern} (see also \cite[Thm. 6]{kakade2009duality}). In particular, $\rconjuset$ is differentiable, so we conclude $\partial \rconjuset = \inbraces{\grad \rconjuset}$ by Lemma~\ref{lem:connect-subdiff-grad}.
\end{proof}

\subsection{Algorithm~\ref{alg:dual-extraction-framework} is well-defined}
\label{subapp:supporting-lemmas-well-defined}

In this section, we show that Algorithm~\ref{alg:dual-extraction-framework} is well-defined in Corollary~\ref{cor:main-framework-well-defined}; namely, whenever a Bregman divergence of the form $\breg{}{u}{w}$ is written in the pseudocode of Algorithm~\ref{alg:dual-extraction-framework}, it is the case that $u \in \uset \cap \interior \pset$. But first, in the following lemma, we argue that a function which is $\mu$-strongly convex relative to the dgf $r$ cannot be minimized on the boundary of $\pset$: 

\begin{lemma}[Minimizer is in the interior]
    \label{lem:min-in-int}
    Fix a $(\uset, \pset, \norm{\cdot}, r)$-dgf setup, and let $g : \pset \to \R$ be continuous on $\pset$ and differentiable on $\interior \pset$. For nonempty, closed, convex $\sset \subseteq \uset$ such that $\sset \cap \interior \pset \ne \emptyset$, suppose $g$ is $\mu$-strongly convex relative to the dgf $r$ over $\sset \cap \interior \pset$ for some $\mu > 0$. Then $\argmin_{q \in \sset} g(q) \in \sset \cap \interior \pset$.
\end{lemma}

\begin{proof}
Recalling Definition~\ref{def:dgf-setup}, the result is immediate if $\uset \subseteq \interior \pset$, so suppose instead \\ $\lim_{u \to \bd \pset} \norm{\grad r(u)}_2 = \infty$. Letting $z_\lambda^{(p, u)} \defeq p + \lambda(u - p)$ for $\lambda \in \R$ and $u, p \in \R^n$, the condition $\lim_{u \to \bd \pset} \norm{\grad r(u)}_2 = \infty$ is equivalent to the following by \cite[Lemma 26.2]{rockafellar1970textbook}:\footnote{Note that if $p \in \bd \pset$ and $u \in \interior \pset$, then $z_\lambda^{(p, u)} \in \interior \pset$ for $\lambda \in (0, 1]$. See, for example, \cite{keepfrog2020linesegmentinterior}.}
\begin{align}
    \label{eq:equivalent-Legendre}
    \lim_{\lambda \to 0} \inangle{\grad r(z_\lambda^{(p, u)}), p - u} = \infty, \quad \text{for all $p \in \bd \pset$ and $u \in \interior \pset$}.
\end{align}
By relative strong convexity, we have the following for all $p \in \sset \cap \bd \pset$, $u \in \sset \cap \interior \pset$, and $\lambda \in (0, 1)$ such that $z_{- \lambda}^{(u, p)} \in \sset \cap \interior \pset$:
\begin{align}
    \label{eq:well-defined-relative-strong-convexity}
    \inangle{\grad g(z_\lambda^{(p, u)}), p - u} &\ge \mu \inangle{\grad r(z_\lambda^{(p, u)}), p - u} - g(z_{- \lambda}^{(u, p)}) + g(z_\lambda^{(p, u)}) + \mu \insquare{ r(z_{- \lambda}^{(u, p)}) - r(z_\lambda^{(p, u)}) }.
\end{align}
By continuity, $r$ and $g$ are both bounded over any compact subset of $\pset$. As a result, the above implies:\footnote{If $u \in \sset \cap \interior \pset$ is such that $z_{- \lambda}^{(u, p)} \in \sset \cap \interior \pset$ for all sufficiently small $\lambda > 0$, then \eqref{eq:well-defined-result} is immediate from \eqref{eq:well-defined-relative-strong-convexity} and \eqref{eq:equivalent-Legendre}. Otherwise, \eqref{eq:well-defined-result} follows from considering \eqref{eq:well-defined-relative-strong-convexity} with the ``$u$'' in \eqref{eq:well-defined-relative-strong-convexity} chosen to be some point in the relative interior of the line segment between $p$ and $u$, and then applying \eqref{eq:equivalent-Legendre}.}
\begin{align}
    \label{eq:well-defined-result}
    \lim_{\lambda \to 0} \inangle{\grad g(z_\lambda^{(p, u)}), p - u} = \infty, \quad \text{for all $p \in \sset \cap \bd \pset$ and $u \in \sset \cap \interior \pset$}.
\end{align}
We now use ideas from the proof of \cite[Thm. 6.7]{orabona2023modern}. Letting $\qopt \defeq \argmin_{q \in \sset} g(q)$, suppose for the sake of contradiction that $\qopt \in \sset \cap \bd \pset$, and fix any $u \in \sset \cap \interior \pset$. Define the ``perturbation function'' $h : [0, 1) \to \R$ via $h(\lambda) \defeq g(z_\lambda^{(\qopt, u)})$, and note that for $\lambda \in (0, 1)$:
\begin{align*}
    h'(\lambda) = \inangle{\grad g (z_\lambda^{(\qopt, u)}), u - \qopt},
\end{align*}
in which case $\lim_{\lambda \to 0} h'(\lambda) = - \infty$. But then the fact that $h$ is continuous over $[0, 1)$ along with the mean value theorem implies $h(\delta) < h(0)$ for some $\delta \in (0, 1)$, a contradiction.
\end{proof}

Now for the main proposition:

\begin{corollary}[Framework is well-defined]
    \label{cor:main-framework-well-defined}
In the setting of Algorithm~\ref{alg:dual-extraction-framework}, we have $q_k \in \uset \cap \interior \pset$ and $y_k \in \yset \cap \interior \pset$ for all $k \in [K]$.
\end{corollary}

\begin{proof}
    Note that $y_0 \in \yset \cap \interior \pset$ by assumption. Consequently, the result  follows immediately by applying Lemma~\ref{lem:min-in-int} inductively with $\sset \gets \yset$ or $\sset \gets \uset$ as appropriate. (For the concave maximization problem in Line~\ref{algline:y_k}, one can flip the sign to obtain an equivalent convex minimization problem.)
\end{proof}

\subsection{Boosting the success probability of a $\DRPOSP$ oracle}
\label{subapp:boosting-DRPOSP}

The following lemma, stated in the context of Definition~\ref{def:DRPOSP}, boosts the success probability of a $\DRPOSP$ oracle:

\begin{lemma}[Boosting the success probability of a $\DRPOSP$ oracle]
    \label{lem:boost-DRPOSP}
With $q, \lambda, \epsprim$ as in Definition~\ref{def:DRPOSP} and given $\delta \in (0, p]$, the oracle call $\DRPOSP(q, \lambda, \epsprim, \delta)$ can be implemented with $N \defeq \ceil*{ \frac{\log_2(1 / \delta)}{\log_2(1 / p)} }$ calls to $\DRPOSP(q, \lambda, \epsprim, p)$, along with $N$ evaluations of the function $f_{\lambda, q}$. (A single evaluation means computing $f_{\lambda, q}(x)$ for a given $x \in \xset$.)
\end{lemma}

\begin{proof}
    We claim $\DRPOSP(q, \lambda, \epsprim, \delta)$ can be implemented by evaluating $\DRPOSP(q, \lambda, \epsprim, p)$ a total of $N$ times yielding outputs $x_1, \dots, x_N$, and then returning some element of $\argmin_{i \in [N]} f_{\lambda, q}(x_i)$. Indeed, letting $\event_i$ for $i \in [N]$ denote the event that the $i$-th call to $\DRPOSP(q, \lambda, \epsprim, p)$ succeeds (meaning $\E \insquare{f_{\lambda, q}(x_i) \mid \event_i} \le f_{\lambda, q}(x) + \epsprim$ for all $x \in \xset$), we have
    \begin{align*}
        \E \insquare{\min_{i \in [N]} f_{\lambda, q}(x_i) \mid \cup_{i \in [N]} \event_i}
        \le f_{\lambda, q}(x) + \epsprim
    \end{align*}
    for all $x \in \xset$. Furthermore, 
    \begin{align*}
        \P \insquare{\cup_{i \in [N]} \event_i} \ge 1 - p^{N} \ge 1 - \delta.
    \end{align*}
\end{proof} 
\section{Deferred proofs from Section \ref{sec:maximin-algorithms}}
\label{app:maximin-proofs}

In Sections \ref{subapp:deferred-matrix-games-proofs} and \ref{subapp:alpha-CVaR}, we give deferred proofs from Sections \ref{subsec:matrix-games} and \ref{subsec:alpha-CVaR} respectively.

\subsection{Deferred proofs from Section \ref{subsec:matrix-games}}
\label{subapp:deferred-matrix-games-proofs}

Our goal in Appendix \ref{subapp:deferred-matrix-games-proofs} is to prove Lemma~\ref{lem:DRPOSP-matrix-games}, which shows that the algorithmic runtime guarantee \cite[Cor. 8.2]{carmon2023whole} for obtaining an expected $\epsilon$-optimal point for \eqref{eq:primal-MG} can be extended to the setting where additional dual regularization is added. Intuitively, additional dual regularization should only stabilize the primal problem and make it ``easier,'' but some care is still required to modify the algorithm and analysis of \cite{carmon2023whole} to handle this extension. In Section \ref{subsubapp:whole-new-ball-game-mod}, we modify an algorithm of \cite{carmon2023whole}, for a more general setting than that of Definitions \ref{def:mat-games-ball-in-body} and \ref{def:mat-games-simplex-in-body}, to handle added regularization. We then use this to prove Lemma~\ref{lem:DRPOSP-matrix-games} in Section \ref{subsubapp:proof-of-DRPOSP-for-mat-games}.

\subsubsection{Modifying an algorithm of \cite{carmon2023whole} to handle additional regularization}
\label{subsubapp:whole-new-ball-game-mod}

In this section, we modify the result of \cite[Sec. 8.1]{carmon2023whole} to handle additional regularization. Though the setting of this section can be viewed as a generalization of \eqref{eq:primal-MG}, we note that Section \ref{subsubapp:whole-new-ball-game-mod} is entirely separate from the context of Section \ref{subsec:matrix-games}. Indeed, Section \ref{subsubapp:whole-new-ball-game-mod} only consists of modifying the results of \cite{carmon2023whole} to obtain a form appropriate for our application in Section \ref{subsubapp:proof-of-DRPOSP-for-mat-games}.

With that said, we consider the following optimization problem in this section for differentiable convex functions $f_i : \R^d \to \R$ and $\Lambda > 0$ (the equivalence can be derived from Lemma~\ref{lem:convex-conjugate-neg-entropy}):

\begin{align}
    \label{eq:matrix-reg-obj}
    \minimize_{x \in \xset} \inbraces{
        f(x) \defeq \max_{y \in \simplex^n} \inbraces{ \sum_{i = 1}^n [ y_i f_i(x) - \Lambda \cdot y_i \ln y_i] } = \Lambda \ln \inparen{
            \sum_{i = 1}^n \exp \inparen{\frac{f_i(x)}{\Lambda}}
        }
	}.
\end{align}

We consider the following setups:

\begin{definition}[General ball setup, Definition 8.1 in \cite{carmon2023whole}]
	\label{def:mat-games-ball-setup-appendix} 
	In the \emph{general ball setup}, we use the Euclidean norm $\norm{\cdot}_2$, the domain $\xset$ is a closed and convex subset of the unit Euclidean ball $B^d$, and the Bregman divergence on $\R^d$ is $\widetilde{V}_x(y) = \half \norm{y-x}_2^2$. Furthermore, we let $\xset_\nu \defeq \xset$ for all $\nu \ge 0$ and set $p \defeq 2$.
\end{definition}

\begin{definition}[General simplex setup, Definition 8.2 in \cite{carmon2023whole}]
	\label{def:mat-games-simplex-setup-appendix} 
	In the \emph{general simplex setup}, we use the $\ell_1$-norm $\norm{\cdot}_1$, the domain $\xset$ is a closed and convex subset of the probability simplex $\Delta^d$, and the Bregman divergence on $\simplex^d$ is $\widetilde{V}_x(y) = \sum_{i = 1}^d y_i \log \frac{y_i}{x_i}$. Furthermore, we let $\xset_\nu \defeq \{x\in\xset \mid x_i \ge \nu, ~ \forall i\in[d]\}$ for all $\nu\ge0$, and set $p \defeq 1$.
\end{definition}

For the rest of this section, unless we explicitly state otherwise, all results and equations hold for the setups of Definitions \ref{def:mat-games-ball-setup-appendix} and \ref{def:mat-games-simplex-setup-appendix} simultaneously. The constant $p$ in Definitions \ref{def:mat-games-ball-setup-appendix} and \ref{def:mat-games-simplex-setup-appendix} is used to, e.g., simultaneously refer to both norms via $\norm{\cdot}_p$ when we state results. We use the notation $\bregtilde{}{u}{w}$ in Definitions \ref{def:mat-games-ball-setup-appendix} and \ref{def:mat-games-simplex-setup-appendix} as opposed to the usual $\breg{}{u}{w}$ for the Bregman divergence because the Bregman divergence in Definitions \ref{def:mat-games-ball-setup-appendix} and \ref{def:mat-games-simplex-setup-appendix} is being defined over the primal space (a subset of $\R^d$), whereas in the main body of the paper, the Bregman divergence is always defined over the dual space (a subset of $\R^n$). This helps avoid confusion when we apply the result of Section \ref{subsubapp:whole-new-ball-game-mod} in Section \ref{subsubapp:proof-of-DRPOSP-for-mat-games}.
We also note here that 
\begin{align}
    \label{eq:more-reg-softmax-grad}
    \grad f(x) = \sum_{i = 1}^n p_i(x) \grad f_i(x), \; \; \text{ where } \;\; p_i(x) = \frac{\exp(f_i(x) / \Lambda)}{\sum_{i = 1}^n \exp (f_i(x) / \Lambda)}.
\end{align}

When the functions $f_i$ are additionally $L_f$-Lipschitz and $L_g$-smooth with respect to $\norm{\cdot}_p$, \cite[Sec. 8.1]{carmon2023whole} provides an algorithm which obtains an $\epsilon$-optimum for \eqref{eq:matrix-reg-obj} when $\Lambda = \Theta( \frac{\epsilon}{\log n})$ (see the proof of Theorem~8.1 in that paper), in which case $f$ is $O(\epsilon)$ additively close to $x \mapsto \max_{y \in \simplex^n} \sum_{i = 1}^n y_i f_i(x)$ pointwise over $\xset$. Indeed, the goal of \cite[Sec. 8.1]{carmon2023whole} is actually to obtain an $\epsilon$-optimum for the optimization problem $\minimize_{x \in \xset} \inbraces{\max_{y \in \simplex^n} \sum_{i = 1}^n y_i f_i(x)}$, and this is done via the proxy objective \eqref{eq:matrix-reg-obj} for sufficiently small $\Lambda$. 
In our application however, we need to solve \eqref{eq:matrix-reg-obj} for much larger $\Lambda$. Intuitively, increasing $\Lambda$ should only make \eqref{eq:matrix-reg-obj} easier since it increases the stability of the dual variables. Still, we detail some minor but necessary modifications to the results of \cite{carmon2023whole} below since \cite{carmon2023whole} tightly couples $\Lambda$ to the target accuracy $\epsilon$.

The only place in \cite{carmon2023whole} where they use the fact that the objective $f$ takes the form \eqref{eq:matrix-reg-obj} (as opposed to $f$ being a black-box convex function which we are trying to minimize) is in Section 7 where they provide an efficient estimator for the gradient \eqref{eq:more-reg-softmax-grad}. Specifically, Algorithm 5 in that section takes as input a sequence of points $(x_t)_{t = 0}^T$ such that for some $r, r' > 0$, we have $\norm{x_t - x_0}_p \le r$ for $t \in [T]$ and $\sum_{t \le T} \norm{x_t - x_{t - 1}}_p \le r'$ (omitting a few other conditions for now). It then (roughly speaking) outputs stochastic gradient estimates for $f$ at the points $(x_t)_{t = 0}^T$. The reason we cannot use Algorithm 5 and the accompanying guarantee Theorem~7.1 as written is that Theorem~7.1 makes the restriction $\Lambda \le L_f r' / 2$ (their notation for $\Lambda$ is $\epsilon'$), where $L_f$ is the Lipschitz constant of $f$ with respect to $\norm{\cdot}_p$. This restriction is made because Line 1 of Algorithm 5 calls their matrix-vector maintenance data structure (see Definition 6.1 in that paper) with target accuracy $\Lambda / L_f$ and total movement bound $r'$, and the data structure enforces a relationship between these two quantities which becomes $\Lambda \le L_f r' / 2$ in the context of Theorem 7.1. They call the matrix-vector maintenance data structure with target error $\Lambda / L_f$ in Algorithm 5 because this is the largest error which suffices to ensure the correctness of Theorem~7.1 (and thereby minimizes the additional runtime of the matrix-vector maintenance data structure). In our case however where $\Lambda$ may be much larger than $\Theta(\frac{\epsilon}{\log n})$, we need to decouple the accuracy to which we call the matrix-vector maintenance data structure and $\Lambda$ to avoid breaking the aforementioned relationship imposed by Definition 6.1, since when we call our modified version of Algorithm 5, $r'$ will be the same as in \cite{carmon2023whole} but $\Lambda$ may be much larger.

Our modified version of Algorithm 5 from \cite{carmon2023whole} is given in Algorithm~\ref{alg:grad-est} below, with the guarantee given in Theorem~\ref{thm:grad-est}. Again, the main change is that we are initializing the matrix-vector data structure in Line~\ref{line:init-data-structure} of Algorithm~\ref{alg:grad-est} at a higher accuracy than we may need, so as to ensure that the restriction on $r'$ is satisfied.

\newcommand{\Lf}{L_f}
\newcommand{\Lg}{L_g}
\newcommand{\gest}{\mathcal{G}}
\newcommand{\filt}[1][t]{\mathcal{F}_{#1}}
\newcommand{\dsStyle}[1]{\textsc{#1}}
\newcommand{\dsInit}{\dsStyle{init}}
\SetKwInput{KwInput}{Input}                
\SetKwInput{KwParameters}{Parameters}   
\newcommand{\False}{\textup{\textsf{False}}}
\newcommand{\True}{\textup{\textsf{True}}}
\newcommand{\dsQuery}{\dsStyle{query}}

\begin{theorem}[Softmax gradient estimator]\label{thm:grad-est}
    
	Let $\epsilon'$ be such that $0 < \epsilon' \le \Lambda$, let $p\in\{1,2\}$, and let $\{ f_i \}_{i\in[n]}$ be $\Lg$-smooth and $\Lf$-Lipschitz with respect to
	$\norm{\cdot}_p$.
	For all $t\in [T]$, assume that input $x_t$ to \Cref{alg:grad-est} is a (deterministic)  function of the previous outputs $\gest(x_1),\ldots, \gest(x_{t-1})$, and that $\norm{x_t-x_0}_p\le r$ and $\sum_{t \le T} \norm{x_t - x_{t-1}}_p \le r'$ hold for parameters $r,r'>0$ such that\footnote{It is only actually necessary for $\frac{1}{2} L_g r^2 \le \Lambda$ for the theorem to go through, but the difference won't matter since we will only call this theorem with the same value of $r$ used in \cite{carmon2023whole}.} $\half \Lg r^2 \le \eps'$ and $\eps' \le \Lf r' / 2$. Let $\filt$ be the filtration induced by all the random bits \Cref{alg:grad-est} draws up to iteration $t$ and all those that may be used by $\mathcal{M}$.
	Then for any error tolerance $\delta \in (0,1)$ there exists event $\mc{E}$ such that the following hold:
	\begin{itemize}
	\item We have $\P(\mc{E}) \ge 1-\delta$.
	\item When $\mc{E}$ holds we have $\Ex*{\gest(x_t)|\filt[t-1]} = \grad f(x_t)$ for all $t\in[T]$.
	\item When $\mc{E}$ holds, \Cref{alg:grad-est} makes $O(n+T\log(1/\delta))$ queries of the form $\{f_i(x), \nabla f_i(x)\}$, and requires additional runtime
	\[O\left(T\bigg(d+\log\bigg(\frac{1}{\delta}\bigg)\bigg)+\left(nd\log^{p-1}\bigg(\frac{\Lf r'}{\eps'}\bigg)+d\bigg(\frac{\Lf r'}{\eps'}\bigg)\right)\log^{p-1}\bigg(\frac{n\Lf r'}{\eps'\delta}\bigg)+n\bigg(\frac{\Lf r'}{\eps'}\bigg)^2\log\frac{n\Lf r'}{\eps'\delta}\right).\]
	\item With probability 1 we have $\norm{\gest(x_t)}_{p^\star} \le \Lf$, where $p^*$ is such that $\frac{1}{p} + \frac{1}{p^*}=1$.
	\end{itemize}
\end{theorem}

\begin{algorithm}[t]
	\DontPrintSemicolon
	\caption{Gradient estimator for \eqref{eq:matrix-reg-obj}}
	\label{alg:grad-est}
	\KwInput{$\{f_i\}_{i\in[n]}$, query sequence $\{x_t\}_{t\le T}$ such that $x_t$ is a function of the previous outputs $\gest(x_1), \ldots, \gest(x_{t-1})$ (i.e., $x_0$ and $x_1$ do not depend on any outputs).}
	\KwParameters{Regularization level $\Lambda$, matrix-vector maintenance data structure accuracy $\epsilon'$, movement bound $r'$, Lipschitz constant $\Lf$, error tolerance $\delta \in (0,1)$, $\ell_p$-matrix-vector maintenance data structure $\mc{M}$.}
	Call 
	$\mathcal{M}.\dsInit(A, 0, r', \frac{\eps'}{\Lf},\frac{\delta}{2})$ where $A = [\frac{1}{\Lf}\nabla f_i(x_0)^\top ]_{i\in[n]}$ \label{line:init-data-structure}\;
	\For{$t=1,2,\cdots, T$}{
	$y_t\gets \Lf \cdot \mathcal{M}.\dsQuery(x_t-x_{t-1})$ \Comment{maintain vector $y_t\approx \Lf A (x_t-x_0) = \brk*{\inner{\grad f_i(x_0)}{x_t-x_0}}_{i\in[n]}$} 
	\textsf{accepted} $\gets \False$\;
	\While{\textup{not \textsf{accepted}}}{
	Draw $i\sim \exp\prn*{\frac{f_i(x_0)+[y_t]_i}{\Lambda}}$ \label{line:draw}\;
	\Block{With probability $\min\crl*{\exp\prn*{\frac{f_i(x_t)-f_i(x_0)-[y_t]_i}{\Lambda}-2}, 1}$\label{line:rejection-step}}{\textbf{yield} $i_t = i$ and $\gest(x_t) = \nabla f_{i_t}(x_t)$\;
	\textsf{accepted} $\gets \True$\;
	}
	}
	}
\end{algorithm}

\begin{proof}
The proof follows exactly as in the proof of Theorem~7.1 in \cite{carmon2023whole}. Indeed, the only difference between Algorithm~\ref{alg:grad-est} above and Algorithm 5 in \cite{carmon2023whole} is that we have replaced $\epsilon'$ with $\Lambda$ in Lines \ref{line:draw} and \ref{line:rejection-step} to account for the fact that the gradient is now given by \eqref{eq:more-reg-softmax-grad}. In particular, it is straightforward to check that the same rejection sampling analysis goes through since $\epsilon' \le \Lambda$, so we can only have initialized the matrix-vector maintenance data structure in Line~\ref{line:init-data-structure} to more than the required accuracy to make the analysis go through.
\end{proof}

Having made this modification to the estimator for \eqref{eq:more-reg-softmax-grad} so as to handle $\Lambda \gg \epsilon / \log n$, we now state our guarantee for the problem \eqref{eq:matrix-reg-obj}, mirroring Theorem~8.1 in \cite{carmon2023whole} which, recall, gives an algorithm for \eqref{eq:matrix-reg-obj} with the same runtime but only when $\Lambda = \Theta (\frac{\epsilon}{\log n})$. For simplicity (e.g., to avoid conditions on $\epsilon$), we only state the guarantee when each $f_i$ is $0$-smooth (equivalently, an affine function), since that is the only case we will need for the matrix games application in Section \ref{subsubapp:proof-of-DRPOSP-for-mat-games}.

\newcommand{\Teval}{\mathcal{T}_{\textup{eval}}}
\newcommand{\Tmd}{\mathcal{T}_{\textup{md}}}

\begin{theorem}
	\label{thm:lambda-reg-softmax}
Let $\epsilon, \Lambda > 0$, and consider the optimization problem \eqref{eq:matrix-reg-obj} in either the general ball or general simplex setup (Definitions \ref{def:mat-games-ball-setup-appendix} and \ref{def:mat-games-simplex-setup-appendix} respectively), where each function $f_i : \R^d \to \R$ is convex, $L_f$-Lipschitz with respect to $\norm{\cdot}_p$, and $0$-smooth. Then there exists an algorithm which returns a point $x \in \xset$ such that 
\begin{align*}
    \E f(x) - \min_{\xopt \in \xset} f(\xopt) \le \epsilon,
\end{align*}
and with probability at least $9/10$, the algorithm has runtime 
\begin{align*}
    \Otilde \inparen{n (\Teval + d) + n \inparen{\frac{(\Teval + d) L_f R}{\epsilon}}^{2/3} + (\Teval + \Tmd + d) \frac{L_f^2 R^2}{\epsilon^2}}.
\end{align*}
Here, $\Teval$ is the time to compute $f_i(x),\grad f_i(x)$ for any $x\in\xset$ and $i\in[n]$, and setting $\nu = \frac{\epsilon}{4d L_f}$,
we let $\Tmd$ denote the time to compute a mirror descent step of the form $\argmin_{z\in\xset_\nu} \crl*{\inner{g}{z} + \lambda \widetilde{V}_{y}(z)+ \widetilde{V}_{x}(z)}$ for any $g\in\R^d$ and $x,y\in\xset$. $R > 0$ is such that for initial point $x_0 \in \xset_\nu$, we have $\max_{x \in \xset_\nu} \bregtilde{}{x_0}{x} \le \frac{1}{2} R^2$.
\end{theorem}

\newcommand{\fsm}{f_{\mathrm{smax}}}
\newcommand{\Tmax}{T_{\mathrm{outer}}}
\newcommand{\Tinner}{T_{\mathrm{inner}}}

\begin{proof}
In summary, we cite the algorithm of Theorem~8.1 in \cite{carmon2023whole} (with almost the same choice of parameters except simplified slightly since $L_g = 0$), except we substitute Algorithm 5 and the corresponding guarantee Theorem~7.1 in that paper with Algorithm~\ref{alg:grad-est} and Theorem~\ref{thm:grad-est} above, which are designed to handle additional regularization.

As for the details, per the proof of Theorem~8.1 in \cite{carmon2023whole}, their algorithm actually\footnote{``Actually'' because as discussed earlier in Section \ref{subapp:deferred-matrix-games-proofs}, this objective is a proxy objective for their original goal of solving $\minimize_{x \in \xset} \inbraces{\max_{y \in \simplex^n} \sum_{i = 1}^n y_i f_i(x)}$.} obtains an $\epsilon / 4$-minimizer of 
\begin{align*}
    \minimize_{x \in \xset_\nu} \inbraces{ \epsilon' \ln \inparen{\sum_{i = 1}^n \exp \inparen{\frac{f_i(x)}{\epsilon'}}}},
\end{align*}
where $\epsilon' \defeq \frac{\epsilon}{2 \log n}$. (Note that the above minimization is over the set $\xset_\nu$ as opposed to $\xset$. Recall from Definitions \ref{def:mat-games-ball-setup-appendix} and \ref{def:mat-games-simplex-setup-appendix} that $\xset$ and $\xset_\nu$ only differ in the simplex setup, where this truncation is necessary for technical reasons to establish a relaxed triangle inequality for the Bregman divergence.) Then, as alluded to earlier in this section, we can instead run the same algorithm with the same choices of parameters (except for $r$ which we will keep general for now but set in a moment), except substituting Algorithm~\ref{alg:grad-est} and Theorem~\ref{thm:grad-est} above (with\footnote{Theorem~\ref{thm:grad-est} restricts $\Lambda \ge \epsilon'$, but we can assume $\Lambda \ge \frac{\epsilon}{2 \log n}$ without loss of generality since if $\Lambda < \frac{\epsilon}{2 \log n}$, we can increase $\Lambda$ in the objective \eqref{eq:matrix-reg-obj} to $\frac{\epsilon}{2 \log n}$ while maintaining an $O(\epsilon)$ additive pointwise approximation to the original objective.} $\epsilon' \defeq \frac{\epsilon}{2 \log n}$) in place of Algorithm 5 and Theorem~7.1 from \cite{carmon2023whole} respectively. One can check that everything goes through upon making this substitution.

Then per the end of the proof of Theorem~8.1 in \cite{carmon2023whole}, with probability at least $\frac{9}{10}$, the algorithm obtains an $\epsilon / 4$-minimizer of
\begin{align}
    \label{eq:lambda-obj-over-xset_nu}
    \minimize_{x \in \xset_\nu} \inbraces{ \Lambda \ln \inparen{\sum_{i = 1}^n \exp \inparen{\frac{f_i(x)}{\Lambda}}}}
\end{align}
with total runtime
\begin{align}
    \label{eq:runtime-before-choosing-r}
	\widetilde{O}\left(n\frac{\Lf^2R^{2/3}r^{4/3}}{\eps^2}+n(\Teval+d)\frac{R^{2/3}}{r^{2/3}}+(\Teval+\Tmd+d)\frac{\Lf^2R^2}{\eps^{2}}\right).
\end{align}
We have yet to choose $r$, which is subject to the conditions $r \le R$ (see Theorem~4.1 in \cite{carmon2023whole}) and $\epsilon' \le L_f r' / 2$ due to Theorem~\ref{thm:grad-est} above. (Note that the condition $\frac{1}{2} L_g r^2 \le \epsilon'$ in Theorem~\ref{thm:grad-est} is trivial due to our simplifying assumption $L_g = 0$.) Here, $r' = \Otilde(r)$ since up to choosing $r$, we are making the same parameter choices as in Theorem~8.1 in \cite{carmon2023whole} (see the proof of Theorem~8.1 in that paper). Thus, we can choose $r = \min \inbraces{R, \Thetatilde \inparen{ \frac{\epsilon \sqrt{\Teval + d}}{L_f}}}$ to obtain the stated runtime upper bound. (We use $\Thetatilde$ here in case $r$ needs to be increased by logarithmic factors to ensure $\epsilon' \le L_f r' / 2$ holds in all regimes.)

Finally, note that \eqref{eq:lambda-obj-over-xset_nu} and our original objective \eqref{eq:matrix-reg-obj} differ since the former is over $\xset_\nu$ instead of $\xset$. However, it is clear from \eqref{eq:more-reg-softmax-grad} that $f$ is $L_f$-Lipschitz, in which case we have
\begin{align*}
    \min_{x \in \xset} f(x) \le \min_{x \in \xset_\nu} f(x) \le \min_{x \in \xset} f(x) + \epsilon / 4.
\end{align*}
by the choice of $\nu$. Thus, an $\epsilon / 4$-minimizer of \eqref{eq:lambda-obj-over-xset_nu} is an $\epsilon/2$-minimizer of \eqref{eq:matrix-reg-obj}, yielding the result.
\end{proof}

\subsubsection{Proof of Lemma~\ref{lem:DRPOSP-matrix-games}}
\label{subsubapp:proof-of-DRPOSP-for-mat-games}

We now give the proof of Lemma 6, restated here for convenience.

\restateLemDRPOSPMatrixGames*

\begin{proof}
We have
\begin{align*}
	f_{\lambda, q}(x) &= \max_{y \in \simplex^n} \inbraces{x^\top A y - \lambda \sum_{i = 1}^n y_i \ln \frac{y_i}{q_i}} \\
	&= \max_{y \in \simplex^n} \sum_{i = 1}^n \inparen{y_i \cdot [A^\top x]_i + \lambda \cdot y_i \ln q_i - \lambda \cdot y_i \ln y_i} \\
	&= \max_{y \in \simplex^n} \sum_{i = 1}^n \inparen{y_i f_i(x) - \lambda \cdot y_i \ln y_i},
\end{align*}
where we defined $f_i(x) \defeq [A^\top x]_i + \lambda \cdot \ln q_i$. Having massaged the objective to obtain an instance of \eqref{eq:matrix-reg-obj}, we apply Theorem~\ref{thm:lambda-reg-softmax} to obtain the desired runtime for computing an expected $\epsprim$-optimal point of $f_{\lambda, q}$ with success probability at least $9/10$.  Indeed, recall from the discussion after Definitions \ref{def:mat-games-ball-in-body} and \ref{def:mat-games-simplex-in-body} in Section \ref{subsec:matrix-games} that $f_i$ is 1-Lipschitz ($L_f = 1$). Otherwise, it is straightforward to check that we can bound $R = \Otilde(1)$ (setting $x_0 = 0$ in the case of Definition~\ref{def:mat-games-ball-in-body} or $x_0 = \frac{1}{d} \ones$ in the case of Definition~\ref{def:mat-games-simplex-in-body}), $\Tmd = \Otilde(d)$, and $\Teval = O(d)$; see the proof of Corollary 8.2 in \cite{carmon2023whole} for details.
\end{proof}

\subsection{Deferred proofs from Section \ref{subsec:alpha-CVaR}}
\label{subapp:alpha-CVaR}

Here we give the proof of Lemma~\ref{lem:DRPO-CVaR}, restated below for convenience. Recall the notation $\poly(\cdots)$, which we also use here, is defined in Section \ref{subsec:alpha-CVaR}.
As for why Lemma~\ref{lem:DRPO-CVaR} requires the conditions $\max_{i \in [n]} q_i^{-1} \le \allpoly$ and $\lambda \le \allpoly$, the proof involves reexpressing $f_{\lambda, q}$ in the form of the general DRO objectives considered \cite{levy2020largescale}. In particular, to apply the results of \cite{levy2020largescale}, it is necessary that the regularization is with respect to the uniform distribution as opposed to an arbitrary distribution $q$. This can be achieved by pushing the dependence on $q$ in $f_{\lambda, q}$ into the loss functions, thereby ending up with a new set of loss functions $\ftilde_i : \R^d \to \R$ for $i \in [n]$, where each $\ftilde_i$ depends on $q_i$ and $\lambda$.
 However, the guarantee in \cite{levy2020largescale} we then apply depends polylogarithmically on $M' > 0$ such that $\ftilde_i(x) \in [0, M']$ for all $x \in \xset$ and $i \in [n]$. Thus, the assumptions $\max_{i \in [n]} q_i^{-1} \le \allpoly$ and $\lambda \le \allpoly$ are used to ensure we can bound $M' \le \poly(\cdots)$. (In fact, the condition $\max_{i \in [n]} q_i^{-1} \le \allpoly$ can be relaxed further.)

\restateLemDRPOCVaR*

\begin{proof}
Note that we can write
\begin{align}
	f_{\lambda, q}(x) &= \max_{y \in \ytrunc} \inbraces{ \sum_{i = 1}^n y_i f_i(x) - \lambda \sum_{i = 1}^n y_i \ln \frac{y_i}{q_i} }  \nonumber \\
	&= \max_{y \in \ytrunc} \inbraces{ \sum_{i = 1}^n y_i \ftilde_i(x) - \lambda \sum_{i = 1}^n y_i \ln y_i } , \label{eq:CVaR-reg-obj}
\end{align}
where $\ftilde_i(x) \defeq f_i(x) + \lambda \ln q_i$. We now apply the multilevel Monte Carlo (MLMC) gradient estimator scheme designed in \cite{levy2020largescale}. Per Appendix A.1 in that paper, \cite{levy2020largescale} considers objectives of the form
\begin{align}
	\label{eq:large-scale-obj}
	\L(x) \defeq \sup_{w \in \simplex^n, D_{h_1}(w, \frac{1}{n} \ones) \le \rho} \inbraces{\sum_{i = 1}^n w_i \ell_i(x) - \gamma D_{h_2}(w, \frac{1}{n} \ones)}
\end{align}
for $\rho, \gamma \ge 0$, convex functions $\ell_i : \R^d \to \R$ which satisfy the assumptions of Section 2 in that paper, and closed convex functions $h_1, h_2$ satisfying $h_1(1) = h_2(1) = 0$. Here, for a convex function $h : \R_{\ge 0} \to (- \infty, \infty]$ satisfying $h(1) = 0$, we define
\begin{align*}
	D_h(w, \frac{1}{n} \ones) \defeq \frac{1}{n} \sum_{i = 1}^n h(n w_i)
\end{align*}
for $w \in \simplex^n$. (See Section 2 in \cite{levy2020largescale}.)
We note that \eqref{eq:large-scale-obj} has been specialized to our application. Namely, it is a finite-sum version of the corresponding Equation 18 in Appendix A.1 of \cite{levy2020largescale}. We have also set $P$ in Equation 18 in that paper to be $\frac{1}{n} \ones$.

Note that \eqref{eq:CVaR-reg-obj} is an instantiation of \eqref{eq:large-scale-obj} with $h_1 \equiv 0$ (i.e., it is identically zero), $h_2(t) \defeq \indc_{[\epsilon / (4M), 1 / \alpha]} + t \ln t - t + 1$, $\ell_i \defeq \ftilde_i$, and $\gamma$ set appropriately. (The value of $\gamma$ will not matter in the rest of the proof; we will be able to obtain the same complexity for any $\gamma > 0$.) Next, note that treating \eqref{eq:CVaR-reg-obj} as an instantiation of \eqref{eq:large-scale-obj} via these parameter choices, \eqref{eq:CVaR-reg-obj} is a $(\alpha^{-1} - 1)$-$\chi^2$-bounded objective per Definition 1 in Appendix A.4 of \cite{levy2020largescale}.

To go into greater detail, while the goal of \cite{levy2020largescale} is to obtain expected $\epsilon$-optimal points of objectives of the form \eqref{eq:large-scale-obj} (namely, minimizing $\mathcal{L}$ over a compact, convex constraint set), their MLMC gradient estimator is an unbiased estimator of the gradient of a surrogate batch objective for \eqref{eq:large-scale-obj}, that being, for $n' \in \N$:
\begin{align}
	\label{eq:general-batch-surrogate}
	\barL (x; n') \defeq \E_{s_1, \dots, s_{n'} \simiid  \uniform[n]}
	\insquare{ \max_{w \in \simplex^n, D_{h_1}(w, \frac{1}{n'} \ones) \le \rho} \inbraces{\sum_{i = 1}^{n'} w_i \ell_{s_i}(x) - \gamma D_{h_2}(w, \frac{1}{n'} \ones)} },
\end{align}
where $\uniform[n]$ denotes the uniform distribution over $\inbraces{1, 2, \dots, n}$; see Equation 7 in \cite{levy2020largescale} (we have replaced their ``$n$'' with $n'$ since $n$ is already reserved in this paper for the number of loss functions). This is done to achieve a complexity independent of $n$ when $n \gg \alpha \epsprim^{-2}$. (Indeed, \cite{levy2020largescale} also handles continuous setups where, informally, $n \to \infty$.) To go into more detail behind the MLMC method of \cite{levy2020largescale} for approximately minimizing $\L$ over a convex, compact constraint set, \cite[Prop. 1]{levy2020largescale} yields a bound on the bias $\abs{\L(x) - \barL (x; n')}$ uniformly over all $x$ in the constraint set, and \cite[Prop. 4]{levy2020largescale} (see also the more formal version Proposition 4' in Appendix C.1 in that paper) yields an MLMC estimator which is an unbiased estimate of $\grad \Lbar(x; n')$ with bounded second moment. Standard stochastic gradient method guarantees for approximately optimizing $\Lbar(x)$ can then be applied \cite[Prop. 3]{levy2020largescale}.

We now apply this outline to our specific instantiation \eqref{eq:CVaR-reg-obj} of the more general \eqref{eq:large-scale-obj}. Let $\fbar_{\lambda, q}(x; n')$ denote the particularization of \eqref{eq:general-batch-surrogate} to the instantiation \eqref{eq:CVaR-reg-obj} of \eqref{eq:large-scale-obj}, making the same parameter choices as before: $h_1 \equiv 0$, $h_2(t) \defeq \indc_{[\epsilon / (4M), 1 / \alpha]} + t \ln t - t + 1$, $\ell_i \defeq \ftilde_i$, and $\gamma$ set appropriately. Then, using the fact that \eqref{eq:CVaR-reg-obj} is $(\alpha^{-1} - 1)-\chi^2$-bounded as mentioned above, Claim 2' and Proposition 4' in Appendix C.1 of \cite{levy2020largescale} yield an estimator $\tildeM[\grad f_{\lambda, q}]$, parameterized by a choice of $n' \in \N$ (we set the parameter $n_0$ in \cite[Sec. 4]{levy2020largescale} to be 2), such that
\begin{align}
	\E \tildeM[\grad f_{\lambda, q}] &= \grad \fbar_{\lambda, q}(x; n'), \nonumber \\
	\E \norm{\tildeM[\grad f_{\lambda, q}]}_2^2 &= O \inparen{G^2 \alpha^{-1} \log_2 n'}, \label{eq:MLMC-moment-bound}
\end{align}
and the estimator requires an expected number of first-order queries bounded by $O(\log_2 n')$. (Following \cite[Sec. 2]{levy2020largescale}, we use $\grad$ here to denote some subgradient.) Also, we can bound the bias $\abs{\fbar_{\lambda, q}(x; n') -f_{\lambda, q}(x) } \le O(M' \sqrt{\alpha^{-1} (n')^{-1} \log n'})$ per Proposition 1 in Appendix B.1.1 in \cite{levy2020largescale} and the comment under it about the bound (9) in that proposition holding for any $\chi^2$-bounded objective. Here, $M'$ is such that $\ftilde_i(x) \in [0, M']$ for all $i \in [n]$ and $x \in \xset$. By adding an appropriate uniform quantity to every $\ftilde_i(x)$ so that they are all nonnegative (for any $\delta > 0$, this doesn't change the set of $\delta$-minimizers of $f_{\lambda, q}$), and then using the upper bounds on $q_i^{-1}$ for $i \in [n]$ and $\lambda$ in the statement of Lemma~\ref{lem:DRPO-CVaR}, as well as the fact that $f_i(x) \in [0, M]$ for all $i \in [n]$ and $x \in \xset$ per Definition~\ref{def:CVaR-setup}, we can bound $M' \le \allpoly$.

Now set $n' \gets \Otilde(\alpha^{-1} \epsprim^{-2} (M')^2)$, so that $\abs{\fbar_{\lambda, q}(x; n') -f_{\lambda, q}(x) } \le O(M' \sqrt{\alpha^{-1} (n')^{-1} \log n'}) \le \epsprim / 4$. Then with this choice of $n'$, \eqref{eq:MLMC-moment-bound} becomes $\E \norm{\tildeM[\grad f_{\lambda, q}]}_2^2 \le \Otilde(G^2 \alpha^{-1})$. As a result, per \cite[Prop. 3]{levy2020largescale}, there exists a stochastic gradient method with gradient estimator given by $\tildeM[\grad f_{\lambda, q}]$ which obtains an expected $\epsprim$-optimal point of $f_{\lambda, q}$ over $\xset$ with complexity $\Otilde(G^2 R^2 \alpha^{-1} \epsprim^{-2})$.
\end{proof}

\section{Deferred proofs from Section \ref{sec:convex-critical-point}}
\label{app:deferred-convex-critical-proofs}

In this section, we give additional lemmas referenced in Section \ref{sec:convex-critical-point}. Every lemma in this section is stated independently of the setup of Section \ref{sec:convex-critical-point}, although we may suggestively mirror the notation of Section \ref{sec:convex-critical-point}.
To start, the following lemma summarizes some relevant properties of a general regularized function:

\begin{lemma}[Relevant properties of a regularized function]
    \label{lem:general-reg-func-properties}
    With $\norm{\cdot}$ denoting the Euclidean norm on $\R^n$, $h : \R^n \to \R$ any convex differentiable function, $x_0 \in \R^n$, and $\alpha > 0$, define
\begin{align*}
    f(x) \defeq h(x) + \frac{\alpha}{2} \norm{x - x_0}^2.
\end{align*}
Let $\fopt$ denote the global minimum of $f$, and suppose $\gap \defeq h(x_0) - \inf_{x \in \R^n} h(x) < \infty$. Then:
\begin{enumerate}
    \item $\norm{\fopt - x_0} \le \sqrt{2 \gap / \alpha}$.
    \item For any $\nu > 0$, let $u \in \R^n$ denote a $\nu$-critical point of $f$, i.e., $\norm{\grad f(u)} \le \nu$. Then $\norm{\grad h(u)} \le 2 \nu + \sqrt{2 \gap \alpha}$.
\end{enumerate}
\end{lemma}

\begin{proof}
The first part follows because $f(x_0) = h(x_0)$, and thus every point $w \in \R^n$ which is distance strictly greater than $\sqrt{2 \gap / \alpha}$ from $x_0$ cannot be optimal for $f$, as such points $w$ satisfy
\begin{align*}
    f(w) = h(w) + \frac{\alpha}{2} \norm{w - x_0}^2 > h(w) + \gap = h(w) + h(x_0) - \inf_{x \in \R^n} h(x) \ge h(x_0) = f(x_0).
\end{align*}
For the second part, note that 
\begin{align}
    \label{eq:bound-on-grad-h}
    \grad f(u) = \grad h(u) + \alpha (u - x_0) \implies \norm{\grad h(u)} \le \norm{\grad f(u)} + \alpha \norm{u - x_0}.
\end{align}
We now focus on bounding $\norm{u - x_0}$.
Recall that if $g : \R^n \to \R$ is a $\mu$-strongly convex function with respect to $\norm{\cdot}$ and $\gopt$ is its global minimum, we have the following for all $x \in \R^n$:
\begin{align*}
    \frac{1}{2 \mu} \norm{\grad g(x)}^2 \ge g(x) - g(\gopt) \ge \frac{\mu}{2} \norm{x - \gopt}^2 \implies \norm{x - \gopt} \le \frac{1}{\mu} \norm{\grad g(x)}.
\end{align*}
Since $f$ is $\alpha$-strongly convex, instantiating the latter with $g \gets f$ and $x \gets u$ yields
\begin{align*}
    \norm{u - \fopt} \le \nu / \alpha.
\end{align*}
Then
\begin{align*}
    \norm{u - x_0} \le \norm{u - \fopt} + \norm{\fopt - x_0} \le \nu / \alpha + \sqrt{2 \gap / \alpha}
\end{align*}
by the first part. We conclude by plugging this back into \eqref{eq:bound-on-grad-h} and using the fact that $\norm{\grad f(u)} \le \nu$ by assumption.
\end{proof}

Next, the following lemma gives some relevant properties of the Fenchel game with added dual regularization; see also Section \ref{subsec:our-results}. Note that the smoothness and strong convexity of $g$ implies that $\gconj$ is strongly convex and differentiable \cite[Thm. 6.11]{orabona2023modern}.

\begin{lemma}[Fenchel game with added dual regularization]
    \label{lem:Fenchel-game-with-reg}
    Let $g : \R^n \to \R$ denote a smooth (in the sense that its gradient is Lipschitz) and strongly convex function.
    For $q \in \R^n$ and $\lambda > 0$, consider $\psi : \R^n \times \R^n \to \R$ defined as follows:
    \begin{align*}
        \psi(x, y) \defeq \inangle{x, y} - \gconj(y) - \lambda \breg{\gconj}{q}{y}.
    \end{align*} 
    Then for any $x \in \R^n$, we have
    \begin{align*}
        \max_{y \in \R^n} \psi(x, y) &=    (1 + \lambda) \cdot g \inparen{\frac{x + \lambda \grad \gconj(q)}{1 + \lambda}}  + C     , \text{ and}  \\
        \argmax_{y \in \R^n} \psi(x, y) &=  \grad g \inparen{\frac{x + \lambda \grad \gconj(q)}{1 + \lambda}}            ,
    \end{align*}
    where $C$ is a quantity with no dependence on $x$.
 \end{lemma}

 \begin{proof}
Note that
\begin{align*}
    \psi(x, y) &= \inangle{x, y} - \gconj(y) - \lambda \breg{\gconj}{q}{y} \\
    &= \inangle{x, y} - \gconj(y) - \lambda \inparen{\gconj(y) - \gconj(q) - \inangle{\grad \gconj(q), y - q}} \\
    &= \inangle{x, y} - (1 + \lambda) \gconj(y) + \lambda \gconj(q) + \lambda \inangle{\grad \gconj(q), y - q} \\
    &= - (1 + \lambda) \gconj(y) + \inangle{x + \lambda \grad \gconj(q), y} + C \\
    &= (1 + \lambda) \insquare{\inangle{\frac{x + \lambda \grad \gconj(q)}{1 + \lambda}, y} - \gconj(y)} + C,
\end{align*}
where $C$ is a quantity with no dependence on $x$ or $y$. Both results follow immediately from the fact that $g^{**} = g$ and Lemma~\ref{lem:subdiff-properties-of-conjugate}.
 \end{proof} 
\section{Convex analysis facts and additional technical lemmas}
\label{app:convex-analysis-facts}

In this appendix, we collect some convex analysis definitions and facts used in this paper for ease of reference. All convex functions we consider in this paper are proper and closed:

\begin{definition}[Proper function {\cite[Def. 2.14]{orabona2023modern}}]
A function $f : \R^n \to (- \infty, \infty]$ is proper if it is finite somewhere. (Note that $f$ must never have value $-\infty$.)
\end{definition}

\begin{definition}[Closed function {\cite[Def. 5.3]{orabona2023modern}}]
A function $f : \R^n \to [- \infty, \infty]$ is closed if the sublevel sets $\inbraces{x : f(x) \le \alpha}$ are closed for all $\alpha \in \R$. (See also \cite[Thm. 7.1]{rockafellar1970textbook} for some equivalent conditions.)
\end{definition}

As a result, it will always be the case that $f = f^{**}$ for any convex function $f$ we consider in this paper:

\begin{lemma}[Fenchel-Moreau theorem {\cite[Thm. 5.6]{orabona2023modern}}]
    \label{lem:Fenchel-Moreau}
If $f : \R^n \to (- \infty, \infty]$ is a proper, closed, and convex, then $f^{**} = f$.
\end{lemma}

The next lemma summarizes some useful properties related to $\partial f$ and $\partial \fconj$:

\begin{lemma}[Subdifferential properties of the convex conjugate {\cite[Thm. 5.7]{orabona2023modern}}]
    \label{lem:subdiff-properties-of-conjugate}
Let $f : \R^n \to (- \infty, \infty]$ be proper. Then the following conditions are equivalent for $x, \theta \in \R^n$:
\begin{enumerate}[label=(\alph*)]
    \item $\theta \in \partial f(x)$.

    \item $\inangle{\theta, y} - f(y)$ achieves its supremum in $y$ at $y = x$.

    \item $f(x) + \fconj(\theta) = \inangle{\theta, x}$.
\end{enumerate}
Moreover, if $f$ is also convex and closed, we have an additional equivalent condition:
\begin{enumerate}
    \item[(d)] $x \in \partial \fconj(\theta)$. 
\end{enumerate}
\end{lemma}

We have the following useful optimality conditions:

\begin{lemma}[Optimality condition for subdifferential {\cite[Thm. 6.12]{orabona2023modern}}]
    \label{lem:subdiff-optimality-cond}
Let $f : \R^n \to (- \infty, \infty]$ be proper. Then $\xopt \in \argmin_{x \in \R^n} f(x)$ if and only if $0 \in \partial f(\xopt)$.
\end{lemma}

\begin{lemma}[First-order optimality condition {\cite[Thm. 2.8]{orabona2023modern}}]
    \label{lem:first-order-optimality}
Let $V$ be a convex and nonempty set, and let $f$ be a convex function, differentiable over an open set which contains $V$. Then $\xopt \in \argmin_{x \in V} f(x)$ if and only if $\inangle{\grad f(\xopt), y - \xopt} \ge 0$ for all $y \in V$.
\end{lemma}

Under mild conditions, the subdifferential of a sum of functions is the Minkowski sum of the individual subdifferentials:

\begin{lemma}[Subdifferential of sum of functions {\cite[Thm. 2.18]{orabona2023modern}} or {\cite[Def. 9.12, Cor. 16.39]{bauschke2011monotone}}]
    \label{lem:subdiff-sum-of-funcs}
    Let $f_1, \dots, f_m : \R^n \to (- \infty, \infty]$ be proper functions, and let $f \defeq f_1 + \dots + f_m$. Then 
    $
        \partial f_1(x) + \dots + \partial f_m(x) \subseteq \partial f(x)
    $
    for all $x \in \R^n$. If additionally the functions $f_1, \dots, f_m$ are convex, closed, and $\dom f_m \cap \bigcap_{i = 1}^{m - 1} \interior \dom f_i \ne \emptyset$, then $\partial f_1(x) + \dots + \partial f_m(x) = \partial f(x)$ for all $x \in \R^n$.
\end{lemma}

The following lemma connects the subdifferential to the gradient:

\begin{lemma}[Connecting the subdifferential and the gradient {\cite[Thm. 25.1]{rockafellar1970textbook}}]
    \label{lem:connect-subdiff-grad}
Let $f : \R^n \to [-\infty, \infty]$ be convex, and let $x$ be a point where $f$ is finite. If $f$ is differentiable at $x$, then $\grad f(x)$ is the unique subgradient of $f$ at $x$. Conversely, if $f$ has a unique subgradient at $x$, then $f$ is differentiable at $x$.
\end{lemma}

We now collect some useful facts regarding the convex conjugate of negative entropy. In the following lemma, recall from Section \ref{sec:notation-and-assumptions} that $r_{\simplex^n} : \R^n \to (-\infty, \infty]$ denotes $r + \indc_{\simplex^n}$, and $\rconj_{\simplex^n} : \R^n \to \R$ denotes the convex conjugate of $r_{\simplex^n}$ (and not $\rconj$ restricted to $\simplex^n$). Recall also that given a sequence of vectors $y_1, y_2, \dots$, we let $[y_j]_i$ denote the $i$-th entry of $y_j$.

\begin{lemma}[Convex conjugate of negative entropy facts]
    \label{lem:convex-conjugate-neg-entropy}
    Define $r : \R_{\ge 0}^n \to \R$ via $r(y) \defeq \sum_{i = 1}^n y_i \ln y_i$ (with $0 \ln 0 \defeq 0$). 
    Then:
\begin{align*}
    \rconj_{\simplex^n}(\theta) &= \ln \inparen{\sum_{i = 1}^n \exp(\theta_i)}, \\
    \grad \rconj_{\simplex^n}(\theta) &= \frac{1}{\sum_{i = 1}^n \exp(\theta_i)} \inparen{\exp(\theta_1), \dots, \exp (\theta_n)}.
\end{align*}
In particular, given some points $y_0, \dots, y_{k - 1} \in \simplex^n_{>0}$ and constants $\lambda_0, \dots, \lambda_{k - 1} > 0$ for $k \in \N$, and letting $\Lambda_k \defeq \sum_{j = 0}^{k - 1} \lambda_j$,
we have for $i \in [n]$:
\begin{align*}
    \insquare{ \grad \rconj_{\simplex^n} \inparen{\frac{1}{\Lambda_k} \sum_{j = 0}^{k - 1} \lambda_j \grad r(y_j)} }_i
    = \frac{\prod_{j = 0}^{k - 1} [y_j]_i^{\lambda_j / \Lambda_k}}{\sum_{\ell = 1}^n \prod_{j = 0}^{k - 1} [y_j]_\ell^{\lambda_j / \Lambda_k}} .
\end{align*}
\end{lemma}

\begin{proof}
    The first part is the standard result that the convex conjugate of negative entropy is the log-sum-exp function (see \cite[Sec. 6.6]{orabona2023modern} or \cite[Example 3.25]{boyd2004convexopttextbook}). For the last claim, we have for $i \in [n]$:
    \begin{align*}
        \insquare{\frac{1}{\Lambda_k} \sum_{j = 0}^{k - 1} \lambda_j \grad r(y_j)}_i &= \frac{1}{\Lambda_k} \sum_{j = 0}^{k - 1} \lambda_j (\ln [y_j]_i + 1), \\
        \implies \exp \inparen{ \insquare{\frac{1}{\Lambda_k} \sum_{j = 0}^{k - 1} \lambda_j \grad r(y_j)}_i} &= e \cdot \prod_{j = 0}^{k - 1} [y_j]_i^{\lambda_j / \Lambda_k} .
    \end{align*}
    The result follows from the expression for $\grad \rconj_{\simplex^n}$.
\end{proof}

Next, we recall a basic fact about KL divergence:

\begin{lemma}[Bound on KL divergence from uniform {\cite[Sec. 2.1.1]{duchi2023infotheorylecturenotes}}]
    \label{lem:bound-on-KL-from-uniform}
Letting $\breg{}{u}{w} \defeq \sum_{i = 1}^n w_i \ln \frac{w_i}{u_i}$ for $u \in \simplex^n_{>0}$ and $w \in \simplex^n$, we have that $\breg{}{\frac{1}{n} \ones }{w} \le \ln n$ for all $w \in \simplex^n$.
\end{lemma}

Finally, we bound the Lipschitz constant of a particular function.

\begin{lemma}
    \label{lem:bounding-Lipschitz-dual-matrix-games}
Define $\phi : \R^n \to \R$ via $\phi(y) \defeq \min_{x \in B^d} x^\top A y$, where each column of $A \in \R^{d \times n}$ has Euclidean norm at most 1. Then $\phi$ is 1-Lipschitz with respect to $\norm{\cdot}_1$.
\end{lemma}

\begin{proof}
Note that we can equivalently express $\phi(y) = - \norm{A y}_2$. Defining $g : \R^d \to \R$ via $g(x) = \norm{x}_2$, it is a standard result (e.g., \cite[Sec. D.4.2]{urruty2004convexanalysistextbook}) that $\partial \phi(y) = - A^\top \partial g(A y)$ for all $y \in \R^n$. Furthermore, it is straightforward to check that every subgradient of $g$ has Euclidean norm at most 1. Thus, every subgradient of $\phi$ takes the form $- A^\top u$ for some $u \in \R^d$ with $\norm{u}_2 \le 1$. Then it is immediate that the $\ell_\infty$-norm of every subgradient of $\phi$ is bounded by 1.
\end{proof}

\end{document}